\documentclass[sn-mathphys-num]{sn-jnl}

\usepackage{graphicx}%
\usepackage{multirow}%
\usepackage{amsmath,amssymb,amsfonts,dsfont}%
\usepackage{amsthm}%
\usepackage{mathrsfs}%
\usepackage[title]{appendix}%
\usepackage{xcolor}%
\usepackage{textcomp}%
\usepackage{manyfoot}%
\usepackage{booktabs}%
\usepackage{algorithm}%
\usepackage{algorithmicx}%
\usepackage{algpseudocode}%
\usepackage{listings}%

\theoremstyle{thmstyleone}%
\newtheorem{theorem}{Theorem}
\newtheorem{proposition}[theorem]{Proposition}%

\theoremstyle{thmstyletwo}%
\newtheorem{example}{Example}%
\newtheorem{remark}{Remark}%

\theoremstyle{thmstylethree}%
\newtheorem{definition}{Definition}%

\begin{document}

\title[Two types of compressible isotropic neo-Hookean material models]{Two types of compressible isotropic neo-Hookean material models}


\author[1]{\fnm{Sergey N.} \sur{Korobeynikov}}\email{s.n.korobeynikov@mail.ru}

\author[1]{\fnm{Alexey Yu.} \sur{Larichkin}}\email{larichking@gmail.com}

\author[2]{\fnm{Patrizio} \sur{Neff}}\email{patrizio.neff@uni-due.de}

\affil[1]{\orgname{Lavrentyev Institute of Hydrodynamics SB RAS}, \orgaddress{\street{Lavrentyev av., 15}, \city{Novosibirsk}, \postcode{630090}, \country{Russia}}}

\affil[2]{\orgdiv{Chair for Nonlinear Analysis and Modelling, Faculty of Mathematics}, \orgname{University of Duisburg-Essen}, \orgaddress{\street{Thea-Leymann Strasse, 9}, \city{Essen}, \postcode{D-45127}, \country{Germany}}}



\abstract{In this contribution, we present a systematic study of the performance of two known types of compressible generalization of the incompressible neo-Hookean material model. The first type of generalization is based on the development of \emph{vol-iso} neo-Hookean models and involves the additive decomposition of the elastic energy into volumetric and isochoric parts. The second simpler type of generalization is based on the development of \emph{mixed} neo-Hookean models that do not use this decomposition. Theoretical studies of model performance and simulations of some homogeneous deformations have shown that when using volumetric functions $(J^q+J^{-q}-2)/(2q^2)$ ($J$ is the volume ratio, and $q\in \mathds{R}$ is a parameter, $q\geq 0$) from the Hartmann--Neff family [Hartmann and Neff, Int. J. Solids Structures, 40: 2767--2791 (2003)] with parameter $q\geq 2$ (the preferred value is $q=5$), mixed and vol-iso models show similar performance in applications and have physically reasonable responses in extreme states, which is convenient for theoretical studies. However, contrary to vol-iso models, mixed models allow the use of a wider set of volumetric functions with physically reasonable responses in extreme states. A further feature of mixed models is simpler expressions for stresses and tangent stiffness tensors.}




\maketitle

\section{Introduction}
\label{sec:1}

In large strain solid mechanics, by the constitutive relations of \emph{hyperelasticity} (or \emph{Green elasticity}) are meant the functional relations between stress and strain tensors using a scalar tensor function called the \emph{elastic energy}, see, e.g., \cite{Hackett2018}. Currently, two approaches to generating elastic energies for isotropic hyperelastic material models exist and are being developed (see, e.g., \cite{Bertram2021,Crisfield1997,deBorst2012,deSouzaNeto2008,Hackett2018,Holzapfel2000,Ogden1984,Wriggers2008,Zhang2025}). The first and historically earlier approach is based on the use of invariants of the left (or right) Cauchy--Green deformation tensor in expressions for elastic energies (note here the outstanding contributions of Mooney and Rivlin, see, e.g., \cite{Hackett2018}). The second approach is based on the use of stretches as independent variables for elastic energies (the first model to use this approach is apparently the Hencky one \cite{HenckyRCT1933,HenckyJAM1933}).

Because of the ease of mathematical derivation of elastic energy expressions and due to experimental studies on the deformation of nearly incompressible rubber-like materials, early hyperelastic models were based on the simplifying assumption of incompressibility of the materials under study. However, the subsequent development of the theoretical foundations of hyperelasticity, driven primarily by numerous applications to large deformations of compressible materials (e.g., elastic foams, graphene, biological tissues, etc.), finite element (FE) implementations of hyperelasticity models,\footnote{In FE-implementations, constitutive relations for incompressible materials are typically replaced by constitutive relations for compressible ones, this replacement is a consequence of using the penalty function method to satisfy the incompressibility condition (see, e.g., \cite{Crisfield1997,deBorst2012,deSouzaNeto2008,PengCS1997}).} and the realization that purely incompressible materials do not exist in nature, led to the study of the properties of compressible (or slightly compressible) hyperelastic material models. Naturally, researchers began to formulate elastic energy expressions for compressible materials based on the well-known elastic energy expressions for incompressible materials.

Within the framework of linear elastic theory, the elastic energy for compressible isotropic materials can be represented in both decoupled and coupled theoretically equivalent  forms. In the first case, the elastic energy has an additive representation as a sum of purely volumetric and purely isochoric energies, and in the second case, the same elastic energy can be represented as a sum of purely volumetric and mixed---volumetric and isochoric---energies. Therefore, when generalizing incompressible hyperelastic material models to account for material compressibility, one faces a dilemma: how to represent the elastic energy: in the decoupled or coupled form? This dilemma arises, e.g., when the Ogden model \cite{OgdenPRSLA1972a}, originally developed for incompressible materials, is generalized to account for material compressibility.\footnote{The Ogden model uses principal stretches as independent variables.} In the original generalization of his model, Ogden used the coupled form of elastic energy \cite{OgdenPRSLA1972b} (see also \cite{PengCS1997,Wriggers2008}), and in the subsequent generalization, Ogden used the uncoupled form of elastic energy \cite{Ogden1984} (see also \cite{deBorst2012,deSouzaNeto2008,Holzapfel2000}). Note that, unlike the linear elastic model, the generalization of the elastic energy for the hyperelastic incompressible material model to account for material compressibility in the decoupled and coupled forms leads to two different material models. Note that the generalization of the elastic energy in the coupled form seems to be simpler and more elegant than the generalization in the decoupled form, since in this case, the term containing the Lagrange multiplier in the potential energy expression for incompressible material is simply replaced by the volumetric energy, whereas in the second case, one uses the complex technique of multiplicative decomposition of principal stretches in volumetric-isochoric form proposed by Richter in the late 1940s--early 1950s (cf., \cite{GrabanMMS2019}).\footnote{Note that this decomposition technique is often attributed to Flory \cite{FloryTFS1961}, whose work was published more than 10 years after the publications by Richter \cite{RichterZAMM1948,RichterZAMM1949,RichterAM1949,RichterMN1952} in Germany (cf., \cite{GrabanMMS2019}, see also \cite{NeffJMPS2025}).} Compressible hyperelastic material models with elastic energies in the coupled form are studied in \cite{AttardIJSS2003,AttardIJSS2004,Bonet2008,Chaves2013,ClaytonMM2014,EhlersAM1998,Hashiguchi2013,KellermannZAMM2016,KorobeynikovAAM2023,OgdenPRSLA1972b,PenceMMS2015,Simo1998,SimoCMAME1984,SpringhettiJElast2022,StickleCM2022,Wriggers2008},
and models of the same material with elastic energies in the decoupled form are addressed in \cite{Bonet2008,Crisfield1997,deSouzaNeto2008,DestradeIJNME2012,EhlersAM1998,Hackett2018,Hashiguchi2013,Haupt2002,Holzapfel2000,KellermannZAMM2016,KorobeynikovAAM2023,KossaMeccanica2023,PenceMMS2015,PennTSR1970,Rubin2021,Simo1998,SpringhettiJElast2022,Wriggers2008,Zhang2025}.
Note that both formulations are rank-one convex and polyconvex provided that $\mu,\,K,\,\lambda\,>0$ and the volumetric function $h(J)$ is convex in $J=\det\,\mathbf{F}$ (cf., \cite{HartmannIJSS2003}), where $\mu$ is the shear modulus, $K$ is the bulk modulus, $\lambda$ is the second Lam\'{e} parameter, and $\mathbf{F}$ is the deformation gradient.

Note that there are also other approaches to the compressible generalization of elastic energies for hyperelastic incompressible materials, although based on the Richter--Flory decomposition, but not involving the additive decomposition of elastic energies into volumetric and isochoric strain energies (see, e.g., \cite{AttardIJSS2003,AttardIJSS2004,FongTSR1975,HuangJAM2014,HuangJAM2016,Ogden1984,RogovoyEJMAS2001,YaoPTRSA2022}). Such elastic energies are used to capture the features of dilatation behavior of slightly compressible rubber-like materials during deformation, in particular, to obtain agreement with the experimental findings by Penn \cite{PennTSR1970}, who established that the use of decoupled forms of elastic energies is inconsistent with his experimental data on the dependence of dilatation on elongation under uniaxial loading. However, the analysis of this type of elastic energies is beyond the scope of this book.

Since the decoupled form of elastic energy expressions contradicts some experimental data, the question arises: do hyperelasticity models based on the decoupled form of elastic energy expressions indeed have such a decisive advantage over similar models based on the direct generalization of elastic energy expressions without using the Richter--Flory decomposition, which led to the fact that models of the first type are implemented in many commercial FE systems to the detriment of models of the second type for simulating deformations of slightly compressible rubber-like materials? The answer to this question can be obtained by comparing the performance characteristics of the above two types of hyperelastic material models, especially under conditions of slight compressibility. It is reasonable to restrict this comparison to compressible generalizations of elastic energy expressions for the simplest incompressible isotropic hyperelastic material model --- the neo-Hookean material model.  First, the elastic energy for this model is written as a dependence on the first invariant of the right (or left) Cauchy--Green deformation tensor (see, e.g., \cite{Hackett2018}), and, second, this elastic energy can be represented as a one-power version of the Ogden material with power $n=2$ (cf., \cite{OgdenPRSLA1972a}). We consider therefore two families of elastic energies for compressible neo-Hookean material models, the mixed models and the vol-iso ones. Both are then composed with the same volumetric functions $h(J)$ (see Fig. \ref{f-I1}).
\begin{figure}
\begin{center}
\includegraphics{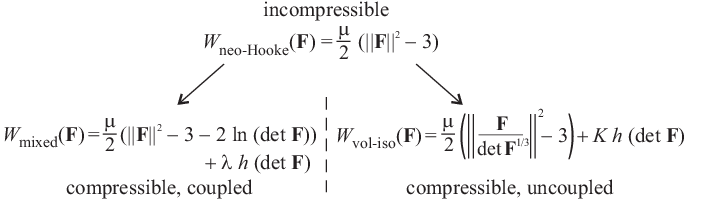}
\end{center}
\caption{Sketch of generation of two different types of the elastic energies $W_{\text{mixed}}$ and $ W_{\text{vol-iso}}$ ($\mathbf{F}$ is the deformation gradient) for compressible neo-Hookean material models based on the elastic energy $W_{\text{neo-Hooke}}$ for the incompressible neo-Hookean material model (see Section \ref{sec:3} for details). Here, $\mu$ is the standard shear modulus, $\lambda$ is the second Lam\'{e} parameter and $K$ is the bulk modulus.}
\label{f-I1}
\end{figure}
For each considered $h(J)$ we give a number \#1,\ldots,\#8. In this way, speaking of a mixed or vol-iso model with \#$k$ ($k=1,\ldots,8$), the energy is uniquely defined (see Tables \ref{t-Intro-1} and \ref{t-Intro-2}), where $J=\det \mathbf{F}$.
\begin{table}
\caption{Some elastic energies for the mixed compressible neo-Hookean materials}
\label{t-Intro-1}
\begin{tabular}{ll}
\hline\noalign{\smallskip}
 Model ID & Expression for the elastic energy                                            \\
\noalign{\smallskip}\hline\noalign{\smallskip}
   1      & $\mu\,(\|\mathbf{F}\|^2 - 3 - 2\ln J)/2 + \lambda\,(\ln J)^2/2$              \\
   2      & $\mu\,(\|\mathbf{F}\|^2 - 3 - 2\ln J)/2 + \lambda\,(J+J^{-1}-2)/2$           \\
   3      & $\mu\,(\|\mathbf{F}\|^2 - 3 - 2\ln J)/2 + \lambda\,(J^2+J^{-2}-2)/8$         \\
   4      & $\mu\,(\|\mathbf{F}\|^2 - 3 - 2\ln J)/2 + \lambda\,(J^5+J^{-5}-2)/50$        \\
   5      & $\mu\,(\|\mathbf{F}\|^2 - 3 - 2\ln J)/2 + \lambda\,(J^2-2\ln J-1)/4$         \\
   6      & $\mu\,(\|\mathbf{F}\|^2 - 3 - 2\ln J)/2 + \lambda\,(J-\ln J-1)$              \\
   7      & $\mu\,(\|\mathbf{F}\|^2 - 3 - 2\ln J)/2 + \lambda\,(J-1)^2/2$                \\
   8      & $\mu\,(\|\mathbf{F}\|^2 - 3 - 2\ln J)/2 + \lambda\,(\mathrm{e}^{\ln^{2}\!\!J} - 1)/2$ \\
\noalign{\smallskip}\hline
\end{tabular}
\end{table}
\begin{table}
\caption{Some elastic energies for the vol-iso compressible neo-Hookean materials}
\label{t-Intro-2}
\begin{tabular}{ll}
\hline\noalign{\smallskip}
 Model ID & Expression for the elastic energy                                     \\
\noalign{\smallskip}\hline\noalign{\smallskip}
   1      & $\mu\,(\|\mathbf{F}/J^{1/3}\|^2 - 3)/2 + K\,(\ln J)^2/2$              \\
   2      & $\mu\,(\|\mathbf{F}/J^{1/3}\|^2 - 3)/2 + K\,(J+J^{-1}-2)/2$           \\
   3      & $\mu\,(\|\mathbf{F}/J^{1/3}\|^2 - 3)/2 + K\,(J^2+J^{-2}-2)/8$         \\
   4      & $\mu\,(\|\mathbf{F}/J^{1/3}\|^2 - 3)/2 + K\,(J^5+J^{-5}-2)/50$        \\
   5      & $\mu\,(\|\mathbf{F}/J^{1/3}\|^2 - 3)/2 + K\,(J^2-2\ln J-1)/4$         \\
   6      & $\mu\,(\|\mathbf{F}/J^{1/3}\|^2 - 3)/2 + K\,(J-\ln J-1)$              \\
   7      & $\mu\,(\|\mathbf{F}/J^{1/3}\|^2 - 3)/2 + K\,(J-1)^2/2$                \\
   8      & $\mu\,(\|\mathbf{F}/J^{1/3}\|^2 - 3)/2 + K\,(\mathrm{e}^{\ln^{2}\!\!J} - 1)/2$ \\
\noalign{\smallskip}\hline
\end{tabular}
\end{table}

At a sufficiently high degree of material compressibility, for example, at Poisson's ratio $\nu=0.125$ (see Section \ref{sec:3-4}), the two types of material models considered give different plots of elastic energies versus longitudinal stretches (see Fig. \ref{f-I2}) when using the same volumetric functions in the uniaxial loading problem and the vol-iso model \#7 even leads to an unphysical non-monotonic plot of the elastic energy versus longitudinal stretch. This suggests that for the same two values of material parameters (shear modulus $\mu$ and Poisson's ratio $\nu$) and the same volumetric function $h(J)$, these two types of material models can lead to different quantitative values of stresses in deformable bodies (see Section \ref{sec:6} for examples). This fact motivates a thorough and comprehensive study of the performance of the two types of compressible neo-Hookean material models under consideration.
\begin{figure}
\begin{center}
\includegraphics{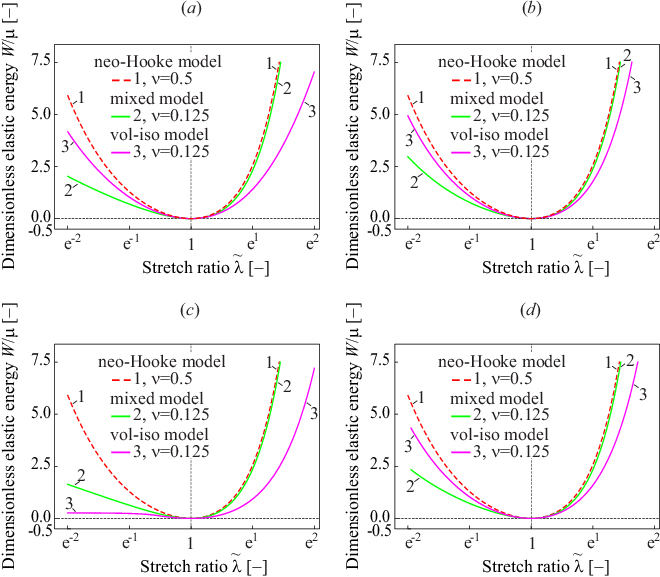}
\end{center}
\caption{Plots of elastic energies versus longitudinal stretches $\tilde{\lambda}$ in the uniaxial loading problem for compressible neo-Hookean models (a dimensionless elastic energy $W/\mu$ depends on Poisson's ratio $\nu$ only due to the expressions $\lambda=2\mu\,\nu/(1-2\nu)\ (\Leftrightarrow\ \nu = \lambda/2(\lambda+\mu))$  and $K=\lambda + 2\mu/3$) with the following volumetric functions ($J=\det \mathbf{F}$) $h(J)$:  $(\ln^2J)/2$ (\emph{a}), $(J^{5}+J^{-5})/50$ (\emph{b}), $(J-1)^2/2$ (\emph{c}), and $(\mathrm{e}^{\ln^{2}\!\!J} - 1)/2$ (\emph{d}) (corresponding to volumetric functions \#1, \#4, \#7, and \#8, see Table \ref{t1} for the volumetric functions definition).}
\label{f-I2}
\end{figure}

One objective of this study is a comparative analysis of the performance of two types (\emph{vol-iso/mixed}) of compressible isotropic hyperelastic neo-Hookean material models obtained by generalizing the standard neo-Hookean isotropic incompressible material model to account for volumetric energy with/without the additive decomposition of a elastic energy into volumetric and isochoric strain energies. In this comparison, we will check whether both types of compressible material models satisfy two fundamental postulates of solid mechanics formulated to test nonlinear material models for satisfaction of some desired properties not following from thermodynamic constraints. The first postulate (the \emph{Hill postulate}), requiring the satisfaction of some inequality, was proposed more than half a century ago by Hill (cf., \cite{HillJMPS1968,HillPRSLA1970,Hill1979}), and the second inequality in the form of the corotational stability postulate (CSP) was proposed recently by Neff et al. (cf., \cite{dAgostinoJElast2025,NeffIJNLM2025,NeffJMPS2025}). Both these postulates extend the \emph{Drucker postulate} (the material stability postulate), well-known in solid mechanics and usually formulated for infinitesimal strains, to the case of finite strains. The satisfaction of the first (Hill) postulate is equivalent to that the Kirchhoff stress tensor is a monotonic function of the logarithmic strain tensor, and the satisfaction of the second postulate (CSP) is equivalent to that the Cauchy stress tensor is a monotonic function of the logarithmic strain tensor.

The next objective of this study is to compare the performance of the above two types of compressible material models by solving some problems with homogeneous deformations. Note that the performance characteristics of these models in solving similar problems have already been studied by a number of researchers (see, e.g., \cite{EhlersAM1998,KellermannZAMM2016,KorobeynikovAAM2023,KossaMeccanica2023,PenceMMS2015,SpringhettiJElast2022}), but the comparative analysis of the performance characteristics in these studies is apparently not thorough enough and not always fully correct. For example, a comparison of stresses and strains was performed \cite{EhlersAM1998,PenceMMS2015} using different types of volumetric energies in different models for the same homogeneous deformations, but for a correct comparison, it is necessary to use the same types of this energy for different models. In addition, previous studies have found some physically unreasonable responses for both types of material models under consideration. Therefore, our task here is to select the types of volumetric energy that would minimize these physically unreasonable responses.

In summary, the main objectives of this study are to answer the following questions:
\begin{enumerate}
  \item what types of volumetric energies minimize the effects of physically unreasonable response for the neo-Hookean models under consideration?
  \item is the mixed neo-Hookean compressible hyperelastic material model taking into account volumetric energy without the additive decomposition of the elastic energy into volumetric and isochoric strain energies inferior in performance to the vol-iso counterpart with this decomposition?
\end{enumerate}

The novelty of this study lies in conclusively answering these questions.

\section{Preliminaries}
\label{sec:2}

In this chapter we present the basic expressions for tensor algebra (Section \ref{sec:2-1}), local body deformations and basic kinematics (Section \ref{sec:2-2}), and stress tensors and their rates (Section \ref{sec:2-3}) required for the exposition of the basic material of the research. In addition, Section \ref{sec:2-4} presents different forms of the elastic energy and constitutive relations for the linear isotropic material model that inspired the derivation of similar forms for the nonlinear neo-Hookean material (see Section \ref{sec:3}).

\subsection{Basic equations of tensor algebra}
\label{sec:2-1}

We define the second-order tensor $\mathbf{H}\in \mathcal{T}^2$ and the fourth-order tensor $\mathbb{H}\in \mathcal{T}^4$ (hereinafter, $\mathcal{T}^2$ and $\mathcal{T}^4$ denote the sets of all second-order and fourth-order tensors). Hereinafter, $\mathcal{T}^2_{\text{sym}},\  \mathcal{T}^2_{\text{skew}}\subset \mathcal{T}^2$ denote the sets of all symmetric and skew-symmetric second-order tensors, respectively; $\text{sym}\,\mathbf{A}\equiv (\mathbf{A}+\mathbf{A}^T)/2\in \mathcal{T}^2_{\text{sym}}$ and $\text{skew}\,\mathbf{A}\equiv (\mathbf{A}-\mathbf{A}^T)/2\in \mathcal{T}^2_{\text{skew}}$ denote the symmetric and skew-symmetric components of the tensor $\mathbf{A}\in \mathcal{T}^2$, respectively. Next the set of fourth-order tensors with major symmetry will be denoted by $\mathcal{T}^4_{\text{Sym}}\subset \mathcal{T}^4$, the set of the same tensors with twice minor symmetry will be denoted by $\mathcal{T}^4_{\text{sym}}\subset \mathcal{T}^4$, and the set of tensors with both major and twice minor symmetries\footnote{Hereinafter, tensors with both major and twice minor symmetries will be called \emph{supersymmetric} (cf., \cite{Itskov2019}) or \emph{fully symmetric} (cf., \cite{FedericoIJNLM2012,FedericoMMS2025}) tensors.} by $\mathcal{T}^4_{\text{Ssym}}\subset \mathcal{T}^4$.

Let $\mathbf{A},\mathbf{H}\in \mathcal{T}^2$. We define the \emph{double inner product} (\emph{double contraction}) operation of tensors:
\begin{equation*}
    \mathbf{A}:\mathbf{H}=\mathbf{H}:\mathbf{A} \equiv \text{tr}(\mathbf{A}\cdot\mathbf{H}^T)=  \text{tr}(\mathbf{A}^T\cdot\mathbf{H})=
    \text{tr}(\mathbf{H}\cdot\mathbf{A}^T)=\text{tr}(\mathbf{H}^T\cdot\mathbf{A})\ (=\,\langle\mathbf{A},\mathbf{H}\rangle).
\end{equation*}
Hereinafter, the superscript $T$ denotes the transposition of a tensor, and the dot between vectors and/or tensors denotes their inner (matrix) product. Using this definition, we can introduce the norm of any second order tensor
\begin{equation*}
  \|\mathbf{A}\| \equiv (\mathbf{A}:\mathbf{A})^{1/2}\ (=\sqrt{\langle\mathbf{A},\mathbf{A}\rangle}).
\end{equation*}

We define the following four external product operations for second-order tensors according to the definitions given in \cite{Curnier1994,Holzapfel2000,PeyrautANM2009}: \emph{dyadic} $\mathbf{A}\otimes\mathbf{H}$, \emph{direct} $\mathbf{A}\underline{\otimes}\,\mathbf{H}$, \emph{alternate} $\mathbf{A}\overline{\otimes}\,\mathbf{H}$, and \emph{symmetric} $\mathbf{A}\!\overset{\text{sym}}{\otimes}\!\mathbf{H}\equiv \frac{1}{2}(\mathbf{A}\underline{\otimes}\,\mathbf{H}+ \mathbf{A}\overline{\otimes}\,\mathbf{H})$ tensor products.

Let $\mathbf{A},\mathbf{B},\mathbf{X}\in \mathcal{T}^2$. We can show that the following identities hold:
\begin{align*}
    (\mathbf{A}\otimes\mathbf{B}):\mathbf{X} &= \mathbf{A}(\mathbf{B}:\mathbf{X}),  \\
    (\mathbf{A}\!\overset{\text{sym}}{\otimes}\!\mathbf{B}):\mathbf{X} & = \mathbf{A}\cdot \text{sym}\,\mathbf{X}\cdot\mathbf{B}^T. \notag
\end{align*}

Let $\mathbf{S}\in\mathcal{T}^2_{\text{sym}}$. This tensor can be represented in the classical spectral form
\begin{equation}\label{3}
\mathbf{S}=\sum_{k=1}^{3} s_k\,\mathbf{n}_k\otimes\mathbf{n}_k,
\end{equation}
where $s_k\in \mathds{R}$ are eigenvalues and $\{\mathbf{n}_k\}$ ($k=1,2,3$) is the triad of the corresponding subordinate right-oriented orthonormal eigenvectors (corresponding to the principal directions) of the tensor $\mathbf{S}$, and the symbol $\otimes$ denotes the dyadic product of vectors. For multiple eigenvalues $s_k$, the corresponding eigenvectors $\mathbf{n}_k$ are defined ambiguously. This ambiguity can be circumvented by representing the tensor $\mathbf{S}$ in terms of \emph{eigenprojections} (see, e.g., \cite{Bertram2021,KorobeynikovAM2011,KorobeynikovAM2023,LuehrCMAME1990}):
\begin{equation}\label{4}
    \mathbf{S}=\sum^m_{i=1}s_i\, \mathbf{S}_i.
\end{equation}
Here, $s_i$ are all different $m$\footnote{The number $m$ ($1\leq m\leq 3$) will be called the \emph{eigenindex}.} eigenvalues of the tensor $\mathbf{S}$ and $\mathbf{S}_i$ ($i=1,\ldots,m$) are the subordinate \emph{eigenprojections}. Without loss of generality, we number the eigenvalues $s_k$ and define the eigenprojections depending on the eigenindex $m$ as follows:
\begin{equation}\label{5}
   m= \begin{cases}
    3, & s_1\neq s_2 \neq s_3 \neq s_1,\quad  \mathbf{S}_i=\mathbf{n}_i\otimes \mathbf{n}_i\ (i=1,2,3); \\
    2, & s_1\neq s_2 = s_3, \quad \quad \quad\  \mathbf{S}_1=\mathbf{n}_1\otimes \mathbf{n}_1,\ \mathbf{S}_2=\mathbf{n}_2\otimes \mathbf{n}_2+\mathbf{n}_3\otimes \mathbf{n}_3 = \mathbf{I} - \mathbf{S}_1; \\
    1, & s_1 = s_2 = s_3, \quad \quad \quad\ \mathbf{S}_1=\mathbf{n}_1\otimes \mathbf{n}_1 + \mathbf{n}_2\otimes \mathbf{n}_2+\mathbf{n}_3\otimes \mathbf{n}_3=\mathbf{I}.
    \end{cases}
\end{equation}
Hereinafter, $\mathbf{I}$ denotes the identity tensor. The eigenprojections have the following properties \cite{LuehrCMAME1990}:
\begin{equation}\label{6}
   \mathbf{S}_i \cdot \mathbf{S}_j=
   \begin{cases}
    \mathbf{S}_i & \text{if}\ i=j \\
    \mathbf{0} & \text{if}\ i\neq j
   \end{cases},
\quad\quad \sum^m_{i=1}\mathbf{S}_i=\mathbf{I}, \quad\quad \text{tr}\,\mathbf{S}_i=m_i\quad (i,j=1,\ldots,m).
\end{equation}
Here, $m_i$ denotes the multiplicity of an eigenvalue $s_i$ and $\mathbf{0}\in \mathcal{T}^2$ denotes the zero second-order tensor.

Let the tensors $\mathbf{S},\mathbf{H} \in \mathcal{T}^2_{\text{sym}}$ be arbitrary tensors, and let the tensor $\mathbf{S}$ have the spectral representation \eqref{4}. We represent the  tensor $\mathbf{H}$ in the following form (see Corollary 2.1 of \cite{KorobeynikovAM2018}):\footnote{Hereinafter, the notation $\sum_{i\neq j=1}^{m}$ denotes the summation over $i,j=1,\ldots, m$ and $i\neq j$ and this summation is assumed to vanish when $m=1$.}
\begin{equation}\label{7}
\mathbf{H} = \hat{\mathbf{H}} + \tilde{\mathbf{H}},\quad\quad\quad \hat{\mathbf{H}}\equiv \sum_{i=1}^{m} \mathbf{S}_i\cdot \mathbf{H}\cdot\mathbf{S}_i,\quad\quad\quad
\tilde{\mathbf{H}}\equiv \sum_{i\neq j=1}^{m} \mathbf{S}_i\cdot \mathbf{H}\cdot\mathbf{S}_j,
\end{equation}
so that $\hat{\mathbf{H}},\, \tilde{\mathbf{H}}\in \mathcal{T}^2_{\text{sym}}$ are the components of the tensor $\mathbf{H}$ that are coaxial \index{Tensors!coaxial} and orthogonal\footnote{Tensors $\mathbf{X},\,\mathbf{Y}\in \mathcal{T}^2_{\text{sym}}$ will be called orthogonal if the equality $\mathbf{X}:\mathbf{Y}=0$ is satisfied.} to the tensor $\mathbf{S}$. Since the tensor $\hat{\mathbf{H}}$ is coaxial with the tensor $\mathbf{S}$, this tensor can be also represented as
\begin{equation}\label{8}
    \hat{\mathbf{H}}=\sum^m_{i=1}H_i\, \mathbf{S}_i.
\end{equation}

\begin{proposition}
\label{Pr-1}
Suppose that $\mathbf{H},\, \mathbf{S}\in \mathcal{T}^2_{\text{\emph{sym}}}$; $\mathbf{S}$ has the spectral representation \eqref{4}; and the tensor $\mathbf{X}$ is an isotropic tensor function of the tensor arguments $\mathbf{S}$ and $\mathbf{H}$ which is linear in $\mathbf{H}$ and can be written as ($2\leq m\leq 3$)
\begin{equation*}
    \mathbf{X}(\mathbf{S},\mathbf{H})= \sum_{i\neq j=1}^{m} x_{ij}\, \mathbf{S}_i\cdot \mathbf{H}\cdot\mathbf{S}_j,\quad x_{ij}\equiv x(s_i,s_j)\ (x_{ji}=x_{ij}), 
\end{equation*}
or, in an alternative form, as\footnote{The fact that $\mathbb{X}\,\in \mathcal{T}^4_\text{{Ssym}}$ follows from the statements of Theorem 2.1 in \cite{KorobeynikovAM2018}.}
\begin{equation}\label{10}
    \mathbf{X}(\mathbf{S},\mathbf{H}) =   \mathbb{X}(\mathbf{S}): \mathbf{H}\quad (\mathbb{X}\,\in \mathcal{T}^4_\text{{\emph{Ssym}}}),
\end{equation}
where
\begin{equation*}
\mathbb{X}(\mathbf{S}) \equiv \sum_{i\neq j=1}^{m} x_{ij}\,\mathbf{S}_i\!\overset{\text{\emph{sym}}}{\otimes}\!\mathbf{S}_j.
\end{equation*}

\noindent
Then the inequalities
\begin{equation*}
  x_{ij}>0\quad (i,j=1,\ldots m,\ i\neq j)
\end{equation*}
are necessary and sufficient for the positive definiteness of the tensorial function $\mathbf{X}(\mathbf{S},\mathbf{H})$ and its associated tensor $\mathbb{X}$ with respect to the tensor $\tilde{\mathbf{H}}$ (which is orthogonal to the tensor $\mathbf{S}$); i.e., $\mathbf{X}:\tilde{\mathbf{H}}>0\  \forall\ \tilde{\mathbf{H}}\neq \mathbf{0} \Leftrightarrow \ \tilde{\mathbf{H}}:\mathbb{X}:\tilde{\mathbf{H}}>0\ \forall\ \tilde{\mathbf{H}}\neq \mathbf{0}$.
\end{proposition}

\begin{proof}
In view of the statements of Theorem 2.2 in \cite{KorobeynikovAM2018}, the tensor $\mathbf{X}$ is also orthogonal to the tensor $\mathbf{S}$, whence it follows that
\begin{equation*}
  \mathbf{X}:\mathbf{H}=\mathbf{X}:\tilde{\mathbf{H}}.
\end{equation*}
Since the tensors considered are coaxial and orthogonal to the tensor $\mathbf{S}$ and hence to each other, we obtain the equality
\begin{equation*}
 \mathbf{H}: \mathbb{X}:\mathbf{H}=\tilde{\mathbf{H}}: \mathbb{X}:\tilde{\mathbf{H}}.
\end{equation*}
In view of the symmetry of the quantities $x_{ij}>0\ (i,j=1,\ldots m,\ i\neq j)$, the symmetry of the tensor $\tilde{\mathbf{H}}$ (which is a consequence of the symmetry of the tensor $\mathbf{H}$), and the representation of the eigenprojections in \eqref{5}, direct calculations yield
\begin{equation}\label{15}
 \mathbf{H}: \mathbb{X}:\mathbf{H}=\tilde{\mathbf{H}}: \mathbb{X}:\tilde{\mathbf{H}}=
\begin{cases}
    2(x_{12}H_{12}^2+x_{13}H_{13}^2+x_{23}H_{23}^2), & \text{if}\ m=3 \\
    2x_{12}(H_{12}^2 + H_{13}^2), & \text{if}\ m=2
   \end{cases},
\end{equation}
where
\begin{equation*}
  H_{kl}(=H_{lk})\equiv \mathbf{n}_k\cdot \mathbf{H}\cdot \mathbf{n}_l\ (k,l=1,2,3,\ k\neq l).
\end{equation*}
The statements of Proposition follow from \eqref{15}.
\end{proof}

\begin{example}
\label{exam:2-1}
Let us consider a simple example for 2D analysis illustrating the statements of Proposition \ref{Pr-1}. Let the tensor $\mathbf{S}\in\mathcal{T}^2_{\text{sym}}$ have the following spectral representation of the form \eqref{3}
\begin{equation*}
\mathbf{S}= s_1\,\mathbf{n}_1\otimes\mathbf{n}_1 + s_1\,\mathbf{n}_2\otimes\mathbf{n}_2.
\end{equation*}
Let any tensor $\mathbf{H}\in\mathcal{T}^2_{\text{sym}}$ have the following representation in the principal axes of the tensor $\mathbf{S}$:
\begin{equation*}
\mathbf{H}= H_{11}\,\mathbf{n}_1\otimes\mathbf{n}_1 + H_{22}\,\mathbf{n}_2\otimes\mathbf{n}_2 + H_{12}(\mathbf{n}_1\otimes\mathbf{n}_2 + \mathbf{n}_2\otimes\mathbf{n}_1),
\end{equation*}
or
\begin{equation*}
\mathbf{H} = \hat{\mathbf{H}} + \tilde{\mathbf{H}},\quad\quad \hat{\mathbf{H}}\equiv H_{11}\,\mathbf{n}_1\otimes\mathbf{n}_1 + H_{22}\,\mathbf{n}_2\otimes\mathbf{n}_2,\quad\quad
\tilde{\mathbf{H}} \equiv H_{12}(\mathbf{n}_1\otimes\mathbf{n}_2 + \mathbf{n}_2\otimes\mathbf{n}_1).
\end{equation*}
We define the following isotropic tensor function of its arguments $\mathbf{X}(\mathbf{S},\mathbf{H})\in\mathcal{T}^2_{\text{sym}}$ which is linear in $\mathbf{H}$
\begin{equation*}
    \mathbf{X}(\mathbf{S},\mathbf{H})= H_{12}(s_1+s_2)(\mathbf{n}_1\otimes\mathbf{n}_2 + \mathbf{n}_2\otimes\mathbf{n}_1).
\end{equation*}
Note that this tensor function can be represented in the form \eqref{10}, where
\begin{equation*}
\mathbb{X}(\mathbf{S}) \equiv (s_1+s_2)(\mathbf{n}_1\otimes\mathbf{n}_1)\!\overset{\text{sym}}{\otimes}\!(\mathbf{n}_2\otimes\mathbf{n}_2).
\end{equation*}
Using the definition of the symmetric external tensor product operation \cite{Curnier1994}, we obtain an explicit representation of the supersymmetric tensor $\mathbb{X}(\mathbf{S})$ in the principal axes of the tensor $\mathbf{S}$
\begin{align}\label{15-5}
\mathbb{X}(\mathbf{S}) = \frac{1}{2} (s_1+s_2)(&\mathbf{n}_1\otimes\mathbf{n}_2\otimes\mathbf{n}_1\otimes\mathbf{n}_2 + \mathbf{n}_2\otimes\mathbf{n}_1\otimes\mathbf{n}_2\otimes\mathbf{n}_1 + \\
&\mathbf{n}_1\otimes\mathbf{n}_2\otimes\mathbf{n}_2\otimes\mathbf{n}_1 + \mathbf{n}_2\otimes\mathbf{n}_1\otimes\mathbf{n}_1\otimes\mathbf{n}_2). \notag
\end{align}
Using the explicit representations of the tensors $\mathbf{X}$ and $\mathbb{X}$ and the definition of the double inner product of tensors, we obtain the scalar
\begin{equation}\label{15-6}
  \mathbf{X}:\mathbf{H}=\mathbf{H}: \mathbb{X}:\mathbf{H} = \mathbf{X}:\tilde{\mathbf{H}}=\tilde{\mathbf{H}}: \mathbb{X}:\tilde{\mathbf{H}} = 2 H^2_{12}(s_1+s_2).
\end{equation}
According to the statements of Proposition \ref{Pr-1}, the necessary and sufficient conditions for the positive definiteness of the tensors $\mathbf{X}$ and $\mathbb{X}$ is the inequality $s_1+s_2>0$, which is confirmed by expression \eqref{15-6}.

In computational mechanics, second-order tensors are usually represented by column vectors, and fourth-order tensors by matrices. Following this rule, we introduce the following vectors and matrix in Voigt notation
\begin{equation*}
  [\mathbf{X}]\equiv \left[
                       \begin{array}{c}
                         X_{11} \\
                         X_{22} \\
                         X_{12} \\
                       \end{array}
                     \right], \quad
    [\mathbf{H}]\equiv \left[
                       \begin{array}{c}
                         H_{11} \\
                         H_{22} \\
                         2H_{12} \\
                       \end{array}
                     \right], \quad
    [\mathbf{C}] \equiv \left[
                       \begin{array}{ccc}
                         C_{11} & C_{12} & C_{13} \\
                         C_{21} & C_{22} & C_{23} \\
                         C_{31} & C_{32} & C_{33} \\
                       \end{array}
                     \right],
\end{equation*}
where $X_{ij}$ and $H_{ij}$ ($i,j=1,2$) are the components of the symmetric tensors $\mathbf{X}$ and $\mathbf{H}$ in the principal axes of the tensor $\mathbf{S}$, and $C_{mn}$ ($m,n=1,2,3$) are the components of the symmetric matrix $\mathbf{C}$, which are related to the components $\mathbb{X}_{ijkl}$ ($i,j,k,l=1,2$) of the tensor $\mathbb{X}$ in the principal axes of the tensor $\mathbf{S}$ by the following expressions:
\begin{align*}
  C_{11} &= \mathbb{X}_{1111},\quad  C_{12} = \mathbb{X}_{1122},\quad C_{13} = \mathbb{X}_{1112} = \mathbb{X}_{1121}, \\
  C_{21} &= \mathbb{X}_{2211},\quad  C_{22} = \mathbb{X}_{2222},\quad C_{23} = \mathbb{X}_{2212} = \mathbb{X}_{2221}, \notag \\
  C_{31} &= \mathbb{X}_{1211},\quad  C_{32} = \mathbb{X}_{1222},\quad C_{33} = \mathbb{X}_{1212} = \mathbb{X}_{1221} = \mathbb{X}_{2112} = \mathbb{X}_{2121}. \notag
\end{align*}
From \eqref{15-5} we obtain the values of the non-zero components of the tensor $\mathbb{X}(\mathbf{S})$ in the principal axes of the tensor $\mathbf{S}$
\begin{equation*}
  \mathbb{X}_{1212} = \mathbb{X}_{1221} = \mathbb{X}_{2112} = \mathbb{X}_{2121} = (s_1+s_2)/2,
\end{equation*}
that is, the matrix $\mathbf{C}$ has the following form:
\begin{equation}\label{AA-4}
      [\mathbf{C}] =  \left[
                       \begin{array}{ccc}
                         0 & 0 & 0 \\
                         0 & 0 & 0 \\
                         0 & 0 & (s_1+s_2)/2 \\
                       \end{array}
                     \right].
\end{equation}
We obtain a vector-matrix counterpart of the tensor expression \eqref{15-6} ($[\mathbf{X}]=[\mathbf{C}]\, [\mathbf{H}]$)
\begin{equation}\label{AA-5}
  [\mathbf{H}]^T\, [\mathbf{X}] = [\mathbf{H}]^T\,[\mathbf{C}]\, [\mathbf{H}] = 2 H^2_{12}(s_1+s_2).
\end{equation}
Naturally, the value of the contraction \eqref{AA-5} in the vector-matrix representation coincide with the value of the same contraction in the tensor representation \eqref{15-6}.
\end{example}

\subsection{Local body deformations and basic kinematics}
\label{sec:2-2}

Consider the motion of a body $\Omega$ in a three-dimensional Euclidean point space, and let $\mathbf{X}$ and $\mathbf{x}$ be the position vectors of some particle $P\in \Omega$ in the \emph{reference} (fixed at time $t_0$) and \emph{current} (moving at time $t$) configurations, respectively. Let $\mathbf{F}\equiv \text{Grad}\,\mathbf{x}= \partial \mathbf{x}/\partial \mathbf{X}\in \mathcal{T}^2$ ($J\equiv\det \mathbf{F} >0$) be the \emph{deformation gradient} (see, e.g., \cite{Bertram2021}).

We use the \emph{left (symmetric, positive definite, Eulerian) stretch tensor} $\mathbf{V}\equiv \sqrt{\mathbf{F}\cdot \mathbf{F}^T}$ as the main kinematic quantity. Let the eigenindex of the tensor $\mathbf{V}$ be equal to $m$. The spectral representations \eqref{3} and \eqref{4} can be written as
\begin{equation*}
\mathbf{V}=\sum_{k=1}^{3} \lambda_k\,\mathbf{n}_k\otimes\mathbf{n}_k=\sum^m_{i=1}\lambda_i \mathbf{V}_i,
\end{equation*}
where $0<\lambda_k<\infty,\ \lambda_k\in \mathds{R}$ ($k=1,2,3$) are the \emph{principal stretches}.

We define the \emph{Finger strain tensor} (see, e.g., \cite{CurnierET1991}) as
\begin{equation}\label{18}
\mathbf{e}^{(2)}\equiv \frac{1}{2}(\mathbf{c}-\mathbf{I})=\frac{1}{2}\sum_{k=1}^{3}(\lambda_k^2-1)\,\mathbf{n}_k\otimes\mathbf{n}_k=\frac{1}{2}\sum^m_{i=1}(\lambda_i^2 -1) \mathbf{V}_i,
\end{equation}
where $\mathbf{c}$ ($\mathbf{B}$ is the alternative standard notation) is the \emph{left Cauchy--Green deformation tensor}
\begin{equation}\label{19}
\mathbf{c}\equiv \mathbf{V}^2 = \mathbf{F}\cdot \mathbf{F}^T=\sum_{k=1}^{3}\lambda_k^2\,\mathbf{n}_k\otimes\mathbf{n}_k=\sum^m_{i=1}\lambda_i^2 \mathbf{V}_i.
\end{equation}

We define the \emph{volume ratio}
\begin{equation}\label{20}
  J\equiv \det \mathbf{F} = \lambda_1 \lambda_2 \lambda_3,
\end{equation}
and the \emph{modified principal stretches}
\begin{equation}\label{21}
  \bar{\lambda}_k\equiv \lambda_k/J^{1/3}\ (k=1,2,3)\ \Rightarrow \bar{J}=\bar{\lambda}_1 \bar{\lambda}_2 \bar{\lambda}_3 = 1,
\end{equation}
which correspond to the eigenvalues of the tensor $\bar{\mathbf{V}}$ due to the Richter--Flory multiplicative decomposition \index{Richter--Flory multiplicative decomposition} \cite{FloryTFS1961}\footnote{It is noted \cite{GrabanMMS2019} (see also \cite{NeffJMPS2025}) that this decomposition was first proposed by Richter in the late 1940s (see \cite{RichterZAMM1948,RichterZAMM1949,RichterAM1949,RichterMN1952}).} of the tensor $\mathbf{V}$
\begin{equation*}
  \mathbf{V}=\bar{\mathbf{V}}\cdot \check{\mathbf{V}},\quad\quad \bar{\mathbf{V}}\equiv J^{-1/3}\mathbf{V},\quad\quad \check{\mathbf{V}}\equiv J^{1/3} \mathbf{I}
\end{equation*}
into a unimodular tensor $\bar{\mathbf{V}}$ and a spherical tensor $\check{\mathbf{V}}$, which are responsible for distortional (isochoric) and volumetric (dilatational) deformations, respectively. We define the \emph{modified Finger strain tensor}
\begin{equation*}
\bar{\mathbf{e}}^{(2)}\equiv \frac{1}{2}(\bar{\mathbf{V}}^2-\mathbf{I}) = \frac{1}{2}(\bar{\mathbf{c}}-\mathbf{I}) = \frac{1}{2}\sum_{k=1}^{3}(\bar{\lambda}_k^2-1)\,\mathbf{n}_k\otimes\mathbf{n}_k = \frac{1}{2}\sum^m_{i=1}(\bar{\lambda}_i^2 -1) \mathbf{V}_i,
\end{equation*}
where
\begin{equation}\label{24}
  \bar{\mathbf{c}}\equiv \bar{\mathbf{V}}^2=J^{-2/3}\mathbf{c}.
\end{equation}
We will further need the \emph{deviator} of the tensor $\bar{\mathbf{e}}^{(2)}$
\begin{equation*}
  \text{dev}\,\bar{\mathbf{e}}^{(2)}\equiv \bar{\mathbf{e}}^{(2)} - \frac{1}{3} \text{tr}\,\bar{\mathbf{e}}^{(2)}\mathbf{I}\quad \Rightarrow \quad  \text{tr}(\text{dev}\,\bar{\mathbf{e}}^{(2)})= \text{dev}\,\bar{\mathbf{e}}^{(2)}:\mathbf{I}=0.
\end{equation*}
It can be shown that
\begin{equation}\label{26}
  \text{dev}\,\bar{\mathbf{e}}^{(2)}=\frac{1}{2}\text{dev}\,\bar{\mathbf{c}},\quad\quad \text{dev}\,\bar{\mathbf{c}}\equiv \bar{\mathbf{c}} - \frac{1}{3} (\text{tr}\,\bar{\mathbf{c}})\mathbf{I}\quad \Rightarrow \quad \text{tr}\,(\text{dev}\,\bar{\mathbf{c}})= \text{dev}\,\bar{\mathbf{c}}:\mathbf{I}=0.
\end{equation}

Hereinafter, we assume that all the tensors $\mathbf{h}\in \mathcal{T}^2$ are sufficiently smooth functions of the monotonically increasing parameter $t$ (time), and we define the \emph{material time derivative}(\emph{material rate}) of the tensor $\mathbf{h}$: $\dot{\mathbf{h}}\equiv\partial \mathbf{h}/\partial t$. We introduce the \emph{spatial velocity vector} $\mathbf{v}$, the \emph{spatial velocity gradient} $\boldsymbol{\ell}$
\begin{equation*}
\mathbf{v} \equiv \dot{\mathbf{x}},\quad\quad \boldsymbol{\ell} \equiv \text{grad}\, \mathbf{v}= \dot{\mathbf{F}} \cdot \mathbf{F}^{-1},
\end{equation*}
the symmetric Eulerian \emph{stretching (strain rate) tensor} $\mathbf{d}\in\mathcal{T}^2_{\text{sym}}$ and the skew-symmetric \emph{vorticity tensor}  $\mathbf{w}\in\mathcal{T}^2_{\text{skew}}$:\footnote{Since we assume that the motion law $\mathbf{x}(\mathbf{X},t)\in C^2$ of $t$, it follows that $\boldsymbol{\ell},\mathbf{d},\mathbf{w}\in C^0$ of $t$ (cf., \cite{ScheidlerMM1991}).}
\begin{equation}\label{28}
\boldsymbol{\ell}=\mathbf{d}+\mathbf{w},\quad\quad \mathbf{d} \equiv \text{sym}\,\boldsymbol{\ell}=\frac{1}{2}(\boldsymbol{\ell} + \boldsymbol{\ell}^T),\quad\quad \mathbf{w} \equiv
\text{skew}\,\boldsymbol{\ell}=\frac{1}{2} (\boldsymbol{\ell} - \boldsymbol{\ell}^T).
\end{equation}

It can be shown that the tensor $\mathbf{d}$ has the following representation of the form \eqref{7}, \eqref{8} (see, e.g., \cite{KorobeynikovAM2011}, Eq. $(115)_2$):
\begin{equation}\label{29}
\mathbf{d} = \hat{\mathbf{d}} + \tilde{\mathbf{d}},\quad\quad \hat{\mathbf{d}}\equiv \sum_{i=1}^{m}\frac{\dot{\lambda}_i}{\lambda_i}\, \mathbf{V}_i,\quad\quad
\tilde{\mathbf{d}}\equiv \sum_{i\neq j=1}^{m} \mathbf{V}_i\cdot \mathbf{d}\cdot\mathbf{V}_j.
\end{equation}
Here $\hat{\mathbf{d}}$ and $\tilde{\mathbf{d}}$ are the components of the tensor $\mathbf{d}$ that are coaxial and orthogonal to the tensor $\mathbf{V}$. Note also the validity of the following equality (see, e.g., \cite{deSouzaNeto2008}, Eq. (3.72)):
\begin{equation}\label{30}
  \dot{J}=J\, \text{tr}\,\mathbf{d}.
\end{equation}

For infinitesimal strains, the following approximate equalities hold (hereinafter, $\boldsymbol{\varepsilon}\equiv \text{sym}\,\text{D}u$ is the \emph{infinitesimal strain tensor}, where $u$ is the displacement vector):
\begin{equation*}
  \mathbf{e}^{(2)}\approx \bar{\mathbf{e}}^{(2)}\approx \boldsymbol{\varepsilon},\quad\quad \mathbf{d}\approx \dot{\boldsymbol{\varepsilon}},\quad\quad \dot{J}/J =\text{tr}\,\mathbf{d}\approx \text{tr}\,\dot{\boldsymbol{\varepsilon}}.
\end{equation*}

\begin{remark}
\label{rem:2-2}
The kinematic tensors $\mathbf{V}$, $\bar{\mathbf{V}}$, $\mathbf{e}^{(2)}$, $\bar{\mathbf{e}}^{(2)}$, $\mathbf{c}$, $\bar{\mathbf{c}}$, and $\mathbf{d}$ are Eulerian objective tensors (cf., \cite{KorobeynikovJElast2008,Ogden1984}).
\end{remark}

\subsection{Stress tensors and their rates}
\label{sec:2-3}

We define the (\emph{true}) \emph{Cauchy stress tensor} $\boldsymbol{\sigma}$ and the (\emph{weighted})\emph{Kirchhoff stress tensor}  $\boldsymbol{\tau}$:
\begin{equation}\label{32}
  \boldsymbol{\tau}\equiv J\, \boldsymbol{\sigma}.
\end{equation}
We also define the (\emph{engineering, nominal}) \emph{first Piola--Kirchhoff} (\emph{1st PK}) \emph{stress tensor} $\mathbf{P}$:
\begin{equation}\label{33}
  (\mathbf{S}_1=)\,\mathbf{P}\equiv \boldsymbol{\tau}\cdot \mathbf{F}^{-T}=J \boldsymbol{\sigma} \cdot \mathbf{F}^{-T}= \boldsymbol{\sigma} \cdot \text{Cof}\,\mathbf{F},
\end{equation}
whose components are usually determined in experimental studies. For hyperelastic materials, the 1st PK stress tensor is determined from the elastic energy $W(\mathbf{F})$
\begin{equation*}
  (\mathbf{S}_1(\mathbf{F})=)\, \mathbf{P}(\mathbf{F}) = \text{D}_F W(\mathbf{F}).
\end{equation*}

\begin{remark}
\label{rem:2-3}
The stress tensors $\boldsymbol{\sigma}$ and $\boldsymbol{\tau}$ are Eulerian objective tensors.\\
\end{remark}

The material rates $\dot{\boldsymbol{\sigma}}$ and $\dot{\boldsymbol{\tau}}$ do not have the Eulerian objectivity property. As the objective rates of the stress tensors we use the Eulerian \emph{Zaremba--Jaumann} (corotational) and \emph{upper Oldroyd} (noncorotational) \emph{stress rates} for the Kirchhoff stress tensor
\begin{equation}\label{34}
   \frac{\text{D}^{\text{ZJ}}}{\text{D}t}[\boldsymbol{\tau}]\equiv \dot{\boldsymbol{\tau}} + \boldsymbol{\tau}\cdot \mathbf{w} - \mathbf{w}\cdot \boldsymbol{\tau},\quad\quad\quad
    \frac{\text{D}^{\overline{\text{Old}}}}{\text{D}t}[\boldsymbol{\tau}]\equiv \dot{\boldsymbol{\tau}} - \boldsymbol{\tau}\cdot \boldsymbol{\ell}^T - \boldsymbol{\ell}\cdot \boldsymbol{\tau},
\end{equation}
and, likewise, for the Cauchy stress tensor
\begin{equation}\label{35}
    \frac{\text{D}^{\text{ZJ}}}{\text{D}t}[\boldsymbol{\sigma}]\equiv \dot{\boldsymbol{\sigma}} + \boldsymbol{\sigma}\cdot \mathbf{w} - \mathbf{w}\cdot \boldsymbol{\sigma},\quad\quad\quad
    \frac{\text{D}^{\overline{\text{Old}}}}{\text{D}t}[\boldsymbol{\sigma}]\equiv \dot{\boldsymbol{\sigma}} - \boldsymbol{\sigma}\cdot \boldsymbol{\ell}^T - \boldsymbol{\ell}\cdot \boldsymbol{\sigma}.
\end{equation}
These rates are related by the equality (see, e.g., \cite{Korobeynikov2023,KorobeynikovZAMM2024}):
\begin{equation}\label{36}
 \frac{\text{D}^{\text{ZJ}}}{\text{D}t}[\boldsymbol{\tau}]=J \frac{\text{D}^{\text{BH}}}{\text{D}t}[\boldsymbol{\sigma}],\quad\quad \frac{\text{D}^{\overline{\text{Old}}}}{\text{D}t}[\boldsymbol{\tau}]=J \frac{\text{D}^{\text{Tr}}}{\text{D}t}[\boldsymbol{\sigma}],
\end{equation}
where the tensors $\frac{\text{D}^{\text{BH}}}{\text{D}t}[\boldsymbol{\sigma}]$ and $\frac{\text{D}^{\text{Tr}}}{\text{D}t}[\boldsymbol{\sigma}]$ are the (Eulerian objective) \emph{Biezeno--Hencky} \cite{Biezeno1928} (or \emph{Hill} \cite{HillJMPS1958}) and \emph{Truesdell} (see, e.g., \cite{Bertram2021}) \emph{stress rates} for the Cauchy stress tensor
\begin{align}\label{37}
    \frac{\text{D}^{\text{BH}}}{\text{D}t}[\boldsymbol{\sigma}]\equiv & \,\dot{\boldsymbol{\sigma}} + \boldsymbol{\sigma}\cdot \mathbf{w} - \mathbf{w}\cdot \boldsymbol{\sigma}  + \boldsymbol{\sigma}\,\text{tr}\,\mathbf{d}\,(=\frac{\text{D}^{\text{ZJ}}}{\text{D}t}[\boldsymbol{\sigma}]+\boldsymbol{\sigma}\,\text{tr}\,\mathbf{d}), \\
    \frac{\text{D}^{\text{Tr}}}{\text{D}t}[\boldsymbol{\sigma}]\equiv & \,\dot{\boldsymbol{\sigma}}- \boldsymbol{\sigma}\cdot \boldsymbol{\ell}^T - \boldsymbol{\ell}\cdot \boldsymbol{\sigma}  +  \boldsymbol{\sigma}\,\text{tr}\,\mathbf{d}\, (=\frac{\text{D}^{\overline{\text{Old}}}}{\text{D}t}[\boldsymbol{\sigma}]+\boldsymbol{\sigma}\,\text{tr}\,\mathbf{d}). \notag
\end{align}

For infinitesimal strains, the following approximate equalities hold:
\begin{equation*}
  \boldsymbol{\tau}\approx \mathbf{P}\approx \boldsymbol{\sigma},\quad\quad\quad \frac{\text{D}^{\text{ZJ}}}{\text{D}t}[\boldsymbol{\tau}]\approx  \frac{\text{D}^{\text{ZJ}}}{\text{D}t}[\boldsymbol{\sigma}] \approx \frac{\text{D}^{\text{BH}}}{\text{D}t}[\boldsymbol{\sigma}]\approx \dot{\boldsymbol{\sigma}}.
\end{equation*}

\subsection{Elastic energy and constitutive relations for the linear isotropic elastic model}
\label{sec:2-4}

Since for infinitesimal strains, the constitutive relations for neo-Hookean models reduce to the constitutive relations for the linear isotropic material model, in this section we provide expressions for the elastic energy and constitutive relations for the latter model as sources inspiring the derivation of corresponding expressions for the former models. In particular, in Section \ref{sec:2-4-1}, we give expressions for the elastic energy and constitutive relations for the incompressible linear isotropic material model, and in Sections \ref{sec:2-4-2} and \ref{sec:2-4-3}, we present two possible representations of constitutive relations and elastic energy for the compressible linear isotropic material model with coupled (Section \ref{sec:2-4-2}) and decoupled (Section \ref{sec:2-4-3}) forms of elastic energy.

\subsubsection{Constitutive relations for incompressible materials}
\label{sec:2-4-1}

The elastic energy for incompressible linear isotropic materials can be written as
\begin{equation}\label{39}
  W_{\text{lin-inc}}\equiv \mu\, \boldsymbol{\varepsilon}:\boldsymbol{\varepsilon} - p\,\text{tr}\,\boldsymbol{\varepsilon} = \mu \|\boldsymbol{\varepsilon}\|^2 - p\,\text{tr}\,\boldsymbol{\varepsilon},
\end{equation}
where $p$ is the indefinite Lagrange multiplier and $\mu>0$ is the \emph{shear modulus} {first Lam\'{e} parameter}. For these materials, the incompressibility condition holds:
\begin{equation}\label{40}
  \text{tr}\,\boldsymbol{\varepsilon}=0.
\end{equation}
Note that for infinitesimal strains, $\text{tr}\,\boldsymbol{\varepsilon}\approx J-1$ holds; i.e., the scalar $\text{tr}\,\boldsymbol{\varepsilon}$ corresponds to the \emph{relative volume ratio} (\emph{dilatation}).

The constitutive relations for this material model can be written as
\begin{equation}\label{41}
  \boldsymbol{\sigma}=\frac{\partial\,W_{\text{lin-inc}}}{\partial\,\boldsymbol{\varepsilon}}\quad \Leftrightarrow \quad \boldsymbol{\sigma}=2\mu\, \boldsymbol{\varepsilon} - p\,\mathbf{I}.
\end{equation}
The infinitesimal strain tensor can be represented as
\begin{equation}\label{42}
  \boldsymbol{\varepsilon}=\text{dev}\,\boldsymbol{\varepsilon} + \frac{1}{3}\text{tr}\,\boldsymbol{\varepsilon}\,\mathbf{I},\quad\quad \text{dev}\,\boldsymbol{\varepsilon} \equiv \boldsymbol{\varepsilon} - \frac{1}{3}\text{tr}\,\boldsymbol{\varepsilon}\,\mathbf{I}.
\end{equation}
Since for incompressible materials, condition \eqref{40} is satisfied, it follows from \eqref{42} that $\boldsymbol{\varepsilon}=\text{dev}\,\boldsymbol{\varepsilon}$ for these materials, and from \eqref{41} we obtain
\begin{equation}\label{43}
  \frac{1}{3}\text{tr}\,\boldsymbol{\sigma}(\equiv \sigma_m)=-p\quad\Rightarrow\quad \sigma_m=-p,
\end{equation}
where $\sigma_m$ is the \emph{mean stress}. It follows from $\eqref{43}_2$ that the Lagrange multiplier $p$ for infinitesimal strains of an incompressible isotropic linear elastic material has the mechanical meaning of the \emph{hydrostatic pressure} in the material.

\subsubsection{Constitutive relations for compressible materials: coupled representation of the elastic energy}
\label{sec:2-4-2}

We represent the elastic energy for linear compressible isotropic materials in the \emph{coupled} form
\begin{equation}\label{44}
  W_{\text{lin-mixed}}\equiv \mu\, \boldsymbol{\varepsilon}:\boldsymbol{\varepsilon} + \frac{\lambda}{2} \text{tr}^2\boldsymbol{\varepsilon} = \mu\, \|\boldsymbol{\varepsilon}\|^2 + \frac{\lambda}{2} \text{tr}^2\boldsymbol{\varepsilon},
\end{equation}
where $\lambda$ is the \emph{second Lam\'{e} parameter}. We call this form coupled since the summands on the r.h.s. of \eqref{44} are not decoupled into volumetric and isochoric strains; i.e., the second summand on the r.h.s. of \eqref{44} depends only on volumetric strain, but the first summand depends on both volumetric and isochoric strains.

The constitutive relations for this material model can be written as
\begin{equation}\label{45}
  \boldsymbol{\sigma}=\frac{\partial\,W_{\text{lin-mixed}}}{\partial\,\boldsymbol{\varepsilon}}\quad \Leftrightarrow \quad \boldsymbol{\sigma} =  2\mu\, \boldsymbol{\varepsilon} + \lambda\,\text{tr}\,\boldsymbol{\varepsilon}\,\mathbf{I}.
\end{equation}
Note the formal similarity between the right-hand sides of $\eqref{41}_2$ and $\eqref{45}_2$ under the following identifications: $- p \mathbf{I}\ \leftrightarrow\ \lambda\,\text{tr}\,\boldsymbol{\varepsilon}\,\mathbf{I}$. However, the mean stress now depends not only on $\lambda$, but also on $\mu$ due to the equality
\begin{equation}\label{46}
  \sigma_m=(\lambda+\frac{2}{3}\mu)\,\text{tr}\,\boldsymbol{\varepsilon} = K \,\text{tr}\,\boldsymbol{\varepsilon},
\end{equation}
where $K\equiv \lambda+\frac{2}{3}\mu$ is the \emph{bulk modulus}.

\subsubsection{Constitutive relations for compressible materials: decoupled representation of the elastic energy}
\label{sec:2-4-3}

The \emph{decoupled} representations of the elastic energy and constitutive relations are obtained from the expressions in Section \ref{sec:2-4-2} using the additive separation of the Cauchy strain tensor into the deviatoric and spherical components in the form $\eqref{42}_1$. In this representation, the elastic energy \eqref{44} is rewritten as
\begin{equation}\label{47}
  W_{\text{lin-vol-iso}}(=W_{\text{lin-mixed}})\equiv \mu\, \text{dev}\,\boldsymbol{\varepsilon} : \text{dev}\,\boldsymbol{\varepsilon} + \frac{K}{2} (\text{tr}\,\boldsymbol{\varepsilon})^2 = \mu\, \|\text{dev}\,\boldsymbol{\varepsilon}\|^2 + \frac{K}{2} \text{tr}^2\boldsymbol{\varepsilon}.
\end{equation}
The second summand on the r.h.s. of \eqref{47} corresponds to volumetric strain, and the first summand corresponds to isochoric strain.

The constitutive relations \eqref{45} can be rewritten as
\begin{equation}\label{48}
  \boldsymbol{\sigma}=\frac{\partial\,W_{\text{lin-vol-iso}}}{\partial\,\boldsymbol{\varepsilon}}\quad \Leftrightarrow \quad \boldsymbol{\sigma} = 2\mu\, \text{dev}\,\boldsymbol{\varepsilon} + K\,\text{tr}\,\boldsymbol{\varepsilon}\,\mathbf{I},
\end{equation}
whence we obtain the expression for the mean stress
\begin{equation*}
  \sigma_m(\equiv\frac{1}{3}\text{tr}\,\boldsymbol{\sigma})=K\,\text{tr}\,\boldsymbol{\varepsilon},
\end{equation*}
which naturally coincides with \eqref{46}.

\section{Constitutive relations for neo-Hookean isotropic hyperelastic material models}
\label{sec:3}

In this chapter we present constitutive relations for three types of \emph{neo-Hookean isotropic hyperelastic material models}: incompressible (classical) model (Section \ref{sec:3-1}), compressible mixed models (Section \ref{sec:3-2}), and compressible vol-iso models (Section \ref{sec:3-3}). The formulations of these three types of models are discussed in Section \ref{sec:3-4}. Although all three types of models are presented in one form or another in the literature (see Section \ref{sec:1}), the objectives of the present section is to unify them in the Eulerian formulation and relate them to the linear isotropic elastic models formulated in Section \ref{sec:2-4}.

\subsection{Constitutive relations for the incompressible neo-Hookean isotropic hyperelastic material model}
\label{sec:3-1}

We represent the incompressible \emph{neo-Hookean} \emph{material model} as the one-power generalized Ogden material model (cf., \cite{OgdenPRSLA1972a}) for $n=2$ with the modified expression for the elastic energy (cf., \cite{KorobeynikovAAM2025})
\begin{equation}\label{50}
   W_{\text{neo-Hooke}}(\lambda_1,\lambda_2,\lambda_3)\equiv \mu\, (\text{tr}\,\mathbf{e}^{(2)} - \ln J) - p \ln J = \frac{\mu}{2}(\|\mathbf{F}\|^2 - 3 - 2\ln J) - p \ln J,
\end{equation}
where the parameter $p$ is the indefinite Lagrange multiplier. According to \eqref{18}, the quantity $\text{tr}\,\mathbf{e}^{(2)}$ is defined as
\begin{equation}\label{51}
  \text{tr}\,\mathbf{e}^{(2)}= \frac{1}{2}(\text{tr}\,\mathbf{c}-3)=\frac{1}{2}(\lambda_1^2+\lambda_2^2+\lambda_3^2-3) = \frac{1}{2}(\|\mathbf{F}\|^2-3).
\end{equation}
Since $\ln J=\ln \lambda_1 + \ln \lambda_2 + \ln \lambda_3$, the potential function \eqref{50} satisfies the Valanis--Landel hypothesis \cite{ValanisIJSS2022,ValanisJAP1967}, i.e.,
\begin{align*}
  W_{\text{neo-Hooke}}(\lambda_1,\lambda_2,\lambda_3) &= \sum_{k=1}^{3}\check{W}_{\text{neo-Hooke}}(\lambda_k),\\
  \check{W}_{\text{neo-Hooke}}(\lambda) &\equiv  \mu\,[\frac{1}{2}(\lambda^2 -1) - \ln \lambda] -p \ln \lambda. \notag
\end{align*}
It has been shown \cite{KorobeynikovAAM2025} that for infinitesimal strains, the elastic energy \eqref{50} reduces to the elastic energy \eqref{39}.

Since for isotropic hyperelastic materials, the principal directions of the tensors $\boldsymbol{\sigma}$ and $\mathbf{V}$ (and $\mathbf{c}$) coincide (see, e.g., \cite{Bertram2021}), the Cauchy stress tensor $\boldsymbol{\sigma}$ can be represented as
\begin{equation*}
\boldsymbol{\sigma}=\sum_{k=1}^{3} \sigma_k\,\mathbf{n}_k\otimes\mathbf{n}_k=\sum^m_{i=1}\sigma_i \mathbf{V}_i,
\end{equation*}
where $\sigma_k$ ($k=1,2,3$) or $\sigma_i$ ($i=1,\ldots,m$) are the \emph{principal Cauchy stresses} defined as follows (see, e.g., \cite{KorobeynikovJElast2019,OgdenPRSLA1972a,Ogden1984}):
\begin{equation}\label{54}
  \sigma_k=\lambda_k\frac{\partial\,W_{\text{neo-Hooke}}}{\partial\,\lambda_k}\ (k=1,2,3)\quad \Leftrightarrow \quad    \sigma_i=\lambda_i\frac{\partial\,W_{\text{neo-Hooke}}}{\partial\,\lambda_i}\ (i=1,\ldots,m).
\end{equation}
Basis-free expressions for the stress tensor $\boldsymbol{\sigma}$ can be obtained from \eqref{50}, \eqref{51}, and \eqref{54} \cite{KorobeynikovAAM2025}
\begin{equation}\label{55}
    \boxed{\boldsymbol{\sigma}= 2\mu\,\mathbf{e}^{(2)} - p\,\mathbf{I} = \mu\,(\mathbf{c}-\mathbf{I}) - p\,\mathbf{I}}.
\end{equation}
It has been shown \cite{KorobeynikovAAM2025} that for infinitesimal strains, the constitutive relations \eqref{55} reduce to the constitutive relations $\eqref{41}_2$ for linear elastic materials.

To determine the mean stress $\sigma_m$, we rewrite the constitutive relations \eqref{55} as
\begin{equation}\label{56}
    \boldsymbol{\sigma}= 2\mu\,\text{dev}\,\mathbf{e}^{(2)}+\frac{2\mu}{3} \text{tr}\,\mathbf{e}^{(2)}\,\mathbf{I} - p\,\mathbf{I},
\end{equation}
where $\text{dev}\,\mathbf{e}^{(2)}$ is the Finger strain tensor deviator defined as \index{Strain tensor!Finger!deviator}
\begin{equation*}
   \text{dev}\,\mathbf{e}^{(2)} \equiv \mathbf{e}^{(2)} - \frac{1}{3}\text{tr}\,\mathbf{e}^{(2)}\,\mathbf{I}.
\end{equation*}
The mean stress $\sigma_m\, (\equiv \text{tr}\,\boldsymbol{\sigma}/3$) can be obtained from \eqref{56}:
\begin{equation}\label{58}
   \sigma_m = - p + \frac{2}{3}\mu\, \text{tr}\,\mathbf{e}^{(2)}.
\end{equation}
Note that unlike the linear elastic incompressible isotropic material model, the Lagrange multiplier $p$ for the neo-Hookean material model can no longer be interpreted as the hydrostatic pressure ($-\sigma_m$). However, for infinitesimal strains,
\begin{equation*}
  \text{tr}\,\mathbf{e}^{(2)}\approx  \text{tr}\,\boldsymbol{\varepsilon}\approx 0,
\end{equation*}
and in this case, equality \eqref{58} becomes equality \eqref{43}; i.e., $\sigma_m= - p$.

\subsection{Constitutive relations for the compressible mixed isotropic hyperelastic neo-Hookean material models}
\label{sec:3-2}

We use the following form of the elastic energy for any \emph{compressible mixed (vol-iso coupled) neo-Hookean} material model (a \emph{mixed} material model) (cf., \cite{KorobeynikovAAM2023,OgdenPRSLA1972b,SimoCMAME1984,Wriggers2008}):
\begin{equation}\label{60}
  W_{\text{mixed}}(\lambda_1,\lambda_2,\lambda_3)\equiv \mu (\text{tr}\,\mathbf{e}^{(2)} - \ln J) + \lambda\, h(J) = \frac{\mu}{2}(\|\mathbf{F}\|^2 - 3 - 2\ln J) + \lambda\, h(J).
\end{equation}
Here $h(J)$ is a scalar \emph{volumetric function} of $J$, whose properties and forms will be considered below\footnote{We now require that for infinitesimal strains, $h(J)\approx \frac{1}{2}(J-1)^2\approx \frac{1}{2}\text{tr}^2\boldsymbol{\varepsilon}$.} and the quantity $\text{tr}\,\mathbf{e}^{(2)}$ is defined in \eqref{51}.

For isotropic hyperelastic materials, the Kirchhoff stress tensor is represented as
\begin{equation}\label{61}
\boldsymbol{\tau}=\sum_{k=1}^{3} \tau_k\,\mathbf{n}_k\otimes\mathbf{n}_k=\sum^m_{i=1}\tau_i \mathbf{V}_i,
\end{equation}
i.e., the principal axes of the tensors $\boldsymbol{\tau}$ and $\mathbf{V}$ coincide. The constitutive relations for this material model can be written as relations between the principal Kirchhoff stresses and the principal stretches (see, e.g., \cite{KorobeynikovJElast2019,Ogden1984})
\begin{equation}\label{62}
  \tau_k=\lambda_k\frac{\partial\,W_{\text{mixed}}}{\partial\,\lambda_k}\ (k=1,2,3)\quad \Leftrightarrow \quad    \tau_i=\lambda_i\frac{\partial\,W_{\text{mixed}}}{\partial\,\lambda_i}\ (i=1,\ldots,m).
\end{equation}
Based on \eqref{61} and \eqref{62}, the Kirchhoff stress tensor can be represented in terms of kinematic quantities by the following basis-free expression (for details, see  \cite{KorobeynikovJElast2021}):
\begin{equation}\label{63}
    \boldsymbol{\tau}= 2\mu\,\mathbf{e}^{(2)} + \lambda\, J h^{\prime}(J) \mathbf{I} =  \mu(\mathbf{c}-\mathbf{I}) + \lambda\, J h^{\prime}(J)  \mathbf{I},
\end{equation}
where
\begin{equation}\label{64}
   h^{\prime}(J)\equiv \frac{d\,h(J)}{d\,J}.
\end{equation}
Based on \eqref{32} and \eqref{63}, the Cauchy stress tensor can be represented by the basis-free expression
\begin{equation}\label{65}
    \boxed{\boldsymbol{\sigma}= \frac{2\mu}{J}\mathbf{e}^{(2)} + \lambda\, h^{\prime}(J) \mathbf{I}  = \frac{\mu}{J}(\mathbf{c}-\mathbf{I}) + \lambda\, h^{\prime}(J) \mathbf{I}.}
\end{equation}

\subsection{Constitutive relations for the compressible vol-iso isotropic hyperelastic neo-Hookean material models}
\label{sec:3-3}

Following \cite{deSouzaNeto2008}, we write the elastic energy for any \emph{compressible vol-iso neo-Hookean} material model (the \emph{vol-iso} material model) as
\begin{equation}\label{66}
  W_{\text{vol-iso}}(\lambda_1,\lambda_2,\lambda_3)\equiv \mu\, (\text{tr}\,\bar{\mathbf{e}}^{(2)} - \ln \bar{J}) + K h(J) = \frac{\mu}{2}(\|\frac{\mathbf{F}}{J^{1/3}}\|^2 - 3) + K h(J).
\end{equation}

The principal components of the Kirchhoff stress tensor $\boldsymbol{\tau}$ are determined using expressions \eqref{62} with the replacement of elastic energy $W_{\text{mixed}}$ by $W_{\text{vol-iso}}$. As a result, we obtain the following basis-free expression for the Kirchhoff stress tensor (for details, see \cite{deSouzaNeto2008}):
\begin{equation}\label{67}
    \boldsymbol{\tau}= 2\mu\,\text{dev}\,\bar{\mathbf{e}}^{(2)} + K J h^{\prime}(J) \mathbf{I} = \mu\,\text{dev}\,\bar{\mathbf{c}} + K\, J h^{\prime}(J) \mathbf{I},
\end{equation}
here we also used equality $\eqref{26}_1$. Equations \eqref{32} and \eqref{67} lead to the following expression for the Cauchy stress tensor:
\begin{equation}\label{68}
    \boxed{\boldsymbol{\sigma}= \frac{2\mu}{J}\,\text{dev}\,\bar{\mathbf{e}}^{(2)} + K\, h^{\prime}(J) \mathbf{I} = \frac{\mu}{J}\,\text{dev}\,\bar{\mathbf{c}} + K\, h^{\prime}(J) \mathbf{I}}.
\end{equation}

\begin{remark}
\label{rem:3-1}
We can show the validity of the equalities
\begin{equation*}
  \frac{\partial\, \bar{J}}{\partial\,\lambda_1}=\frac{\partial\, \bar{J}}{\partial\,\lambda_2}=\frac{\partial\, \bar{J}}{\partial\,\lambda_3}=0,
\end{equation*}
whence the first summand on the r.h.s. of \eqref{66} can be rewritten as $\mu\, \text{tr}\,\bar{\mathbf{e}}^{(2)}$, but expression \eqref{67} for the Kirchhoff stress tensor $\boldsymbol{\tau}$ does not change in this case. We retained the term $\ln \bar{J}$ in expression \eqref{66} to emphasize that this expression for compressible materials is similar to expression \eqref{50} for incompressible materials.
\end{remark}

\subsection{Discussion of expressions for compressible neo-Hookean material models}
\label{sec:3-4}

For infinitesimal strains, the following equality holds:
\begin{equation*}
   \ln J \approx J-1 \approx \text{tr}\,\boldsymbol{\varepsilon}.
\end{equation*}
It has been shown \cite{KorobeynikovAAM2025} that for these strains, expression \eqref{50} for the elastic energy reduces to expression \eqref{39} and expression \eqref{55} for constitutive relations reduces to expression $\eqref{41}_2$. In addition, for small strains, the following approximations are valid:
\begin{equation*}
   h(J) \approx \frac{1}{2}(\ln J)^2 \approx  \frac{1}{2}(\text{tr}\,\boldsymbol{\varepsilon})^2\quad \Rightarrow\quad h^{\prime}(J)\approx J h^{\prime}(J)\approx \ln J \approx \text{tr}\,\boldsymbol{\varepsilon}.
\end{equation*}
Then for infinitesimal strains, expression \eqref{60} for the elastic energy for any mixed material model becomes expression \eqref{44}, and expression \eqref{66} for the elastic energy for any vol-iso material model becomes expression \eqref{47}. In addition, for these strains, expressions \eqref{63} and \eqref{65} for constitutive relations become expression $\eqref{45}_2$, and expressions \eqref{67} and \eqref{68} for constitutive relations become expression $\eqref{48}_2$.

Since, within the framework of linear elasticity theory, both the elastic energy expressions \eqref{44} and \eqref{47} and the constitutive relations \eqref{45} and \eqref{48} are equivalent, the mixed and vol-iso models for infinitesimal strains reduce to the same linear compressible isotropic elastic material model. However, in the general case of strains of arbitrary magnitude, the mixed and vol-iso models are different models of compressible isotropic hyperelastic materials. In particular, the elastic energy \eqref{60} for any mixed material model is an additive decomposition into an uncoupled component (first summand on the r.h.s. of \eqref{60}) and a coupled component (second summand on the r.h.s. of \eqref{60}) with respect to $\lambda_k$ ($k=1,2,3$) (cf., \cite{KellermannZAMM2016}). That is, the elastic energy for any mixed model is consistent with the Valanis--Landel hypothesis \cite{ValanisJAP1967} extended to compressible materials in \cite{ValanisIJSS2022}.\footnote{Valanis \cite{ValanisIJSS2022} requires such decomposition of the elastic energy into uncoupled and coupled components corresponding to the mixed (isochoric and volumetric) and pure volumetric energies.} At the same time, the elastic energy for any vol-iso model is inconsistent with this hypothesis; however, the elastic energy \eqref{66} for this model corresponds to the additive decomposition into a volumetric component (second summand on the right-hand side of \eqref{66}) and an isochoric component (first summand on the right-hand side of \eqref{66}). That is, any vol-iso model allows the uncoupled representation of the elastic energy as a sum of volumetric and isochoric strains, but the mixed model does not allow this representation. With regard to the above decompositions in the elastic energy expressions, we cannot give preference to either one of these two models, since both types of decomposition have not been confirmed experimentally. In particular, experimental studies by Vangerko and Treloar \cite{VangerkoJPD1978} have shown that the Valanis--Landel hypothesis is not valid for sufficiently large values of the principal stretches ($\gtrsim 3$). At the same time, it has been shown \cite{PennTSR1970} that the additive decomposition of the elastic energy into volumetric and isochoric parts is inconsistent with experimental data on the dependence of the dilatation $J-1$ on longitudinal extension under uniaxial loading.

Next, along with the Lam\'{e} parameters $\lambda$ and $\mu$, we will also use more physically reasonable parameters --- \emph{Young's modulus} $E$ and \emph{Poisson's ratio $\nu$}, which are related to the Lam\'{e} parameters as follows (see, e.g., Table 5 in \cite{Rubin2021}):
\begin{equation}\label{72}
  \mu=\frac{E}{2(1+\nu)},\quad\lambda=\frac{E\nu}{(1+\nu)(1-2\nu)}\quad \Leftrightarrow\quad E=\frac{\mu\,(3\lambda+2\mu)}{\lambda+\mu},\quad \nu=\frac{\lambda}{2(\lambda+\mu)}.
\end{equation}

We now perform a preliminary analysis to determine admissible values of the parameters $\lambda$ and $\mu$ for the mixed and vol-iso models. Since for infinitesimal strains, both types of models reduce to the linear elastic compressible material model, we use the well-known constraints on the parameters $K$ and $\mu$ (see, e.g., \cite{Batra2006})
\begin{equation}\label{73}
  \mu>0,\quad\quad K=\lambda+2\mu/3>0,
\end{equation}
under which the potential energy $W_{\text{lin-vol-iso}}$ in \eqref{47} is positive definite, i.e., $\varepsilon \rightarrow W_{\text{lin}}(\varepsilon)$ is strictly convex. In view of expressions \eqref{72}, the constraints \eqref{73} can be represented in terms of $\mu$ and $\nu$:
\begin{equation}\label{74}
  \mu>0,\quad\quad -1<\nu\leq 0.5.
\end{equation}
For isotropic hyperelastic models with elastic energies of the form \eqref{66}, the convexity condition will be imposed on the function $K h(J)$ \cite{HartmannIJSS2003}. Since we assume that $K>0$ (see \eqref{73}), we further consider convex functions $h(J)$. Next, we use identical functions $h(J)$ for both material models (mixed and vol-iso) and relax the constraint on the function $\lambda h(J)$ for the mixed model: we replace the convexity condition by the requirement of non-concavity; i.e., for any mixed model, the parameters $\lambda$ and $\mu$ (or $\nu$ and $\mu$) are subjected to the constraints
\begin{equation}\label{75}
  \mu>0,\ \lambda \geq 0\quad \Leftrightarrow \quad \mu>0,\ 0\leq \nu \leq 0.5,
\end{equation}
while for any vol-iso model, the constraints on the parameters $\nu$ and $\mu$ are still given by \eqref{74}. Note that by our convention, in both constraints \eqref{74} and \eqref{75}, the value $\nu = 0.5$ means the use of the classical neo-Hookean model for incompressible materials instead of the mixed and vol-iso models.

Particular attention should be given to the performance analysis of the mixed and vol-iso models for slightly compressible materials characterized by Poisson's ratio $\nu\approx 0.5^-$ (e.g., $\nu=0.4999$). First, all elastomers in practice are slightly compressible materials, and, second, even if a researcher wants to perform computer simulations of deformations of incompressible materials using some commercial FE systems, he/she will be faced with the fact that these systems use slightly compressible approximations, rather than purely incompressible materials, to impose the incompressibility condition by means of the penalty function method. Ideally, the mixed and vol-iso compressible material models in the case of slight compressibility should lead to similar solutions for stresses and kinematic quantities that approximate the corresponding solutions for incompressible materials. We now find constraints on the model parameters that should lead to this result.

We assume that the principal stretches $\lambda_k$ ($k=1,2,3$) for slightly compressible materials are similar to those for incompressible materials. Since in the last case, $J=1$, it follows that for slightly compressible materials, the approximate equality $J\approx 1$ should hold. In this case, \eqref{21} lead to the approximate equalities $\bar{\lambda}_k\approx \lambda_k$ ($k=1,2,3$). Hence for slightly compressible materials, the first terms on the right-hand sides of \eqref{60} and \eqref{66} should approximate  the first term on the r.h.s. of \eqref{50}.

Since for slightly compressible materials, the \emph{dilatation} $J-1$ is small, functions $h(J)$ satisfy the approximations (see Section \ref{sec:4}):
\begin{equation*}
  h(J)\approx \frac{1}{2}(\ln J)^2\approx \frac{1}{2}(J-1)^2.
\end{equation*}
We assume that the potential energies of volumetric and distortional strains have comparable values. Then in view of \eqref{60} and \eqref{66}, the quantities $\lambda$ and $K$ should far exceed the shear modulus. It follows from \eqref{72} that this condition should be satisfied for values of Poisson's ratio close to 0.5. Typically, in computer simulations of deformations of incompressible materials, it is assumed that $\nu=0.4999$. In this case,
\begin{equation*}
  \frac{\lambda}{\mu}=\frac{2\nu}{1-2\nu} \approx 5000;
\end{equation*}
i.e., $\lambda\approx K \approx 5000\mu$.

We assume that for solutions of boundary-value problems of deformations for slightly compressible hyperelastic materials,
\begin{equation}\label{78}
  \lim \limits_{\nu \to 0.5^-} \frac{\lambda(\nu)}{2}\ln J(\nu)=-p,
\end{equation}
which is in fact equivalent to the application of the penalty function method to the solution of deformation problems for incompressible hyperelastic materials. Comparing the first terms on the right-hand sides of \eqref{60} and \eqref{66} with the first term on the right-hand side of \eqref{50}, we conclude that simulations of deformations using both mixed and vol-iso models in the case of slight compressibility approximate similar simulations using the neo-Hookean model provided that equality \eqref{78} holds.
\vspace*{2mm}

\begin{remark}
\label{rem:3-2}
Note that for slightly compressible materials, both types of models (mixed and vol-iso) can only be used to determine the Cauchy stress tensor $\boldsymbol{\sigma}$ and the principal stretches $\lambda_k$ ($k=1,2,3$). However, the use of the obtained values of the principal stretches generally does not lead to a correct determination of the volume ratio from Eq. \eqref{20}, as shown by the experimental studies of Penn \cite{PennTSR1970}. The reason for the disagreement between the values of the quantity $J$ obtained using standard compressible material models and experimental data is the difference in the scales of mechanical quantities; i.e., although the values of $\lambda_k$ ($k=1,2,3$) can differ greatly from 1, their product in \eqref{20} is close to 1. Various modifications of elastic energies for vol-iso models have been proposed to improve their performance compared to experimental data (see e.g., \cite{AttardIJSS2003,AttardIJSS2004,FongTSR1975,HuangJAM2014,HuangJAM2016,Ogden1984,RogovoyEJMAS2001,YaoPTRSA2022}). However, the performance analysis of these models is beyond the scope of this study.
\end{remark}
\vspace*{2mm}

\begin{remark}
\label{rem:3-3}
Our conclusion that in simulations of deformations, solutions for stresses and principal stretches for slightly compressible materials are close to those for incompressible neo-Hookean materials refer only to finite values of the principal stretches ($0.1\lesssim \lambda_k \lesssim 10$) ($k=1,2,3$), which is quite sufficient for applications. However, for extreme values of $\lambda_k$ ($\lambda_k \rightarrow 0$ and $\lambda_k \rightarrow \infty$), such closeness of solutions is not always the case. The behavior of solutions for extreme values of $\lambda_k$ ($k=1,2,3$) is heavily affected by the choice of the function $h(J)$ (see Section \ref{sec:6}).
\end{remark}

\section{Some volumetric functions and their properties}
\label{sec:4}

The objective of this chapter is to summarize properties of volumetric functions available in the literature (Section \ref{sec:4-1}) and to analyze the performance of eight such functions widely presented in the literature (Section \ref{sec:4-2}).

\subsection{Properties of volumetric functions}
\label{sec:4-1}

Following \cite{HartmannIJSS2003,OgdenPRSLA1972b}, we impose the following constraints on the function $h(J)\in C^2$ ($0<J< \infty$); i.e., we require it to have the following properties:
\begin{enumerate}
  \item Properties of the function $h(J)$ at the point $J=1$ \cite{HartmannIJSS2003,OgdenPRSLA1972b}
\begin{equation*}
  h(1)=0,\quad\quad h^{\prime}\equiv \frac{d\,h(J)}{d\,J})\vert {}_{J=1}=0,\quad\quad h^{\prime\prime} \equiv \frac{d^2\,h(J)}{d\,J^2})\vert {}_{J=1}=1.
\end{equation*}
  \item Properties of the function $h^{\prime}$ \cite{OgdenPRSLA1972b}
\begin{equation*}
  h^{\prime}\vert {}_{J<1}<0,\quad\quad h^{\prime}\vert {}_{J>1}>0.
\end{equation*}
  \item Rank-one convexity of the function $\mathbf{F}\rightarrow h(\det \mathbf{F})$, equivalent to the convexity of $J \rightarrow h(J)$ (desired property of the function $h(J)$) \cite{HartmannIJSS2003}
\begin{equation*}
  h^{\prime\prime}(J)>0\quad \forall\ 0<J<\infty .
\end{equation*}
  \item Hill's stability condition applied to a purely volumetric function for compressible materials \cite{OgdenPRSLA1972b}
\begin{equation}\label{82}
   \chi (J) \equiv  h^{\prime}(J) + Jh^{\prime\prime}(J) > 0\quad \forall\ 0<J<\infty.
\end{equation}
  \item Desired properties of the function $h(J)$ in extreme states \cite{HartmannIJSS2003}
\begin{equation*}
   h(J)\rightarrow \infty\ \text{for}\ J\rightarrow 0\ \text{and}\ J\rightarrow \infty.
\end{equation*}
\end{enumerate}

\begin{remark}
\label{rem:4-1a}
When replacing the variable $J$ with $\log J$ for the function $h(J)$, we obtain a modified function \cite{NeffMMS2020}
\begin{equation*}
  \tilde{h}(\log J) := h(J).
\end{equation*}
It is easy to show the following equality
\begin{equation*}
\chi (J) = \frac{\partial^2\,\tilde{h}(\log J)}{\partial\,(\log J)^2}\frac{1}{J},
\end{equation*}
from which it follows that Hill's stability condition in the sense of \eqref{82} means convexity of the function $\tilde{h}(\log J)$ in $\log J$ or convexity of $\tilde{h}(\text{tr}\,\log \mathbf{V})$ in $\log \mathbf{V}$.
\end{remark}

\subsection{Some volumetric functions}
\label{sec:4-2}

In the literature, one can find some families of functions $h(J)$, which are used, in particular, in hyperelasticity models for slightly compressible materials. Consider the one-parameter (with parameter $q\in \mathds{R}$, $q\geq 0$) family of volumetric functions proposed by Hartmann and Neff \cite{HartmannIJSS2003}:
\begin{equation}\label{84}
h^{(q)}(J) \equiv
\begin{cases}
\frac{1}{2q^2} (J^q+J^{-q}-2)=\frac{1}{2}[\frac{1}{q}(J^{q/2}-J^{-q/2})]^2,   &\text{if}\  q > 0, \\
\frac{1}{2}(\ln J)^2\, [=\lim \limits_{q \to 0} \frac{1}{2q^2} (J^q+J^{-q}-2)],  &\text{if} \ q=0.
\end{cases}
\end{equation}
Note the symmetry of functions of this family:
\begin{equation*}
  h^{(q)}(J^{-1})=h^{(q)}(J).
\end{equation*}
The functions $h^{(q)\,\prime}(J)$ for this family have the form
\begin{equation*}
h^{(q)\,\prime}(J)\, (\equiv \frac{d\,h^{(q)}(J)}{d\,J})=
\begin{cases}
\frac{1}{2q} (J^{q-1}-J^{-q-1}),   &\text{if}\  q > 0, \\
\ln J/J,  &\text{if} \ q=0,
\end{cases}
\end{equation*}
the functions $J h^{(q)\,\prime}(J)$ can be written as
\begin{equation}\label{87}
J h^{(q)\,\prime}(J)=
\begin{cases}
\frac{1}{2q} (J^{q}-J^{-q}),   &\text{if}\  q > 0, \\
\ln J,  &\text{if} \ q=0,
\end{cases}
\end{equation}
and the functions $h^{(q)\,\prime\prime}(J)$ have the form
\begin{equation*}
h^{(q)\,\prime\prime}(J)\, (\equiv \frac{d^2\,h^{(q)}(J)}{d\,J^2}) =
\begin{cases}
\frac{1}{2q} [(q-1)J^{q-2}+(q+1)J^{-q-2}],   &\text{if}\  q > 0, \\
J^{-2}(1-\ln J),  &\text{if} \ q=0.
\end{cases}
\end{equation*}

Other widely used volumetric functions are those from the one-parameter family (with parameter $\beta \in \mathds{R}\smallsetminus{0}$) proposed by Ogden \cite{OgdenPRSLA1972b}
\begin{equation}\label{89}
  h^{(\beta)}(J) \equiv \beta^{-2}(\beta\ln J+J^{-\beta}-1).
\end{equation}
The functions $h^{(\beta)\,\prime}(J)$ for this family have the form
\begin{equation*}
  h^{(\beta)\,\prime}(J)\, (\equiv \frac{d\,h^{(\beta)}(J)}{d\,J})= \beta^{-1}(J^{-1}-J^{-\beta-1}),
\end{equation*}
the function $J h^{(\beta)\,\prime}(J)$ can be written as
\begin{equation*}
  J h^{(\beta)\,\prime}(J) = \beta^{-1}(1-J^{-\beta}),
\end{equation*}
and the functions $h^{(\beta)\,\prime\prime}(J)$ have the form
\begin{equation*}
  h^{(\beta)\,\prime\prime}(J)\,(\equiv \frac{d^2\,h^{(\beta)}(J)}{d\,J^2}) =  \beta^{-1}[(\beta+1)J^{-\beta -2}-J^{-2}].
\end{equation*}
Next, the functions $h^{(q)}(J)$ corresponding to integer values $q=0,\,1,\,2,\,5$ are identified by ID numbers \#1--4, and the  functions $h^{(\beta)}(J)$ corresponding to integer values $\beta=-2,\,-1$ by ID numbers \#5 and \#6 (see Table \ref{t1}).
\begin{table}
\caption{Some volumetric functions $h(J)$ and their properties}
\label{t1}
\begin{tabular}{llllllll}
\hline\noalign{\smallskip}
     &                            & Expression                 &  \multicolumn{5}{c}{Satisfaction of constraints ${}^{\flat}$} \\
  ID &  Equation                  & for $h(J)$                 & (1) & (2) & (3)${}^{\S}$ & (4) & (5) \\
\noalign{\smallskip}\hline\noalign{\smallskip}
   1 & Eq. \eqref{84}, $q=0$      & $(\ln J)^2/2$               &  +  &  +  &  -  &  +  &  +  \\
   2 & Eq. \eqref{84}, $q=1$      & $(J+J^{-1}-2)/2$           &  +  &  +  &  +  &  +  &  +  \\
   3 & Eq. \eqref{84}, $q=2$      & $(J^2+J^{-2}-2)/8$         &  +  &  +  &  +  &  +  &  +  \\
   4 & Eq. \eqref{84}, $q=5$      & $(J^5+J^{-5}-2)/50$        &  +  &  +  &  +  &  +  &  +  \\
   5 & Eq. \eqref{89}, $\beta=-2$ & $(J^2-2\ln J-1)/4$         &  +  &  +  &  +  &  +  &  +  \\
   6 & Eq. \eqref{89}, $\beta=-1$ & $J-\ln J-1$                &  +  &  +  &  +  &  +  &  +  \\
   7 & Eq. $\eqref{93}_1$         & $(J-1)^2/2$                &  +  &  +  &  +  &  -  &  -  \\
   8 & Eq. $\eqref{94}_1$         & $(\mathrm{e}^{\ln^{2}\!\!J} - 1)/2$ &  +  &  +  &  +  &  +  &  +  \\
\noalign{\smallskip}\hline
\end{tabular}
\footnotesize
\begin{itemize}
\item[${}^{\flat}$]$+/-$, the corresponding constraint is satisfied/not satisfied;
\item[${}^{\S}$]The convexity property (3) does not hold for volumetric function \#1 for $J\geq \mathrm{e}\ (\approx\ 2.71)$.
\end{itemize}
\normalsize
\end{table}
Volumetric functions \#1 and \#3--6 with literature references are shown in Table 4 in \cite{HartmannIJSS2003}.

Along with the above volumetric functions, we consider the widely used function (see, e.g., \cite{HartmannIJSS2003}, Table 4)
\begin{equation}\label{93}
  h(J)\equiv \frac{1}{2}(J-1)^2\ \Rightarrow\ h^{\prime}(J)=J-1,\quad J h^{\prime}(J)=J(J-1),\quad h^{\prime\prime}(J)=1.
\end{equation}
This function is denoted by ID number \#7 (see Table \ref{t1}).

In addition, we consider the volumetric function
\begin{align}\label{94}
  h(J) &\equiv \frac{1}{2}(\mathrm{e}^{\ln^{2}\!\!J} - 1) \Rightarrow\ h^{\prime}(J)=\frac{1}{J}\mathrm{e}^{\ln^{2}\!\!J}\ln J,\quad J h^{\prime}(J)=\mathrm{e}^{\ln^{2}\!\!J}\ln J, \\
  h^{\prime\prime}(J) &= \frac{1}{J^{2}}[\mathrm{e}^{\ln^{2}\!\!J}(1 - \ln J + 2\ln^{2}J)]. \notag
\end{align}
This new function is denoted by ID number \#8 (see Table \ref{t1}).

Plots of the functions $h(J)$, $h^{\prime}(J)$, $h^{\prime\prime}(J)$, and $\chi(J)=h^{\prime}(J)+J h^{\prime\prime}(J)$ are presented in Figure \ref{f1}.
\begin{figure}
\begin{center}
\includegraphics{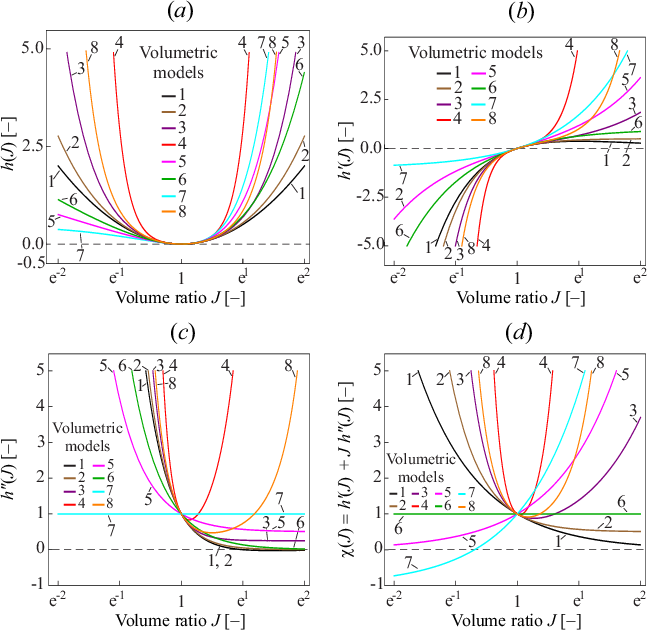}
\end{center}
\caption{Plots of the functions $h(J)$ (\emph{a}), $h^{\prime}(J)$ (\emph{b}), $h^{\prime\prime}(J)$ (\emph{c}), and $\chi(J)=h^{\prime}(J)+Jh^{\prime\prime}(J)$ (\emph{d}).}
\label{f1}
\end{figure}
These plots provide an answer to the question of whether the volumetric functions considered here satisfy constraints (1)--(5); these answers are given in Table \ref{t1}. Analysis of the performance of the volumetric functions presented in Table \ref{t1} shows that not all of these functions satisfy constraints (1)--(5). However, for slightly compressible materials ($J\approx 1$), the plots in Fig. \ref{f1} demonstrate that the approximate equalities
\begin{equation*}
  h(J)\approx\frac{1}{2}(\ln J)^2\approx \frac{1}{2}(J-1)^2,\quad\quad h^{\prime}(J)\approx Jh^{\prime}(J)\approx \ln J \approx J-1
\end{equation*}
are valid in the vicinity of the value $J=1$.

It is interesting to compare the dependencies of the mean stress $\sigma_m=1/3(\sigma_1+\sigma_2+\sigma_3)$ (or the pressure $p\equiv -\sigma_m$) on the stretch ratio $k\,(=\lambda_1=\lambda_2=\lambda_3)$ for dilatational deformation of the form
\begin{equation*}
  \mathbf{V}=k\,\mathbf{I}\quad (0<k<\infty)
\end{equation*}
obtained for the two models considered here (mixed and vol-iso) using the volumetric functions presented in Table \ref{t1} and to compare these dependencies with the experimental data obtained by Bridgman for the volume change of sodium at high pressures \cite{BridgmanPAAAS1923,BridgmanPNAS1935,BridgmanPAAAS1938,BridgmanPAAAS1945}. Setting $\mu=2.53$ GPa and $\nu=0.34$ \cite{KellermannZAMM2016}, we obtain the following values for the quantities $\lambda$ and $K$:
\begin{equation*}
  \lambda=\frac{2\mu\, \nu}{1-2\nu}=5.37\,\text{GPa},\quad\quad K=\lambda+\frac{2}{3}\mu=7.06\,\text{GPa}.
\end{equation*}
Using the expression $J=k^3$, from \eqref{65} we obtain the dependence $\sigma_m(k)$ for the mixed models
\begin{equation}\label{98}
  \sigma_m(k)=\frac{\mu}{k^3}(k^2-1)+\lambda\, h^{\prime}(k^3),
\end{equation}
and from \eqref{68}, the dependence $\sigma_m(k)$ for the vol-iso models
\begin{equation}\label{99}
  \sigma_m(k)=K h^{\prime}(k^3).
\end{equation}
Hereinafter, the ID numbers of the mixed and vol-iso models are identified with the ID numbers of volumetric functions (see Tables \ref{t-Intro-1}, \ref{t-Intro-2}, and \ref{t1}).\footnote{Note that our mixed model \#1 is the \emph{Simo--Pister}hyperelastic compressible material model (cf., \cite{SimoCMAME1984}) and our mixed model \#5 is the \textit{Ciarlet--Geymonat} hyperelastic compressible material model (cf., \cite{CiarletCRASP1982}).} Plots of the functions $\sigma_m(k)$ in \eqref{98} and \eqref{99} for the mixed and vol-iso models are presented in Fig. \ref{f2},\emph{a}, and plots of the functions $p(k)=-\sigma_m(k)$ for these models are presented in Fig. \ref{f2},\emph{b}.
\begin{figure}
\begin{center}
\includegraphics{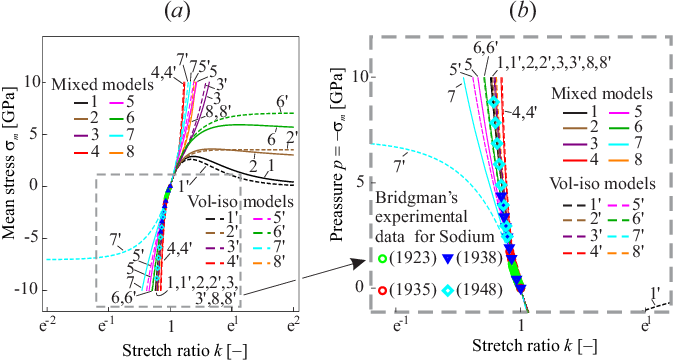}
\end{center}
\caption{Plots of the functions $\sigma_m(k)$ in \eqref{98} and \eqref{99} (\emph{a}) and plots of the functions $p(k)$ (\emph{b}); markers correspond to the experimental values of Bridgman for Sodium.}
\label{f2}
\end{figure}
Markers in the figures represent the experimental dependencies $\sigma(k)$ and $p(k)$ obtained by Bridgman \cite{BridgmanPAAAS1923,BridgmanPNAS1935,BridgmanPAAAS1938,BridgmanPAAAS1945}. It can be seen that in the compression region ($k<1$), the plots of $\sigma(k)$ (or $p(k)$) for models \#1--6 and \#8 are in qualitative agreement with the experimental data of Bridgman (and for models \#1-4,8 they are in quantitative agreement, which is consistent with the conclusions in \cite{HartmannIJSS2003}). However, the vol-iso model \#7 yields a physically unreasonable response in the form of a finite limiting value for $\sigma_m$ for $k\rightarrow 0$, whereas the mixed model \#7 does not lead to a physically unreasonable response. In the extension region ($k>1$), experimental data on the dependence $\sigma_m(k)$ are difficult to obtain. From the point of view of the idealized concept of deformation, the reasonable response manifested in the tendency of $\sigma_m(k)$ to infinity for $k\rightarrow \infty$ is predicted by models \#3--8. However, it is models \#1 and \#2 that correctly describe the response of real materials manifested as loosening under triaxial extension. Obviously, volumetric function \#1 provides the best approximation of the dependencies $\sigma_m(k)$ ($k>1$) for real materials.

\section{Constitutive inequalities for neo-Hookean materials}
\label{sec:5}

As noted in Section \ref{sec:1}, in the literature there are some constitutive inequalities that allow the selection of material models with desired properties and/or the establishment of constraints on the parameters included in the constitutive relations for the material models considered. In this book, we test neo-Hookean material models for satisfaction of the constitutive inequalities due to the \emph{Hill postulate} and the \emph{corotational stability postulate} (CSP), which extend the constitutive inequality due to \emph{Drucker's postulate} from infinitesimal to finite strains. Formulations of these inequalities are given in Section \ref{sec:5-1}, and testing material models for satisfaction of the Hill and CSP postulates is performed in Sections \ref{sec:5-2} and \ref{sec:5-3}, respectively.

\subsection{Hill and CSP constitutive inequalities}
\label{sec:5-1}

\begin{definition}
\label{def:5-1}
(\emph{Drucker's postulate}) (see, e.g., \cite{dAgostinoJElast2025,NeffIJNLM2025,NeffJMPS2025}) A material model satisfies the \emph{Drucker postulate} (\emph{Drucker's material stability condition}) if under infinitesimal strain conditions for the stress rates $\dot{\boldsymbol{\sigma}}$ satisfying the constitutive relations of this model, the following inequality holds at each time instant:
\begin{equation}\label{100}
  \langle\frac{\text{D}}{\text{D}t}[\sigma],\frac{\text{D}}{\text{D}t}[\varepsilon]\rangle\,=\,\dot{\boldsymbol{\sigma}}:\dot{\boldsymbol{\varepsilon}}>0 \quad\quad \forall\ \dot{\boldsymbol{\varepsilon}}\neq \mathbf{0},
\end{equation}
where $\frac{\text{D}}{\text{D}t}$ denotes the material time derivative.
\end{definition}
\vspace*{2mm}

\begin{definition}
\label{def:5-2}
(\emph{Hill's postulate}) \cite{HillJMPS1968,HillPRSLA1970,Hill1979} A material model satisfies the \emph{Hill postulate} if for the objective stress rates $\frac{\text{D}^{\text{ZJ}}}{\text{D}t}[\boldsymbol{\tau}]$ or $\frac{\text{D}^{\text{BH}}}{\text{D}t}[\boldsymbol{\sigma}]$ satisfying the constitutive relations of this model, the following inequality holds at each time instant:
\begin{equation*}
  \langle\frac{\text{D}^{\text{ZJ}}}{\text{D}t}[\tau],d\rangle\,=\,\frac{\text{D}^{\text{ZJ}}}{\text{D}t}[\boldsymbol{\tau}]:\mathbf{d}>0\quad \Leftrightarrow \quad \langle\frac{\text{D}^{\text{BH}}}{\text{D}t}[\sigma],d\rangle\,=\, \frac{\text{D}^{\text{BH}}}{\text{D}t}[\boldsymbol{\sigma}]:\mathbf{d}>0 \quad \forall\ \mathbf{d}\neq \mathbf{0}.
\end{equation*}
\end{definition}
Satisfaction of any material model to the Hill postulate is equivalent to the assertion that the \emph{Kirchhoff} stress tensor $\boldsymbol{\tau}$ is a monotone function in the left Hencky (logarithmic) strain tensor $\log \mathbf{V}\equiv \sum^m_{i=1}\log \lambda_i \mathbf{V}_i$ in the sense that \cite{NeffJMPS2025} (see also \cite{NeffMMS2020})
\begin{align*}
    \langle\tau(\log\,V_1) & -\tau(\log\,V_2),\log\,V_1 - \log\,V_2 \rangle \\
  {} &=(\boldsymbol{\tau}(\log\,\mathbf{V}_1) - \boldsymbol{\tau}(\log\,\mathbf{V}_2)):(\log\,\mathbf{V}_1-\log\,\mathbf{V}_2)>0\quad \forall\, \mathbf{V}_1\neq\mathbf{V}_2.
\end{align*}
In addition, for a hyperelastic solid, the Hill inequality can be equivalently expressed as the convexity of the elastic energy in terms of $\log \mathbf{V}$ since we have the Richter formula \cite{NeffMMS2020}
\begin{equation*}
  \boldsymbol{\tau}= \frac{\partial\,\widehat{W}(\log \mathbf{V})}{\partial\,\log \mathbf{V}},\quad\quad \widehat{W}(\log \mathbf{V}) := W(\mathbf{V}).
\end{equation*}

\begin{remark}
\label{rem:5-1}
The equivalence of the formulations of Hill's inequality for $\frac{\text{D}^{\text{ZJ}}}{\text{D}t}[\boldsymbol{\tau}]$ and $\frac{\text{D}^{\text{BH}}}{\text{D}t}[\boldsymbol{\sigma}]$ is based on equality \eqref{36} ($J>0$).
\end{remark}
\vspace*{3mm}

\begin{definition}
\label{def:5-3}
(\emph{Corotational stability postulate}) \cite{dAgostinoJElast2025,Ghiba2025,NeffIJNLM2025,NeffJMPS2025} A material model satisfies the \emph{corotational stability postulate} (CSP) if for the corotational stress rates $\frac{\text{D}^{\text{o}}}{\text{D}t}[\boldsymbol{\sigma}]$ satisfying the constitutive relations of this model, the following inequality holds at each time instant:
\begin{equation}\label{102}
  \langle\frac{\text{D}^{\text{o}}}{\text{D}t}[\sigma],d\rangle\,=\,\frac{\text{D}^{\text{o}}}{\text{D}t}[\boldsymbol{\sigma}]:\mathbf{d}>0\quad \forall\ \mathbf{d}\neq \mathbf{0},
\end{equation}
where $\frac{\text{D}^{\text{o}}}{\text{D}t}[\boldsymbol{\sigma}]$ is \emph{any} corotational stress rate \cite{NeffAM2025} based on a spin tensor from the family of \emph{continuous material spin tensors} \cite{XiaoIJSS1998,XiaoJElast1998,XiaoArchMech1998} (see also \cite{KorobeynikovAM2011}). In particular, we can use $\frac{\text{D}^{\text{ZJ}}}{\text{D}t}[\boldsymbol{\sigma}]$ as this tensor rate.
\end{definition}
Satisfaction of any material model to the CSP postulate is equivalent to the assertion that the \emph{Cauchy} stress tensor $\boldsymbol{\sigma}$ is a monotonic function in the left Hencky (logarithmic) strain tensor $\log \mathbf{V}$ (cf., \cite{NeffJMPS2025}) meaning that
\begin{align*}
   \langle\sigma(\log\,V_1) &- \sigma(\log\,V_2),\log\,V_1 - \log\,V_2 \rangle \\
  {} &=(\boldsymbol{\sigma}(\log\,\mathbf{V}_1) - \boldsymbol{\sigma}(\log\,\mathbf{V}_2)):(\log\,\mathbf{V}_1-\log\,\mathbf{V}_2)>0\quad \forall\, \mathbf{V}_1\neq\mathbf{V}_2.
\end{align*}
The latter is the \textbf{T}rue \textbf{S}tress \textbf{T}rue \textbf{S}train Hilbert-\textbf{M}onotonicity condition (postulate) $\text{TSTS-M}^+$ (cf., \cite{NeffJMPS2025}).
\vspace*{2mm}

\begin{remark}
\label{rem:5-1-a}
Both compressible neo-Hookean formulations (mixed and vol-iso) are polyconvex, therefore rank-one convex provided that $\mu,\,K,\,\lambda\,>0$ and $h$ is convex in $J$ (cf., \cite{HartmannIJSS2003}).
\end{remark}
\vspace*{2mm}

\begin{remark}
\label{rem:5-2}
The nonequivalence of the constitutive inequalities due to Definitions \ref{def:5-2} and \ref{def:5-3} follows from the equality
\begin{equation}\label{103}
  \langle\frac{\text{D}^{\text{BH}}}{\text{D}t}[\sigma],d\rangle\,=\,\frac{\text{D}^{\text{BH}}}{\text{D}t}[\boldsymbol{\sigma}]:\mathbf{d} = \frac{\text{D}^{\text{ZJ}}}{\text{D}t}[\boldsymbol{\sigma}]:\mathbf{d} + (\boldsymbol{\sigma}:\mathbf{d})\,\text{tr}\,\mathbf{d},
\end{equation}
which is a consequence of equality $\eqref{37}_1$. However, for incompressible materials, $\text{tr}\,\mathbf{d}=0$; then it follows from \eqref{103} that Definitions \ref{def:5-2} and \ref{def:5-3} are equivalent for these materials.
\end{remark}
\vspace*{2mm}

\begin{remark}
\label{rem:5-3}
For infinitesimal strains, both postulates (the Hill and CSP) reduce to the Drucker postulate.
\end{remark}
\vspace*{2mm}

Our present investigation is partly motivated by the fact that polyconvexity or the Legendre--Hadamard ellipticity condition above is clearly not sufficient to guarantee a physically admissible response (see our discussion of the appearance of non-monotonic Cauchy stresses for polyconvex compressible neo-Hookean type models in Section \ref{sec:6}).

\subsection{Testing linear elastic isotropic material models using Drucker's postulate --- convexity of the energy}
\label{sec:5-1a}

We obtain from $\eqref{41}_2$ taking into account equality \eqref{40} the following expression for the contraction for the incompressible material
\begin{equation*}
  \dot{\boldsymbol{\sigma}}:\dot{\boldsymbol{\varepsilon}} = 2\mu \,\dot{\boldsymbol{\varepsilon}}:\dot{\boldsymbol{\varepsilon}}= 2\mu\,\|\dot{\boldsymbol{\varepsilon}}\|^2 = 2\mu\,\|\text{dev}\,\dot{\boldsymbol{\varepsilon}}\|^2,
\end{equation*}
from which follows a necessary and sufficient condition for the fulfillment of the inequality of Drucker's postulate \eqref{100}: $\mu>0$.

Similarly, we obtain from $\eqref{45}_2$ or $\eqref{48}_2$ the contraction expression for a compressible material
\begin{equation}\label{103a}
  \dot{\boldsymbol{\sigma}}:\dot{\boldsymbol{\varepsilon}} = 2\mu\,\|\text{dev}\,\dot{\boldsymbol{\varepsilon}}\|^2 + K\, (\text{tr}\,\dot{\boldsymbol{\varepsilon}})^2.
\end{equation}
The first term on the r.h.s. of \eqref{103a} corresponds to the isochoric strain rate, and the second term corresponds to the volumetric strain rate. Since they cannot be equal to zero simultaneously for an arbitrary non-zero strain rate, the necessary and sufficient conditions for the fulfillment of the Drucker postulate inequality \eqref{100} are the constraints $\mu,\, K>0$ on the parameters of the linear elastic material.

Satisfaction of any linear elastic isotropic material model to the Drucker postulate is equivalent to the assertion that the \emph{Cauchy} stress tensor $\boldsymbol{\sigma}$ is a monotonic function in the infinitesimal strain tensor $\boldsymbol{\varepsilon}$ meaning that
\begin{align*}
 &\langle\frac{\text{D}}{\text{D}t}[\sigma],\frac{\text{D}}{\text{D}t}[\varepsilon]\rangle\,=\,\dot{\boldsymbol{\sigma}}:\dot{\boldsymbol{\varepsilon}}>0 \quad\quad \forall\ \dot{\boldsymbol{\varepsilon}}\neq \mathbf{0}\\
 \Leftrightarrow \quad &  \langle\sigma(\varepsilon_1) - \sigma(\varepsilon_2),\varepsilon_1 - \varepsilon_2 \rangle\,=\, (\boldsymbol{\sigma}(\boldsymbol{\varepsilon}_1) - \boldsymbol{\sigma}(\boldsymbol{\varepsilon}_2)):(\boldsymbol{\varepsilon}_1 - \boldsymbol{\varepsilon}_2)\quad \forall\, \boldsymbol{\varepsilon}_1 \neq \boldsymbol{\varepsilon}_2\\
 \Leftrightarrow \quad & \boldsymbol{\varepsilon} \rightarrow W_{\text{lin}}(\boldsymbol{\varepsilon})\ \text{is strictly convex}\\
 \Leftrightarrow \quad & \mu,\, K>0.
\end{align*}

\subsection{Testing neo-Hookean models using the Hill postulate}
\label{sec:5-2}

We test the incompressible neo-Hookean model and the compressible mixed and vol-iso models for satisfaction of the Hill and CSP postulates in Subsections \ref{sec:5-2-1}, \ref{sec:5-2-2}, and \ref{sec:5-2-3}, respectively.

\subsubsection{Testing the incompressible neo-Hookean material model}
\label{sec:5-2-1}

From \eqref{55} we obtain the non-objective rate form of constitutive relations for the incompressible neo-Hookean model:
\begin{equation}\label{104}
  \dot{\boldsymbol{\sigma}} = \mu\, \dot{\mathbf{c}} - \dot{p}\,\mathbf{I}.
\end{equation}
The equality (see, e.g., \cite{KorobeynikovJElast2008}, p. 126)
\begin{equation*}
  \frac{\text{D}^{\overline{\text{Old}}}}{\text{D}t}[\mathbf{c}] = \mathbf{0},\quad\quad \frac{\text{D}^{\overline{\text{Old}}}}{\text{D}t}[\mathbf{c}]\equiv \dot{\mathbf{c}} - \boldsymbol{\ell} \cdot \mathbf{c} - \mathbf{c}\cdot \boldsymbol{\ell}^T,
\end{equation*}
and Eq. $\eqref{28}_1$ lead to the following expression for $\dot{\mathbf{c}}$:
\begin{equation}\label{106}
  \dot{\mathbf{c}}=(\mathbf{d} + \mathbf{w})\cdot \mathbf{c} + \mathbf{c}\cdot (\mathbf{d} - \mathbf{w}).
\end{equation}
The objective rate form of constitutive relations for this model is obtained using $\eqref{35}_1$, \eqref{55}, \eqref{104}, and \eqref{106}:
\begin{equation}\label{107}
  \frac{\text{D}^{\text{ZJ}}}{\text{D}t}[\boldsymbol{\sigma}]= -\dot{p}\,\mathbf{I} + \mu\,(\mathbf{d}\cdot \mathbf{c} + \mathbf{c}\cdot \mathbf{d}).
\end{equation}

Representing the tensor $\mathbf{c}$ in the form \eqref{19} and the tensor $\mathbf{d}$ in the form \eqref{29} and using the eigenprojection property $\eqref{6}_1$, we obtain
\begin{equation}\label{108}
\mathbf{d}\cdot \mathbf{c} + \mathbf{c}\cdot \mathbf{d}=2\sum_{i=1}^{m}\dot{\lambda}_i \lambda_i\mathbf{V}_i +
\sum_{i\neq j=1}^{m}(\lambda_i^2+\lambda_j^2) \mathbf{V}_i\cdot \mathbf{d}\cdot\mathbf{V}_j.
\end{equation}
The contraction is given by
\begin{equation}\label{109}
  A(\mathbf{d})\equiv (\mathbf{d}\cdot \mathbf{c} + \mathbf{c}\cdot \mathbf{d}):\mathbf{d}.
\end{equation}
Since the first summand on the r.h.s. of \eqref{108} and the component $\hat{\mathbf{d}}$ of the tensor $\mathbf{d}$ in \eqref{29} are coaxial with the tensor $\mathbf{V}$ and since the second summand on the r.h.s. of \eqref{108} and the component $\tilde{\mathbf{d}}$ of the tensor $\mathbf{d}$ in \eqref{29} are orthogonal to the tensor $\mathbf{V}$, using the eigenprojection property $\eqref{6}_3$ and \eqref{109} we obtain
\begin{equation*}
  A(\mathbf{d})=2\sum_{i=1}^{m}(\dot{\lambda}_i)^2\,m_i + [\sum_{i\neq j=1}^{m}(\lambda_i^2+\lambda_j^2) \mathbf{V}_i\cdot \mathbf{d}\cdot\mathbf{V}_j]:\tilde{\mathbf{d}}.
\end{equation*}

The quadratic form $A(\mathbf{d})$ can be represented as a sum of two quadratic forms
\begin{equation}\label{111}
  A(\mathbf{d})=P(\hat{\mathbf{d}}) + R(\tilde{\mathbf{d}}),
\end{equation}
where
\begin{equation}\label{112}
  P(\hat{\mathbf{d}})\equiv 2\sum_{i=1}^{m}(\dot{\lambda}_i)^2\,m_i,\quad\quad R(\tilde{\mathbf{d}})\equiv \tilde{\mathbf{d}}:\mathbb{R}(\mathbf{V}):\tilde{\mathbf{d}},\quad\quad
  \mathbb{R}(\mathbf{V})\equiv \sum_{i\neq j=1}^{m}(\lambda_i^2+\lambda_j^2)\mathbf{V}_i\!\overset{\text{sym}}{\otimes}\!\mathbf{V}_j.
\end{equation}

According to the statements of Proposition \ref{Pr-1}, the tensor $\mathbb{R}(\mathbf{V})\in \mathcal{T}^4_\text{{Ssym}}$, and the quadratic form $R(\tilde{\mathbf{d}})$ is a positive definite form of $\tilde{\mathbf{d}}$. Since $P(\hat{\mathbf{d}})$ is also a positive definite form of $\hat{\mathbf{d}}$, it follows that $A(\mathbf{d})$ is a positive definite quadratic form of $\mathbf{d}$.

Since
\begin{equation}\label{113}
  \mathbf{I}:\mathbf{d}= \text{tr}\,\mathbf{d}
\end{equation}
and since equality \eqref{30} is valid, the equality $\text{tr}\,\mathbf{d}=0$ is valid for incompressible materials ($\dot{J}=0$). Then from \eqref{107} we obtain
\begin{equation*}
  \frac{\text{D}^{\text{ZJ}}}{\text{D}t}[\boldsymbol{\sigma}]:\mathbf{d} = \mu(\mathbf{d}\cdot \mathbf{c} + \mathbf{c}\cdot \mathbf{d}):\mathbf{d} = \mu  A(\mathbf{d}).
\end{equation*}
Since $A(\mathbf{d})$ is a positive definite quadratic form of $\mathbf{d}$ and $\mu>0$, we conclude that the incompressible neo-Hookean material model satisfies the Hill postulate.\\

\begin{remark}
\label{rem:5-4}
The above conclusion is consistent with the conclusions that the Ogden incompressible material model satisfies the Hill postulate (cf., \cite{OgdenPRSLA1972a}) and that the incompressible neo-Hookean material model is Ogden's model with one power term to the power $n=2$.
\end{remark}
\vspace*{2mm}

\begin{remark}
\label{rem:5-4-a}
Note that the positive definiteness of the quadratic form $A(\mathbf{d})$ can be proved more simply in the basis-free form for the tensors $\mathbf{c}$ and $\mathbf{d}$, without using the eigenprojections of the tensor $\mathbf{V}$ or, equivalently, the tensor $\mathbf{c}$. We obtain an alternative representation for the quadratic form $A(\mathbf{d})$, based on expressions \eqref{109}
\begin{align*}
  A(\mathbf{d}) &= (\mathbf{c}\cdot \mathbf{d}):\mathbf{d} +  (\mathbf{d}\cdot \mathbf{c}):\mathbf{d} = (\mathbf{c}\cdot \mathbf{d}):\mathbf{d} + (\mathbf{c}^T\cdot \mathbf{d}^T):\mathbf{d} = 2\, (\mathbf{c}\cdot \mathbf{d}):\mathbf{d}  \\
  &  = 2\,(\mathbf{F}\cdot\mathbf{F}^T\cdot \mathbf{d}):\mathbf{d}=2\,(\mathbf{F}^T\cdot \mathbf{d}):(\mathbf{F}^T\cdot \mathbf{d}) = 2\, \|\mathbf{F}^T\cdot \mathbf{d}\|^2, \notag
\end{align*}
from which the positive definiteness of the quadratic form $A(\mathbf{d})$ follows. Nevertheless, the explicit representation of the quadratic form \eqref{111} turns out to be useful for testing the compressible vol-iso material models (see Section \ref{sec:5-2-3}).
\end{remark}

Alternatively, for the incompressible case, we can write
\begin{equation*}
  W_{\text{neo-Hooke}}(\mathbf{F}) = \frac{\mu}{2}(\|\mathbf{F}\|^2 - 3) = \frac{\mu}{2}(\|\mathbf{V}\|^2 - 3) = \frac{\mu}{2}(\|\exp(\log\mathbf{V})\|^2 - 3) =: \widehat{W}(\log\mathbf{V}),
\end{equation*}
and it is clear that $\log\mathbf{V} \rightarrow \widehat{W}(\log\mathbf{V})$ is convex in $\log\mathbf{V}$, hence Hill's inequality is easily seen to be satisfied.

\subsubsection{Testing the compressible mixed material models}
\label{sec:5-2-2}

By analogy with the derivation of Eq. \eqref{107}, expression \eqref{63} for the Kirchhoff stress tensor $\boldsymbol{\tau}$ leads to the following objective rate form of constitutive relations for any mixed material model (here we also use equality \eqref{30}):
\begin{equation}\label{115}
  \frac{\text{D}^{\text{ZJ}}}{\text{D}t}[\boldsymbol{\tau}]= \mu\,(\mathbf{d}\cdot \mathbf{c} + \mathbf{c}\cdot \mathbf{d}) + \lambda\, \chi(J)J\,\text{tr}\,\mathbf{d}\, \mathbf{I},
\end{equation}
the quantity $\chi(J)$ is defined in \eqref{82}. Next, using notation \eqref{109} and equality \eqref{113}, we get
\begin{equation}\label{116}
  \frac{\text{D}^{\text{ZJ}}}{\text{D}t}[\boldsymbol{\tau}]:\mathbf{d} = \mu  A(\mathbf{d}) + \lambda\, \chi(J)J(\text{tr}\,\mathbf{d})^2.
\end{equation}
Note that for $\chi(J)>0$, the second summand on the r.h.s. of \eqref{116} is a non-positive definite quadratic form of $\mathbf{d}$. We allow, first, that for this material model, the Lam\'{e} parameter $\lambda=0$ (and hence $\nu=0$) (see \eqref{75}), and, second, that the scalar $\text{tr}\,\mathbf{d}$ can take zero value for $\mathbf{d}\neq 0$, whence it follows that this quadratic form is non-positive definite. However, since $A(\mathbf{d})$ is a positive definite form (see Section \ref{sec:5-2-1}) and since we assume that $\mu>0$, the r.h.s. of \eqref{116} is a positive definite quadratic form of $\mathbf{d}$ for $\chi(J)>0$. Thus, inequality $\chi(J)>0$ is a sufficient condition for any mixed compressible neo-Hookean model to satisfy the Hill postulate.\\

\begin{remark}
\label{rem:5-5}
The above statement agrees with the conclusion that the Ogden compressible material model satisfies the Hill postulate for $\chi(J)>0$ (cf., \cite{OgdenPRSLA1972b}).
\end{remark}

\subsubsection{Testing the compressible vol-iso material models}
\label{sec:5-2-3}

Expression \eqref{67} leads to the following non-objective rate form of constitutive relations for the Kirchhoff stress tensor $\boldsymbol{\tau}$:
\begin{equation}\label{117}
  \dot{\boldsymbol{\tau}}= \mu\,\dot{\overline{\text{dev}\,\bar{\mathbf{c}}}} + K \chi(J)J\text{tr}\,\mathbf{d}\, \mathbf{I}.
\end{equation}
Using \eqref{24}, from $\eqref{26}_2$ we obtain the equality
\begin{equation}\label{118}
  \dot{\overline{\text{dev}\,\bar{\mathbf{c}}}} = -\frac{2}{3}J^{-5/3}\dot{J}(\mathbf{c}-\frac{1}{3}\text{tr}\,\mathbf{c}\,\mathbf{I}) + J^{-2/3}(\dot{\mathbf{c}}-\frac{1}{3}\text{tr}\,\dot{\mathbf{c}}\,\mathbf{I}).
\end{equation}
In view of equalities \eqref{30} and \eqref{106}, equality \eqref{118} can be transformed to
\begin{equation}\label{119}
  \dot{\overline{\text{dev}\,\bar{\mathbf{c}}}} = J^{-2/3}\left[-\frac{2}{3}(\mathbf{c}-\frac{1}{3}\text{tr}\,\mathbf{c}\,\mathbf{I})\,\text{tr}\,\mathbf{d} + (\mathbf{d} + \mathbf{w})\cdot \mathbf{c} + \mathbf{c}\cdot (\mathbf{d} - \mathbf{w})-\frac{1}{3}\text{tr}\,\dot{\mathbf{c}}\,\mathbf{I}\right].
\end{equation}
Representation \eqref{106} of the tensor $\dot{\mathbf{c}}$ leads to the equality
\begin{equation}\label{120}
  \text{tr}\,\dot{\mathbf{c}}= \text{tr}(\mathbf{d}\cdot \mathbf{c} + \mathbf{w}\cdot \mathbf{c} + \mathbf{c}\cdot \mathbf{d} - \mathbf{c}\cdot \mathbf{w}) = 2\mathbf{c}:\mathbf{d},
\end{equation}
which was derived using the equalities
\begin{equation*}
  \text{tr}(\mathbf{d}\cdot \mathbf{c})= \text{tr}(\mathbf{c}\cdot \mathbf{d})=\mathbf{c}:\mathbf{d},\quad\quad \text{tr}(\mathbf{w}\cdot \mathbf{c})=\mathbf{w}:\mathbf{c}=0,\quad\quad \text{tr}(\mathbf{c}\cdot \mathbf{w})=\mathbf{c}:\mathbf{w}=0.
\end{equation*}
Using \eqref{30}, \eqref{119}, and \eqref{120}, constitutive relations \eqref{67}, and the definition of the Zaremba--Jaumann stress rate $\eqref{34}_1$, we transform \eqref{117} to the objective rate form of constitutive relations for any vol-iso neo-Hookean material model:
\begin{equation}\label{122}
  \frac{\text{D}^{\text{ZJ}}}{\text{D}t}[\boldsymbol{\tau}]= K \chi(J)J\text{tr}\,\mathbf{d}\, \mathbf{I} + \mu J^{-2/3}\left[-\frac{2}{3}(\mathbf{c} - \frac{1}{3}\text{tr}\,\mathbf{c}\,\mathbf{I})\text{tr}\,\mathbf{d} + \mathbf{d}\cdot \mathbf{c} + \mathbf{c}\cdot \mathbf{d} - \frac{2}{3}(\mathbf{c}:\mathbf{d})\mathbf{I}\right].
\end{equation}
Performing the contraction of the left- and right-hand sides of Eq. \eqref{122} with the tensor $\mathbf{d}$ and using equality \eqref{113}, we arrive at the equality
\begin{equation}\label{123}
  \frac{\text{D}^{\text{ZJ}}}{\text{D}t}[\boldsymbol{\tau}]:\mathbf{d} = K \chi(J)J(\text{tr}\,\mathbf{d})^2 + \mu J^{-2/3}[-\frac{4}{3}(\mathbf{c}:\mathbf{d})\text{tr}\,\mathbf{d} + \frac{2}{9}\text{tr}\,\mathbf{c}(\text{tr}\,\mathbf{d})^2 + (\mathbf{d}\cdot \mathbf{c} + \mathbf{c}\cdot \mathbf{d}):\mathbf{d}].
\end{equation}
Since the unit tensor $\mathbf{I}$ is coaxial with the tensor $\mathbf{V}$, using the property $\eqref{6}_3$ of eigenprojections and equalities \eqref{19} and \eqref{29}, we obtain
\begin{equation}\label{124}
  \text{tr}\,\mathbf{d}=\mathbf{d}:\mathbf{I}=\hat{\mathbf{d}}:\mathbf{I}=\sum_{i=1}^{m}\frac{\dot{\lambda}_i}{\lambda_i}m_i,\quad\quad \text{tr}\,\mathbf{c}=\sum_{i=1}^{m}\lambda_i^2 m_i,\quad\quad \mathbf{c}:\mathbf{d}=\mathbf{c}:\hat{\mathbf{d}}=\sum_{i=1}^{m}\dot{\lambda}_i\lambda_i m_i.
\end{equation}
We introduce the quadratic forms
\begin{equation}\label{125}
  B(\hat{\mathbf{d}})\equiv(\mathbf{c}:\hat{\mathbf{d}})\,\text{tr}\,\hat{\mathbf{d}},\quad\quad C(\hat{\mathbf{d}})\equiv \text{tr}\,\mathbf{c}\,(\text{tr}\,\hat{\mathbf{d}})^2.
\end{equation}
In view of expressions \eqref{124}, the quadratic forms \eqref{125} can be rewritten in explicit form
\begin{equation}\label{126}
  B(\hat{\mathbf{d}})= (\sum_{i=1}^{m}\dot{\lambda}_i\lambda_i m_i)(\sum_{j=1}^{m}\frac{\dot{\lambda}_j}{\lambda_j}m_j),\quad\quad
  C(\hat{\mathbf{d}})= (\sum_{i=1}^{m}\lambda_i^2 m_i)(\sum_{j=1}^{m}\frac{\dot{\lambda}_j}{\lambda_j}m_j)^2.
\end{equation}
The quadratic form in square brackets on the r.h.s. of \eqref{123} will be denoted by
\begin{equation*}
  D(\mathbf{d})\equiv -\frac{4}{3}B(\hat{\mathbf{d}}) + \frac{2}{9}C(\hat{\mathbf{d}}) + A(\mathbf{d}).
\end{equation*}
According to the representation \eqref{111} of the quadratic form $A(\mathbf{d})$, the quantity $D(\mathbf{d})$ can be written as
\begin{equation}\label{128}
  D(\mathbf{d})= E(\hat{\mathbf{d}}) + R(\tilde{\mathbf{d}}),
\end{equation}
where
\begin{equation}\label{129}
  E(\hat{\mathbf{d}})\equiv -\frac{4}{3}B(\hat{\mathbf{d}}) + \frac{2}{9}C(\hat{\mathbf{d}}) + P(\hat{\mathbf{d}}).
\end{equation}

Our next goal is to determine the properties of the quadratic form $E(\hat{\mathbf{d}})$. Consider the general case $m=3$, i.e., $m_i=1$ ($i=1,2,3$). We use expressions $\eqref{112}_1$ and \eqref{125} for the quadratic forms on the right-hand side of \eqref{129}. After a series of transformations, we arrive at the final expression for the quadratic form $E(\hat{\mathbf{d}})$:
\begin{align}\label{130}
  E(x_1,x_2,x_3)\equiv\, & x_1^{2}(4+a^{-2}+b^{-2}) + x_2^{2}(4+a^{2}+c^{-2}) \\
  &+ x_3^{2}(4+b^{2}+c^{2}) + x_1 x_2(-4a-4a^{-1}+2b^{-1}c^{-1}) \notag \\
  &+ x_1 x_3(-4b-4b^{-1}+2a^{-1}c) + x_2 x_3 (-4c-4c^{-1}+2ab), \notag
\end{align}
where the following notations are used:
\begin{equation*}
  x_1\equiv \dot{\lambda}_1^2,\quad x_2\equiv \dot{\lambda}_2^2,\quad x_3\equiv \dot{\lambda}_3^2,\quad a\equiv\frac{\lambda_1}{\lambda_2},\quad b\equiv\frac{\lambda_1}{\lambda_3},\quad c\equiv\frac{\lambda_2}{\lambda_3}\quad (a>0,\ b>0,\ c>0).
\end{equation*}
In view of the identity
\begin{equation*}
  (x-y-z)^2=x^2+y^2+z^2-2xy-2xz-2yz,
\end{equation*}
which is valid for all $x,y,z\in \mathds{R}$, expression \eqref{130} for the quadratic form $ E(x_1,x_2,x_3)$ can be simplified to
\begin{equation}\label{133}
   E(x_1,x_2,x_3)=(2x_1-x_2 a -x_3 b)^2 + (2x_2-x_1 a^{-1} -x_3 c)^2 + (2x_3-x_1 b^{-1} -x_2 c^{-1})^2.
\end{equation}
It follows from \eqref{133} that $E(\hat{\mathbf{d}})\geq 0\ \forall\ \hat{\mathbf{d}}$ and $\exists\ \hat{\mathbf{d}}\neq \mathbf{0}$ are such that $E(\hat{\mathbf{d}})= 0$; i.e., the quadratic form $E(\hat{\mathbf{d}})$ is semi-positive definite. Since the quadratic form $R(\tilde{\mathbf{d}})$ is positive definite, it follows from \eqref{128} that the quantity $D(\mathbf{d})$ is semi-positive definite, i.e.,
\begin{equation*}
  D(\mathbf{d})\geq 0\ \forall\, \mathbf{d}\quad \text{and}\quad \exists\ \mathbf{d}\neq \mathbf{0}\ \text{that}\ D(\mathbf{d})=0.
\end{equation*}

The expression for the quadratic form $E(x_1,x_2,x_3)$ for $m=2$ is derived from expression \eqref{133} by setting in the latter
\begin{equation*}
  x_1\equiv \dot{\lambda}_1^2,\quad x_2=x_3\equiv \dot{\lambda}_2^2,\quad a=b\equiv\frac{\lambda_1}{\lambda_2},\quad c=1\quad (a=b>0).
\end{equation*}
For $m=1$, we get the identities
\begin{equation*}
  x_1=x_2=x_3\equiv \dot{\lambda}_1^2,\quad\quad a=b=c=1,
\end{equation*}
whence it follows that $E(\hat{\mathbf{d}})=0\ \forall\, \hat{\mathbf{d}}$. Since in this case, $R(\tilde{\mathbf{d}})=0$, we have $D(\mathbf{d})=0\ \forall\, \mathbf{d}$; i.e., only the first summand is retained on the right-hand side of \eqref{123}.

Returning to expression \eqref{123}, we note that for $\chi(J)>0$, the first summand on the r.h.s. is a semi-positive definite quadratic form which can vanish for body motions with $\text{tr}\,\mathbf{d}=0$ for $\mathbf{d}\neq \mathbf{0}$. Since the second summand on the r.h.s. is also a semi-positive definite quadratic form, we need to determine whether both forms can vanish simultaneously when $\mathbf{d}\neq \mathbf{0}$.

An explicit expression for the contraction $\text{tr}\,\mathbf{d}$ is presented in $\eqref{124}_1$. We now determine conditions under which the quadratic form \eqref{133} vanishes for nonzero values of $x_1$, $x_2$, and $x_3$. Since in this case, each expression in parentheses on the r.h.s. of \eqref{133} must vanish, we come to the search for nontrivial solutions of the system of homogeneous equation
\begin{equation}\label{137}
  \mathbf{A}\mathbf{x}=\tilde{\mathbf{0}},\quad\quad
\mathbf{A}\equiv \left[
     \begin{array}{rrr}
    2       & -a      & -b \\
    -a^{-1} & 2       & -c \\
    -b^{-1} & -c^{-1} & 2 \\
     \end{array}
     \right],\quad
\mathbf{x}\equiv \left[
                   \begin{array}{c}
                     x_1 \\
                     x_2 \\
                     x_3 \\
                   \end{array}
                 \right],\quad
\tilde{\mathbf{0}}\equiv \left[
                   \begin{array}{c}
                     0 \\
                     0 \\
                     0 \\
                   \end{array}
                 \right].
\end{equation}
We are interested in nontrivial solutions of system \eqref{137} that are possible when the quantity $\det \mathbf{A}$ vanishes, whence we obtain
\begin{equation*}
  \det \mathbf{A}=(b-ac)^2\ \Rightarrow\ \det \mathbf{A}=0\ \Rightarrow\ b=ac\ \Leftrightarrow\ \frac{\lambda_1}{\lambda_3} = \frac{\lambda_1}{\lambda_3};
\end{equation*}
i.e., the equality $\det \mathbf{A}=0$ holds for any values of the principal stretches $\lambda_1$, $\lambda_2$, and $\lambda_3$. Thus, for any values of the principal stretches, there exist values of the quantities $\dot{\lambda}_1$, $\dot{\lambda}_2$, and $\dot{\lambda}_3$, not equal to zero simultaneously, for which the quadratic form $E(\dot{\lambda}_1,\dot{\lambda}_2,\dot{\lambda}_3)$ vanishes. The only nontrivial solution of the system is
\begin{equation}\label{139}
  x_1\in R,\ x_2=a^{-1}x_1,\ x_3=b^{-1}x_1\quad\Leftrightarrow\quad \frac{\dot{\lambda}_1}{\lambda_1}\in \mathds{R},\ \frac{\dot{\lambda}_2}{\lambda_2}=\frac{\dot{\lambda}_1}{\lambda_1},\ \frac{\dot{\lambda}_3}{\lambda_3}=\frac{\dot{\lambda}_1}{\lambda_1}.
\end{equation}
A particular case of this solution is the solution for $m=1$ ($\lambda_1=\lambda_2=\lambda_3$, $\dot{\lambda}_1=\dot{\lambda}_2=\dot{\lambda}_3$) considered above (we established that in this case, $E(\hat{\mathbf{d}})=0$). Comparing the expressions in \eqref{139} and $\eqref{124}_1$, we see that the quantities $(\text{tr}\,\mathbf{d})^2$ and  $E(\hat{\mathbf{d}})$ cannot vanish simultaneously for  motions with $\hat{\mathbf{d}}\neq \mathbf{0}$. The results of our analysis lead to the conclusion that any vol-iso  neo-Hookean material model with $\chi(J)>0$ satisfies the Hill postulate.

An alternative derivation of Hill's inequality is based on the convexity property of elastic energy for the vol-iso material models. Let us represent the elastic energy for these  material models in the form ($V = \sqrt{F \, F^T}$)
\begin{equation}
	\mathrm{W}_{\text{vol-iso}}(F) = \frac{\mu}{2} \, \Big\Vert \frac{F}{(\det F)^{\frac{1}{3}}} \Big\Vert^2 + K \, h(\det F) =
		 \frac{\mu}{2} \, \Big\Vert \frac{V}{(\det V)^{\frac{1}{3}}} \Big\Vert^2 + K \, h(\det V) =: \widehat{\mathrm{W}}(\log V)\, .
\end{equation}
We shall show that $\log V \mapsto \widehat{\mathrm{W}}(\log V)$ is strictly convex if $h \colon \mathds{R}^+ \to \mathds{R}$ satisfies $h^{\prime \prime}(\mathrm{e}^{\xi}) \, \mathrm{e}^{\xi} + h^{\prime}(\mathrm{e}^{\xi}) > 0$ for all $\xi \in \mathds{R}$. First, we have to express $\widehat{\mathrm{W}}$. Using the properties of $\log V$ it holds that
\begin{equation}
	\log \det V = \text{tr}(\log V) \quad \iff \quad \det V = \mathrm{e}^{\text{tr}(\log V)}\,,
\end{equation}
\begin{equation}
	\begin{alignedat}{2}
		\frac{V}{(\det V)^{\frac{1}{3}}} &= \frac{\exp(\log V)}{(\det V)^{\frac{1}{3}}} = \frac{\exp(\log V)}{\mathrm{e}^{\frac{1}{3} \, \text{tr}(\log V)}} = \exp(\log V - \frac{1}{3}\, \text{tr}(\log V) \, I) = \exp(\text{dev} \log V) \, .
	\end{alignedat}
\end{equation}
Therefore,
\begin{equation}
	\begin{alignedat}{2}
		\widehat{\mathrm{W}}(\log V) &= \frac{\mu}{2} \, \|\exp(\text{dev} \log V)\|^2 + K \, h(\mathrm{e}^{\text{tr}(\log V)}) \, ,\\
		\widehat{\mathrm{W}}(S) &= \frac{\mu}{2} \, \|\exp(\text{dev} S)\|^2 + K \, h(\mathrm{e}^{\text{tr}(S)}) \, .
	\end{alignedat}
\end{equation}
Note that $\text{tr}(S) \mapsto \mathrm{e}^{\text{tr}(S)}$ is clearly convex, and if $h \colon \mathds{R}^+ \to \mathds{R}$ satisfies $h^{\prime \prime}(\mathrm{e}^{\xi}) \, \mathrm{e}^{\xi} + h^{\prime}(\mathrm{e}^{\xi}) > 0$ for all $\xi \in \mathds{R}$, then the composition $\text{tr}(S) \mapsto h(\mathrm{e}^{\text{tr}(S)})$ is strictly convex in $\text{tr}(S)$. This amounts to $\chi(J) > 0$.  Moreover, it is well known that $X \mapsto \|\exp(X)\|^2$ is strictly convex \cite{HillPRSLA1970,Ogden1984}, therefore $\text{dev}\, S \mapsto \|\exp(\text{dev}\, S)\|^2$ is strictly convex. Since $S = \text{dev}\, S + \frac{1}{3} \, \text{tr}(S) \, \mathbf{I}$ this shows that for $\mu, K > 0$ the function $S \mapsto \widehat{\mathrm{W}}(S)$ is strictly convex. Using Richter's formula \cite{RichterZAMM1948,RichterZAMM1949,RichterAM1949,RichterMA1950,RichterMN1952} for the Kirchhoff stress
\begin{equation}
	\tau = \mathrm{D}_{\log V} \widehat{\mathrm{W}}(\log V)
\end{equation}
we observe that $\tau$ must then be strictly monotone in $\log V$, i.e.
\begin{equation}
	\langle \tau(\log V_1) - \tau(\log V_2) , \log V_1 - \log V_2 \rangle > 0 \quad \forall V_1 \neq V_2
\end{equation}
and this is equivalent to Hill's inequality.

\subsection{Testing neo-Hookean models using the CSP}
\label{sec:5-3}

As noted in Remark \ref{rem:5-2}, the Hill postulate and the CSP are equivalent for incompressible materials, and since we established in Section \ref{sec:5-2-1} that the incompressible neo-Hookean  model satisfies the Hill postulate, this model also satisfies the CSP.

We now need to determine whether any mixed model satisfies inequality \eqref{102}. In view of $\eqref{36}_1$, Eq. \eqref{103} leads to the following expression for the contraction on the left-hand side of inequality \eqref{102}:
\begin{equation}\label{140}
  \frac{\text{D}^{\text{ZJ}}}{\text{D}t}[\boldsymbol{\sigma}]:\mathbf{d}=\frac{1}{J}\frac{\text{D}^{\text{ZJ}}}{\text{D}t}[\boldsymbol{\tau}]:\mathbf{d} - (\boldsymbol{\sigma}:\mathbf{d})\, \text{tr}\,\mathbf{d}.
\end{equation}
Using expression \eqref{116} for the contraction $\frac{\text{D}^{\text{ZJ}}}{\text{D}t}[\boldsymbol{\tau}]:\mathbf{d}$ and expression \eqref{65} for the Cauchy stress tensor $\boldsymbol{\sigma}$, from \eqref{140} we obtain
\begin{equation}\label{141}
  \frac{\text{D}^{\text{ZJ}}}{\text{D}t}[\boldsymbol{\sigma}]:\mathbf{d}= \lambda J h^{\prime\prime}(\text{tr}(J)\,\mathbf{d})^2 + \frac{\mu}{J}[A(\mathbf{d}) + (\text{tr}\,\mathbf{d})^2 - (\mathbf{c}:\mathbf{d})\,\text{tr}\,\mathbf{d}],
\end{equation}
where the quadratic form $A(\mathbf{d})$ is defined in \eqref{109}. We introduce the quadratic form
\begin{equation}\label{142}
  F(\hat{\mathbf{d}})\equiv (\text{tr}\,\hat{\mathbf{d}})^2 (=(\text{tr}\,\mathbf{d})^2)=(\sum_{i=1}^{m}\frac{\dot{\lambda}_i}{\lambda_i}m_i)^2,
\end{equation}
where equality $\eqref{124}_1$  is used. In view of equalities \eqref{111}, $\eqref{125}_1$, and \eqref{142}, the expression on the r.h.s. of \eqref{141} can be rewritten as
\begin{equation*}
  \frac{\text{D}^{\text{ZJ}}}{\text{D}t}[\boldsymbol{\sigma}]:\mathbf{d}= \lambda Jh^{\prime\prime}(J) F(\hat{\mathbf{d}}) + \frac{\mu}{J}[P(\hat{\mathbf{d}}) + R(\tilde{\mathbf{d}}) + F(\hat{\mathbf{d}}) - B(\hat{\mathbf{d}})].
\end{equation*}
We are interested in the properties of the quadratic form
\begin{equation*}
  G(\hat{\mathbf{d}})\equiv P(\hat{\mathbf{d}}) +  F(\hat{\mathbf{d}}) - B(\hat{\mathbf{d}}).
\end{equation*}
Equalities $\eqref{112}_1$, $\eqref{126}_1$, and \eqref{142} lead to the explicit expression for  this quadratic form
\begin{equation*}
  G(\hat{\mathbf{d}})=2\sum_{i=1}^{m}(\dot{\lambda}_i)^2m_i + (\sum_{i=1}^{m}\frac{\dot{\lambda}_i}{\lambda_i}m_i)^2 -
(\sum_{i=1}^{m}\dot{\lambda}_i\lambda_i m_i)(\sum_{j=1}^{m}\frac{\dot{\lambda}_j}{\lambda_j}m_j).
\end{equation*}
This quadratic form is neither positive nor semi-positive definite. This statement can be illustrated by the following example. Let $m=1$, i.e., let $\lambda_1=\lambda_2=\lambda_3$ and $m_1=3$. Then
\begin{equation*}
  G(\hat{\mathbf{d}})= 3 (\dot{\lambda}_1)^2(\frac{3}{\lambda_1^2}-1),
\end{equation*}
and for $\lambda_1>\sqrt{3}$, we have $G(\hat{\mathbf{d}})<0$. Thus, the quadratic form $\frac{\text{D}^{\text{ZJ}}}{\text{D}t}[\boldsymbol{\sigma}]:\mathbf{d}$ is generally neither positive definite nor even semi-positive definite. Hence any mixed neo-Hookean model does not satisfy the CSP.

The fact that the vol-iso neo-Hookean models do not satisfy the CSP has been shown by Neff et al. (cf., \cite{NeffJMPS2025}, Appendix A.4).\\

\begin{remark}
\label{rem:5-6}
While no model of our neo-Hookean material compressible family does satisfy the $\text{CSP}\ \Leftrightarrow \text{TSTS-M}^+$ condition, it is known that the so-called exponential Hencky energy \cite{NeffJElast2015,NeffJMPS2025}
\begin{equation}\label{146-1}
  W_{\text{exp-Hencky}}(\mathbf{F}) = \mu\, (\mathrm{e}^{\|\log \mathbf{V}\|^2}-1) + \frac{\lambda}{2}\,(\mathrm{e}^{\log^2\!\!J}-1)
\end{equation}
satisfies $\text{TSTS-M}^+$ for $\mu,\,\lambda>0$. However, \eqref{146-1} is not polyconvex or LH-elliptic throughout. The last author has offered a 500 euro challenge for finding a compressible isotropic energy that satisfies simultaneously $\text{TSTS-M}^+$ and polyconvexity \cite{NeffRG2025}.
\end{remark}

\section{Testing neo-Hookean materials in homogeneous deformations}
\label{sec:6}

In this section, general forms of expressions for stress and strain/deformation tensors for all types of homogeneous deformations considered here are presented in Section \ref{sec:6-1}. The dependencies of kinematic and static quantities on prescribed stretches for uniaxial loading, equibiaxial loading in plane stress, and uniaxial loading in plain strain are given in Sections \ref{sec:6-2}, \ref{sec:6-3}, and \ref{sec:6-4}, respectively.

\subsection{Forms of stress and strain/deformation tensors for homogeneous deformations}
\label{sec:6-1}

Following Kossa et al. \cite{KossaMeccanica2023}, we consider three types of homogeneous deformation (Fig.~\ref{f3}): \emph{uniaxial loading (UL)} (Fig. \ref{f3},\emph{a}), \emph{equibiaxial loading in plane stress (ELP)} (Fig. \ref{f3},\emph{b}), and \emph{uniaxial loading in plane strain (ULP)} (Fig. \ref{f3},\emph{c}).
\begin{figure}
\begin{center}
\includegraphics{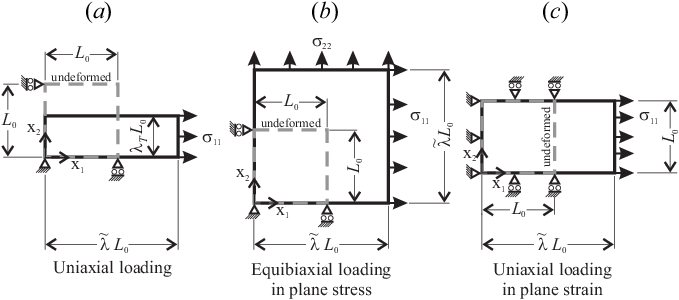}
\end{center}
\caption{Sketches of homogeneous deformations: uniaxial loading (\emph{a}), equibiaxial loading in plane stress (\emph{b}), and uniaxial loading in plane strain (\emph{c}).}
\label{f3}
\end{figure}
In these three cases, the tensors $\mathbf{F}$, $\mathbf{c}$, $\boldsymbol{\sigma}$, and $\mathbf{P}$ have the following forms:
\begin{align}\label{147}
  \mathbf{F}=& \left[
               \begin{array}{ccc}
                 \lambda_1 & 0         & 0         \\
                 0         & \lambda_2 & 0         \\
                 0         & 0         & \lambda_3 \\
               \end{array}
             \right],\quad\quad
  \mathbf{c}=\left[
               \begin{array}{ccc}
                 \lambda_1^2 & 0           & 0           \\
                 0           & \lambda_2^2 & 0           \\
                 0           & 0           & \lambda_3^2 \\
               \end{array}
             \right],      \\
  \boldsymbol{\sigma}=& \left[
               \begin{array}{ccc}
                 \sigma_{11} & 0           & 0           \\
                 0           & \sigma_{22} & 0           \\
                 0           & 0           & \sigma_{33} \\
               \end{array}
             \right],\quad\quad
 \mathbf{P}=\left[
               \begin{array}{ccc}
                 P_{11} & 0      & 0      \\
                 0      & P_{22} & 0      \\
                 0      & 0      & P_{33} \\
               \end{array}
             \right],              \notag
\end{align}
and the volume ratio $J$ has the form \eqref{20}. The Finger strain tensor $\mathbf{e}^{(2)}$ is written as follows (see \eqref{18}):
\begin{equation}\label{148}
  \mathbf{e}^{(2)}= \frac{1}{2}\left[
               \begin{array}{ccc}
                 \lambda_1^2-1 & 0             & 0             \\
                 0             & \lambda_2^2-1 & 0             \\
                 0             & 0             & \lambda_3^2-1 \\
               \end{array}
             \right].
\end{equation}
Using Eqs. \eqref{24} and $\eqref{26}_2$, we obtain the following expression for the tensor $\text{dev}\,\bar{\mathbf{c}}$:
\begin{equation}\label{149}
  \text{dev}\,\bar{\mathbf{c}}=\frac{1}{3}J^{-2/3}\left[
  \begin{array}{ccc}
                 2\lambda_1^2-\lambda_2^2-\lambda_3^2 & 0                                      & 0                                      \\
                 0                                    & -\lambda_1^2 +2\lambda_2^2-\lambda_3^2 & 0                                      \\
                 0                                    & 0                                      & -\lambda_1^2 -\lambda_2^2+2\lambda_3^2 \\
               \end{array}
             \right].
\end{equation}

We now specify the expressions for the quantities $J$, $\mathbf{e}^{(2)}$, and $\text{dev}\,\bar{\mathbf{c}}$ for the homogeneous deformations considered (hereinafter, the stretch $\tilde{\lambda}$ is assumed to be a prescribed quantity).

$\bullet$ For UL (Fig. \ref{f3},\emph{a}), the stress-strain state is subjected to the following constraints:
\begin{equation}\label{150}
  \lambda_1=\tilde{\lambda},\quad\quad \sigma_{22}=\sigma_{33}=0,\quad\quad P_{22}=P_{33}=0.
\end{equation}
Transverse strains in UL will be denoted by
\begin{equation}\label{151}
  \lambda_T\equiv \lambda_2=\lambda_3.
\end{equation}
The volume ratio $J$ is determined from \eqref{20}, $\eqref{150}_1$, and \eqref{151}
\begin{equation}\label{152}
  J= \tilde{\lambda} \lambda_T^2.
\end{equation}
For UL, the Finger strain tensor $\mathbf{e}^{(2)}$ defined in \eqref{148} has the form
\begin{equation*}
  \mathbf{e}^{(2)}= \frac{1}{2}\left[
               \begin{array}{ccc}
                 \tilde{\lambda}^2-1 & 0             & 0             \\
                 0                   & \lambda_T^2-1 & 0             \\
                 0                   & 0             & \lambda_T^2-1 \\
               \end{array}
             \right],
\end{equation*}
and in view of expressions \eqref{149}, $\eqref{150}_1$, and \eqref{151}, the tensor $\text{dev}\,\bar{\mathbf{c}}$ can be written as
\begin{equation}\label{154}
  \text{dev}\,\bar{\mathbf{c}}=\frac{1}{3}J^{-2/3}\left[
  \begin{array}{ccc}
                 2(\tilde{\lambda}^2-\lambda_T^2) & 0                             & 0                             \\
                 0                                & \lambda_T^2-\tilde{\lambda}^2 & 0                             \\
                 0                                & 0                             & \lambda_T^2-\tilde{\lambda}^2 \\
               \end{array}
             \right].
\end{equation}
Using \eqref{33}, \eqref{147}, and \eqref{150}, we obtain the following expression for the quantity $P_{11}$:
\begin{equation}\label{155}
  P_{11}=\lambda_T^2 \sigma_{11}.
\end{equation}

$\bullet$ For ELP (Fig. \ref{f3},\emph{b}), the stress-strain state is subjected to the following constraints:
\begin{equation}\label{156}
  \lambda_1=\lambda_2=\tilde{\lambda},\quad\quad \sigma_{33}=0,\quad\quad \sigma_{11}=\sigma_{22},\quad\quad P_{11}=P_{22},\quad P_{33}=0.
\end{equation}
Transverse strains in ELP will be denoted by
\begin{equation}\label{157}
  \lambda_T\equiv \lambda_3.
\end{equation}
The volume ratio $J$ is determined from \eqref{20}, $\eqref{156}_1$, and \eqref{157}:
\begin{equation}\label{158}
  J= \tilde{\lambda}^2 \lambda_T.
\end{equation}
For ELP, the Finger strain tensor $\mathbf{e}^{(2)}$ defined in \eqref{148} has the form
\begin{equation}\label{159}
  \mathbf{e}^{(2)}= \frac{1}{2}\left[
               \begin{array}{ccc}
                 \tilde{\lambda}^2-1 & 0                   & 0             \\
                 0                   & \tilde{\lambda}^2-1 & 0             \\
                 0                   & 0                   & \lambda_T^2-1 \\
               \end{array}
             \right],
\end{equation}
and in view of expressions \eqref{149}, $\eqref{156}_1$, and \eqref{157}, the tensor $\text{dev}\,\bar{\mathbf{c}}$ can be written as
\begin{equation}\label{160}
  \text{dev}\,\bar{\mathbf{c}}=\frac{1}{3}J^{-2/3}\left[
  \begin{array}{ccc}
                 \tilde{\lambda}^2-\lambda_T^2 & 0                             & 0                                \\
                 0                             & \tilde{\lambda}^2-\lambda_T^2 & 0                                \\
                 0                             & 0                             & 2(\lambda_T^2-\tilde{\lambda}^2) \\
               \end{array}
             \right].
\end{equation}
In view of \eqref{33}, \eqref{147}, and \eqref{158}, the nonzero components of the first P-K stress tensor can be written as
\begin{equation}\label{161}
  P_{11}= P_{22}= \tilde{\lambda}\lambda_T\sigma_{11}.
\end{equation}

$\bullet$ For ULP (Fig. \ref{f3},\emph{c}), the stress-strain state is subjected to the following constraints:
\begin{equation}\label{162}
  \lambda_1=\tilde{\lambda},\quad \lambda_2=1,\quad \sigma_{33}=0,\quad P_{33}=0.
\end{equation}
Transverse strains in ULP will be denoted by
\begin{equation}\label{163}
  \lambda_T\equiv \lambda_3.
\end{equation}
The volume ratio $J$ is determined from \eqref{20}, $\eqref{162}_{1,2}$, and \eqref{163}:
\begin{equation}\label{164}
  J= \tilde{\lambda}\lambda_T.
\end{equation}
For ULP, the Finger strain tensor $\mathbf{e}^{(2)}$ defined in \eqref{148} has the form
\begin{equation}\label{165}
  \mathbf{e}^{(2)}= \frac{1}{2}\left[
               \begin{array}{ccc}
                 \tilde{\lambda}^2-1 & 0    & 0             \\
                 0                   & 0    & 0             \\
                 0                   & 0    & \lambda_T^2-1 \\
               \end{array}
             \right],
\end{equation}
and in view of expressions \eqref{149}, $\eqref{162}_{1,2}$, and \eqref{163}, the tensor $\text{dev}\,\bar{\mathbf{c}}$ can be written as
\begin{equation}\label{166}
  \text{dev}\,\bar{\mathbf{c}}=\frac{1}{3}J^{-2/3}\left[
  \begin{array}{ccc}
                 2\tilde{\lambda}^2 -1 -\lambda_T^2 & 0                                  & 0                                  \\
                 0                                  & -\tilde{\lambda}^2 +2 -\lambda_T^2 & 0                                  \\
                 0                                  & 0                                  & 2\lambda_T^2 -1 -\tilde{\lambda}^2 \\
               \end{array}
             \right].
\end{equation}
In view of \eqref{33}, \eqref{147}, and \eqref{164}, the nonzero components of the first P-K stress tensor can be written as
\begin{equation}\label{167}
  P_{11}= \lambda_T\sigma_{11},\  P_{22}=\tilde{\lambda}\lambda_T \sigma_{22}.
\end{equation}

Kossa et al. \cite{KossaMeccanica2023} used the vol-iso material model with volumetric energy given by the expression $\frac{K}{2}(J^2-1)$ (i.e., using volumetric function \#7) and varied Poisson's ratio in the range $-1\leq \nu \leq 0.5$. Since the bulk modulus $K$ is non-negative in this range of Poisson's ratio and since the functions $h(J)$ used in the present work are non-negative (see Section \ref{sec:4}), it follows that the volumetric energy is also non-negative. For the mixed models, the non-negativity of the volumetric energy $\lambda h(J)$ is determined by the non-negativity of the parameter $\lambda$. Since the non-negativity of this parameter is guaranteed by the range of Poisson's ratio $0\leq \nu \leq 0.5$, in order for the volumetric energies to be non-negative for both material models (see, e.g., \cite{HartmannIJSS2003}), here we restrict ourselves to the following set of values of Poisson's ratio $\nu$  for compressible neo-Hookean materials:
\begin{equation}\label{168}
  \nu = \{0,\,0.25,\,0.4,\,0.45,\,0.499,\,0.4999\}.
\end{equation}
The value $\nu=0.5$ is assigned to the classical incompressible neo-Hookean model.

\subsection{Uniaxial loading}
\label{sec:6-2}

The dependencies of stresses and unknown lateral principal stretches on the prescribed longitudinal stretch obtained by solving the uniaxial loading problem for the incompressible isotropic neo-Hookean material model and mixed and vol-iso compressible isotropic neo-Hookean material models are presented in Sections \ref{sec:6-2-1}, \ref{sec:6-2-2}, and \ref{sec:6-2-3}, respectively.

\subsubsection{Incompressible isotropic neo-Hookean material}
\label{sec:6-2-1}

Setting $J=1$, from \eqref{152} we obtain
\begin{equation}\label{169}
  \lambda_T= \tilde{\lambda}^{-1/2}.
\end{equation}
In view of \eqref{55} and \eqref{148}, the components of the Cauchy stress tensor can be written as
\begin{equation}\label{170}
  \sigma_{11} = \mu\,(\tilde{\lambda}^2-1)-p,\quad\quad \sigma_{22} = \sigma_{33} = \mu\,(\lambda_T^2-1)-p.
\end{equation}
Determining the Lagrange multiplier $p$ from $\eqref{150}_2$ and $\eqref{170}_2$ and using \eqref{169}, we obtain
\begin{equation}\label{171}
  \sigma_{11} = \mu\,(\tilde{\lambda}^2-\tilde{\lambda}^{-1}).
\end{equation}
In view of \eqref{169} and \eqref{171}, from \eqref{155} we get
\begin{equation}\label{172}
   P_{11}= \tilde{\lambda}^{-1} \sigma_{11} = \mu\,(\tilde{\lambda}-\tilde{\lambda}^{-2}).
\end{equation}

The limiting values of $\lambda_T$, $\sigma_{11}$, and $P_{11}$ in extreme states are obtained from expressions \eqref{169}, \eqref{171}, and \eqref{172} and presented in Table \ref{t2}.
\begin{table}
\caption{Limiting values of $\lambda_T$, $\sigma_{11}$, and $P_{11}$ in extreme states where $\tilde{\lambda}\rightarrow 0$ and $\tilde{\lambda}\rightarrow \infty$ in the UL and ELP problems for the incompressible isotropic neo-Hookean material model}
\label{t2}
\begin{tabular}{lll}
\hline\noalign{\smallskip}
 Quantity                       & $\tilde{\lambda}\rightarrow 0$ & $\tilde{\lambda}\rightarrow \infty$ \\
\noalign{\smallskip}\hline\noalign{\smallskip}
 $\lambda_T(\tilde{\lambda})$   & $+\infty$                      &  0                                  \\
 $\sigma_{11}(\tilde{\lambda})$ & $-\infty$                      &  $+\infty$                          \\
 $P_{11}(\tilde{\lambda})$      & $-\infty$                      &  $+\infty$                          \\
\noalign{\smallskip}\hline
\end{tabular}
\end{table}
We assume that these limiting values correspond to physically reasonable responses for idealized hyperelastic materials.

\subsubsection{Compressible isotropic mixed neo-Hookean material models}
\label{sec:6-2-2}

In view of \eqref{65} and \eqref{148}, the components of the Cauchy stress tensor can be written as
\begin{equation}\label{173}
  \sigma_{11} = \lambda\, h^{\prime}(J) + \frac{\mu}{J} (\tilde{\lambda}^2-1),\quad\quad
  \sigma_{22} = \sigma_{33} = \lambda\, h^{\prime}(J) + \frac{\mu}{J} (\lambda_T^2-1).
\end{equation}
Using $\eqref{150}_2$ and \eqref{152}, we obtain the nonlinear implicit dependence of $\lambda_T$ on $\tilde{\lambda}$ in the general case:
\begin{equation}\label{174}
  \lambda\, h^{\prime}(\tilde{\lambda}\lambda_T^2) + \frac{\mu}{\tilde{\lambda}} (1-\lambda_T^{-2})=0.
\end{equation}

We first consider the value $\nu=0$ for Poisson's ratio. Since for this value of $\nu$, $\lambda=0$ (see $\eqref{72}_2$), from \eqref{174} we obtain
\begin{equation}\label{175}
  \lambda_T=1.
\end{equation}
Then from \eqref{152}, \eqref{155} and $\eqref{173}_1$, we get
\begin{equation}\label{176}
   \sigma_{11}=P_{11}=  \mu\,(\tilde{\lambda}-\tilde{\lambda}^{-1}).
\end{equation}
The values of $\lambda_T$ in \eqref{175} and $\sigma_{11}$ and $P_{11}$ in \eqref{176} are valid for any functions $h^{\prime}(J)$. The fact that for $\nu=0$, there is no lateral deformation for any value $\tilde{\lambda}$ is considered a physically reasonable result (see, e.g., \cite{KellermannZAMM2016}) by analogy with deformation under UL when using the equations of linear elasticity theory.

For the remaining values of $\nu$ from the interval $0<\nu<0.5$, the value of $\lambda_T$ should be determined from the nonlinear equation \eqref{174}. In the particular case of mixed model \#7, the function $h^{\prime}(J)$ has the form $\eqref{93}_2$. In this case, Eq. \eqref{174} has an explicit solution (cf., \cite{PenceMMS2015}, Eq. (4.25))
\begin{equation*}
  \lambda_T=\left[\frac{1}{2a}(-b + \sqrt{b^2-4ac})\right]^{1/2},
\end{equation*}
where
\begin{equation*}
  a\equiv \lambda\tilde{\lambda},\quad\quad b\equiv \frac{\mu}{\tilde{\lambda}},\quad\quad c\equiv -\frac{\mu}{\tilde{\lambda}}.
\end{equation*}
For the remaining volumetric functions $h^{\prime}(J)$ considered in this book, the dependence $\lambda_T(\tilde{\lambda})$ is determined from \eqref{174} using the Wolfram Mathematica software. Substitution of the obtained dependence into $\eqref{173}_1$ with the use of  expression \eqref{152} yields the dependence $\sigma_{11}(\tilde{\lambda})$. Using the dependencies $\lambda_T(\tilde{\lambda})$ and $\sigma_{11}(\tilde{\lambda})$, from \eqref{155} we find the dependence  $P_{11}(\tilde{\lambda})$.

Plots of $\lambda_T$ versus $\tilde{\lambda}$ for mixed models \#1--4 are shown  in Fig.~\ref{f4} along  with dependencies \eqref{169} for the incompressible neo-Hookean material model ($\nu=0.5$).
\begin{figure}
\begin{center}
\includegraphics{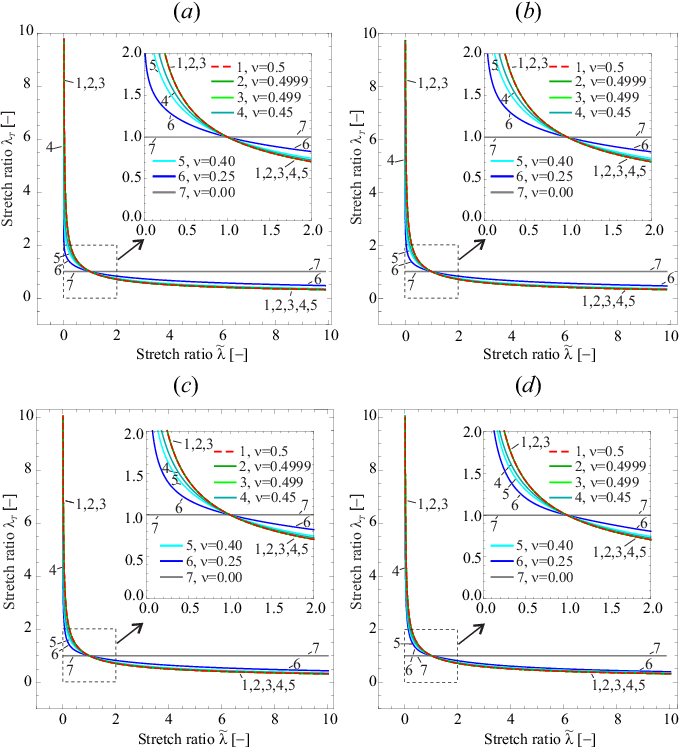}
\end{center}
\caption{Plots of $\lambda_T$ versus $\tilde{\lambda}$ in the UL problem for mixed models \#1 (\emph{a}), \#2 (\emph{b}), \#3 (\emph{c}), and \#4 (\emph{d}).}
\label{f4}
\end{figure}
From the plots and an analysis of the solutions using the Wolfram Mathematica software, it follows that the limiting values of $\lambda_T$ for these material models with $\nu\neq 0$ (see Table~\ref{t3}) agree with the physically reasonable limiting values in Table \ref{t2}.
\begin{table}
\caption{Limiting values of $\lambda_T$, $\sigma_{11}$, and $P_{11}$ in extreme states where $\tilde{\lambda}\rightarrow 0$ and $\tilde{\lambda}\rightarrow \infty$ in the UL problem for compressible isotropic material models with $0\leq\nu< 0.5$}
\label{t3}
\begin{tabular}{llllll}
\hline\noalign{\smallskip}
Model    &                                & \multicolumn{2}{c}{Mixed models}                                    & \multicolumn{2}{c}{Vol-iso models} \\
ID       & Quantity${}^{\flat}$           & $\tilde{\lambda}\rightarrow 0$ & $\tilde{\lambda}\rightarrow \infty$ & $\tilde{\lambda}\rightarrow 0$ & $\tilde{\lambda}\rightarrow \infty$ \\
\noalign{\smallskip}\hline\noalign{\smallskip}
         & $\lambda_T(\tilde{\lambda})$   & $\textcolor{green}{+\infty}$   &  $\textcolor{green}{0}$             &
                                           \textcolor{red}{0}              &  $\textcolor{red}{+\infty}$         \\
   1     & $\sigma_{11}(\tilde{\lambda})$ & $\textcolor{green}{-\infty}$   &  $\textcolor{green}{+\infty}$       &
                                            $\textcolor{green}{-\infty}$   &  \textcolor{red}{0}                 \\
         & $P_{11}(\tilde{\lambda})$      & $\textcolor{green}{-\infty}$   &  $\textcolor{green}{+\infty}$       &
                                            $\textcolor{green}{-\infty}$   &  \textcolor{red}{0}                 \\
\noalign{\smallskip}\hline\noalign{\smallskip}
         & $\lambda_T(\tilde{\lambda})$   & $\textcolor{green}{+\infty}$   &  $\textcolor{green}{0}$             &
                                            $\textcolor{green}{+\infty}$   &  $\textcolor{red}{+\infty}$         \\
   2     & $\sigma_{11}(\tilde{\lambda})$ & $\textcolor{green}{-\infty}$   &  $\textcolor{green}{+\infty}$       &
                                            $\textcolor{green}{-\infty}$   &  $\textcolor{red}{\ast}$            \\
         & $P_{11}(\tilde{\lambda})$      & $\textcolor{green}{-\infty}$   &  $\textcolor{green}{+\infty}$       &
                                            $\textcolor{green}{-\infty}$   &  $\textcolor{green}{+\infty}$       \\
\noalign{\smallskip}\hline\noalign{\smallskip}
         & $\lambda_T(\tilde{\lambda})$   & $\textcolor{green}{+\infty}$   &  $\textcolor{green}{0}$             &
                                            $\textcolor{green}{+\infty}$   &  $\textcolor{green}{0}$            \\
   3     & $\sigma_{11}(\tilde{\lambda})$ & $\textcolor{green}{-\infty}$   &  $\textcolor{green}{+\infty}$       &
                                            $\textcolor{green}{-\infty}$   &  $\textcolor{green}{+\infty}$       \\
         & $P_{11}(\tilde{\lambda})$      & $\textcolor{green}{-\infty}$   &  $\textcolor{green}{+\infty}$       &
                                            $\textcolor{green}{-\infty}$   &  $\textcolor{green}{+\infty}$       \\
\noalign{\smallskip}\hline\noalign{\smallskip}
         & $\lambda_T(\tilde{\lambda})$   & $\textcolor{green}{+\infty}$   &  $\textcolor{green}{0}$             &
                                            $\textcolor{green}{+\infty}$   &  $\textcolor{green}{0}$            \\
   4     & $\sigma_{11}(\tilde{\lambda})$ & $\textcolor{green}{-\infty}$   &  $\textcolor{green}{+\infty}$       &
                                            $\textcolor{green}{-\infty}$   &  $\textcolor{green}{+\infty}$       \\
         & $P_{11}(\tilde{\lambda})$      & $\textcolor{green}{-\infty}$   &  $\textcolor{green}{+\infty}$       &
                                            $\textcolor{green}{-\infty}$   &  $\textcolor{green}{+\infty}$       \\
\noalign{\smallskip}\hline\noalign{\smallskip}
         & $\lambda_T(\tilde{\lambda})$   & $\textcolor{red}{\ast}$        &  $\textcolor{green}{0}$             &
                                            \textcolor{red}{0}             &  $\textcolor{green}{0}$             \\
   5     & $\sigma_{11}(\tilde{\lambda})$ & $\textcolor{green}{-\infty}$   &  $\textcolor{green}{+\infty}$       &
                                            $\textcolor{green}{-\infty}$   &  $\textcolor{green}{+\infty}$       \\
         & $P_{11}(\tilde{\lambda})$      & $\textcolor{green}{-\infty}$   &  $\textcolor{green}{+\infty}$       &
                                            $\textcolor{green}{-\infty}$   &  $\textcolor{green}{+\infty}$       \\
\noalign{\smallskip}\hline\noalign{\smallskip}
         & $\lambda_T(\tilde{\lambda})$   & $\textcolor{red}{\ast}$        &  $\textcolor{green}{0}$             &
                                            \textcolor{red}{0}             &  $\textcolor{red}{+\infty}$         \\
   6     & $\sigma_{11}(\tilde{\lambda})$ & $\textcolor{green}{-\infty}$   &  $\textcolor{green}{+\infty}$       &
                                            $\textcolor{green}{-\infty}$   &  $\textcolor{green}{+\infty}$       \\
         & $P_{11}(\tilde{\lambda})$      & $\textcolor{green}{-\infty}$   &  $\textcolor{green}{+\infty}$       &
                                            $\textcolor{green}{-\infty}$   &  $\textcolor{green}{+\infty}$       \\
\noalign{\smallskip}\hline\noalign{\smallskip}
         & $\lambda_T(\tilde{\lambda})$   & \textcolor{red}{1}             &  $\textcolor{green}{0}$             &
                                            \textcolor{red}{0}             &  $\textcolor{green}{0}$             \\
   7     & $\sigma_{11}(\tilde{\lambda})$ & $\textcolor{green}{-\infty}$   &  $\textcolor{green}{+\infty}$       &
                                            $\textcolor{red}{-3K}$         &  $\textcolor{green}{+\infty}$       \\
         & $P_{11}(\tilde{\lambda})$      & $\textcolor{green}{-\infty}$   &  $\textcolor{green}{+\infty}$       &
                                            \textcolor{red}{0}             &  $\textcolor{green}{+\infty}$       \\
\noalign{\smallskip}\hline\noalign{\smallskip}
         & $\lambda_T(\tilde{\lambda})$   & $\textcolor{green}{+\infty}$   &  $\textcolor{green}{0}$             &
                                            $\textcolor{green}{+\infty}$   &  $\textcolor{green}{0}$             \\
   8     & $\sigma_{11}(\tilde{\lambda})$ & $\textcolor{green}{-\infty}$   &  $\textcolor{green}{+\infty}$       &
                                            $\textcolor{green}{-\infty}$   &  $\textcolor{green}{+\infty}$       \\
         & $P_{11}(\tilde{\lambda})$      & $\textcolor{green}{-\infty}$   &  $\textcolor{green}{+\infty}$       &
                                            $\textcolor{green}{-\infty}$   &  $\textcolor{green}{+\infty}$       \\
\noalign{\smallskip}\hline
\end{tabular}
\footnotesize
\begin{itemize}
\item[${}^{\flat}$]An asterisk ($\ast$) denotes some finite limiting values, and green and red colors indicate physically reasonable and unreasonable values of a quantity, respectively.
\end{itemize}
\normalsize
\end{table}
Note also that the obtained plots for mixed model \#1 agree with the plots obtained for the same material model by Ehlers and Eipper (cf., \cite{EhlersAM1998}, Fig. 3, curve (a)).

Plots of $\lambda_T$ versus $\tilde{\lambda}$ for mixed models \#5--8 are presented in Fig.~\ref{f5}.
\begin{figure}
\begin{center}
\includegraphics{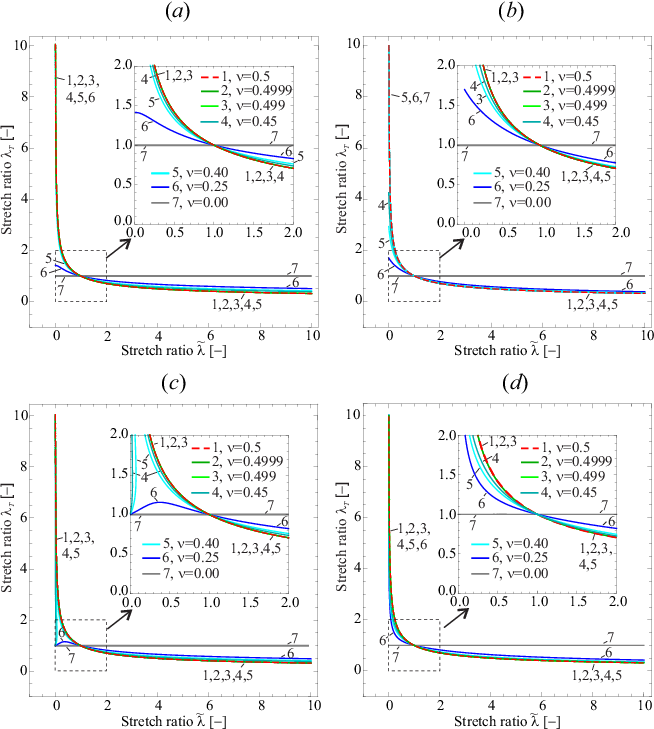}
\end{center}
\caption{Plots of $\lambda_T$ versus $\tilde{\lambda}$ in the UL problem for mixed models \#5 (\emph{a}), \#6 (\emph{b}), \#7 (\emph{c}), and \#8 (\emph{d}).}
\label{f5}
\end{figure}
Analysis of the solutions using the Wolfram Mathematica software shows that for $\lambda_T\rightarrow 0$, the limiting values of $\lambda_T$ for material models \#5-7 with $\nu\neq 0$ do not agree with the physically reasonable limiting values in Table \ref{t2}. In particular, using \eqref{169}, for mixed model \#7 we obtain the same value $\lim \limits_{\tilde{\lambda} \to 0} \lambda_T = 1$ for all values of Poisson's ratio $0<\nu <0.5$ (see also \cite{PenceMMS2015}, Eq. $(4.26)_1$). The dependencies $\lambda_T(\tilde{\lambda})$ obtained for model \#7 (see Fig.~\ref{f5},\emph{c}) agree with the dependencies derived by Pence and Gou (cf., \cite{PenceMMS2015}, Fig. 6). For mixed models \#5 and \#6, different limiting values of $\lambda_T$ for $\tilde{\lambda}\rightarrow 0$ are obtained using different values of Poisson's ratio from the set \eqref{168} (see Table~\ref{t3}).

Plots of the Cauchy stress $\sigma_{11}$ versus the stretch $\tilde{\lambda}$ for mixed material models \#1--4 and \#5--8 are shown in Figs.~\ref{f6} and~\ref{f7}, respectively.
\begin{figure}
\begin{center}
\includegraphics{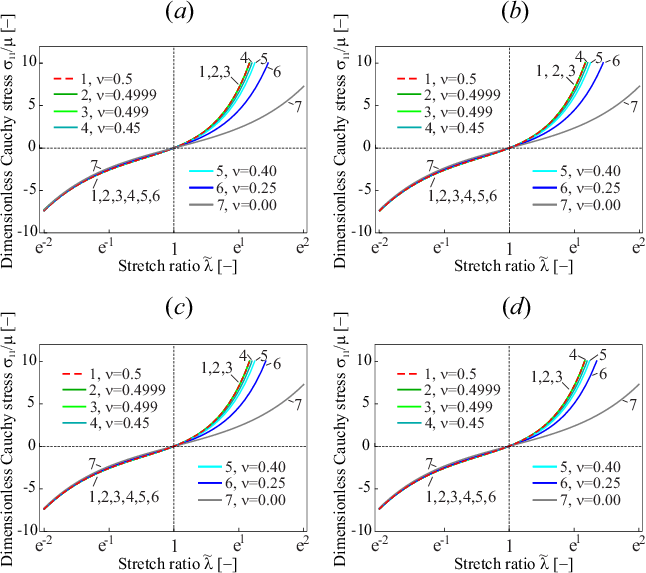}
\end{center}
\caption{Plots of $\sigma_{11}$ versus $\tilde{\lambda}$ in the UL problem for mixed models \#1 (\emph{a}), \#2 (\emph{b}), \#3 (\emph{c}), and \#4 (\emph{d}).}
\label{f6}
\end{figure}
\begin{figure}
\begin{center}
\includegraphics{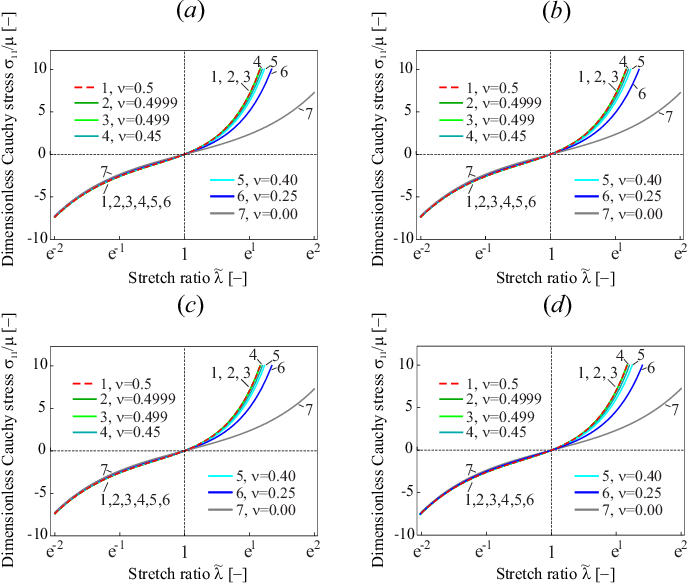}
\end{center}
\caption{Plots of $\sigma_{11}$ versus $\tilde{\lambda}$ in the UL problem for mixed models \#5 (\emph{a}), \#6 (\emph{b}), \#7 (\emph{c}), and \#8 (\emph{d}).}
\label{f7}
\end{figure}
Similar plots of the engineering (1st P-K, nominal) stress $P_{11}$ versus the stretch $\tilde{\lambda}$ for mixed material models \#1--4 and \#5--8 are shown in Figs.~\ref{f8} and~\ref{f9}, respectively.
\begin{figure}
\begin{center}
\includegraphics{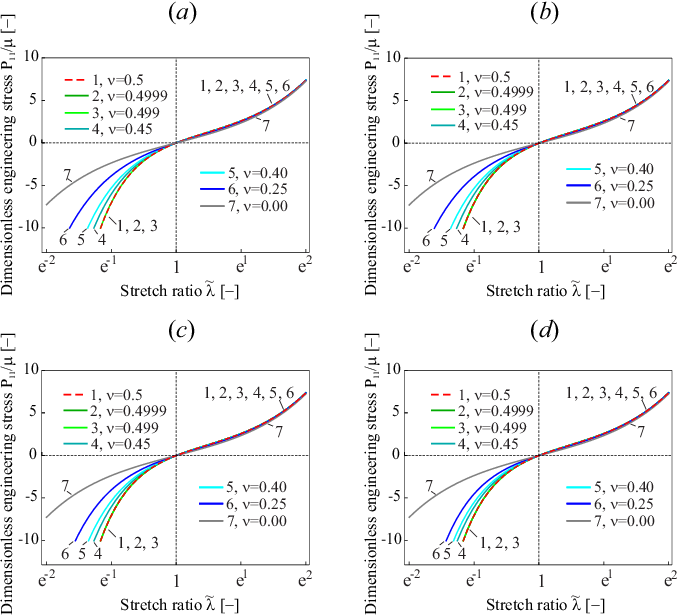}
\end{center}
\caption{Plots of $P_{11}$ versus $\tilde{\lambda}$ in the UL problem for mixed models \#1 (\emph{a}), \#2 (\emph{b}), \#3 (\emph{c}), and \#4 (\emph{d}).}
\label{f8}
\end{figure}
\begin{figure}
\begin{center}
\includegraphics{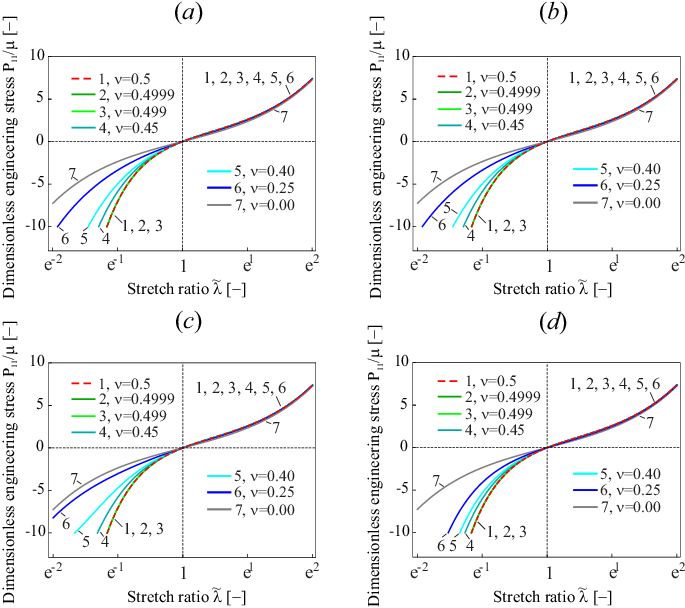}
\end{center}
\caption{Plots of $P_{11}$ versus $\tilde{\lambda}$ in the UL problem for mixed models \#5 (\emph{a}), \#6 (\emph{b}), \#7 (\emph{c}), and \#8 (\emph{d}).}
\label{f9}
\end{figure}
We see that the stresses obtained for all compressible material models considered in this book under slight compressibility conditions (with $\nu=0.499$ and $\nu=0.4999$) are close to the stresses for the incompressible nH model (with $\nu=0.5$). Note that the plots of $\sigma_{11}(\tilde{\lambda})$ obtained for mixed model \#7 (see Fig.~\ref{f7},\emph{c}) agree with the plots obtained by Pence and Gou for the same material model (cf., \cite{PenceMMS2015}, Fig. 7). We used the Wolfram Mathematica software to determine the limiting values of $\sigma_{11}$ and $P_{11}$ in extreme states and obtained the limiting values of these quantities (see Table~\ref{t3}) that agree with the physically reasonable values  presented in Table ~\ref{t2}.

The solution of the UL problem for the mixed models leads to the conclusion that the most reliable physically reasonable solutions can be obtained using volumetric functions from the Hartmann--Neff family and volumetric function \#8. First, for values of Poisson's ratio $\nu=0$, there is no lateral strength. Second, for values of Poisson's ratio $\nu\neq 0$ for material models with these volumetric functions in extreme states, the lateral stretch $\lambda_T$ and the stresses $\sigma_{11}$ and $P_{11}$ have physically reasonable limiting values that agree with the limiting values of these quantities for incompressible materials. Third, the solutions for $\lambda_T$, $\sigma_{11}$, and $P_{11}$ for mixed models with $\nu=0.499$ and $\nu=0.4999$ are close to the corresponding solutions for the classical incompressible nH material model.

\subsubsection{Compressible isotropic vol-iso neo-Hookean material models}
\label{sec:6-2-3}

In view of \eqref{68} and \eqref{154}, the components of the Cauchy stress tensor can be written as
\begin{align}\label{182}
  \sigma_{11} &= K\, h^{\prime}(J) + \frac{2}{3}\mu\, J^{-5/3} (\tilde{\lambda}^2-\lambda_T^2), \\
  \sigma_{22} &= \sigma_{33} = K\, h^{\prime}(J) + \frac{1}{3} \mu\, J^{-5/3} (\lambda_T^2-\tilde{\lambda}^2). \notag
\end{align}
Regardless of the choice of Poisson's ratio $\nu \in [0,0.5)$, Eqs. $\eqref{150}_2$, \eqref{152}, and $\eqref{182}_2$ lead to the following nonlinear equation for the dependence  $\lambda_T$ vs. $\tilde{\lambda}$:
\begin{equation*}
  K\, h^{\prime}(\tilde{\lambda}\lambda_T^2) + \frac{1}{3} \mu\, J^{-5/3} (\lambda_T^2-\tilde{\lambda}^2)=0.
\end{equation*}
Summing the left and right sides of equalities in \eqref{182} and taking into account equalities $\sigma_{22}=\sigma_{33}=0$, we obtain
\begin{equation}\label{184}
  \sigma_{11} = 3K\, h^{\prime}(\tilde{\lambda}\lambda_T^2).
\end{equation}
Substitution of $\lambda_T(\tilde{\lambda})$ into the r.h.s. of \eqref{184} yields the dependence $\sigma_{11}(\tilde{\lambda})$. The dependence $P_{11}(\tilde{\lambda})$ is obtained from \eqref{155} using the dependencies $\lambda_T(\tilde{\lambda})$ and $\sigma_{11}(\tilde{\lambda})$.

Plots of $\lambda_T$ versus $\tilde{\lambda}$ for vol-iso models \#1--4, and \#5--8 are presented in Figs.~\ref{f10} and \ref{f11}, respectively, along with dependencies \eqref{169} for incompressible nH material ($\nu=0.5$).
\begin{figure}
\begin{center}
\includegraphics{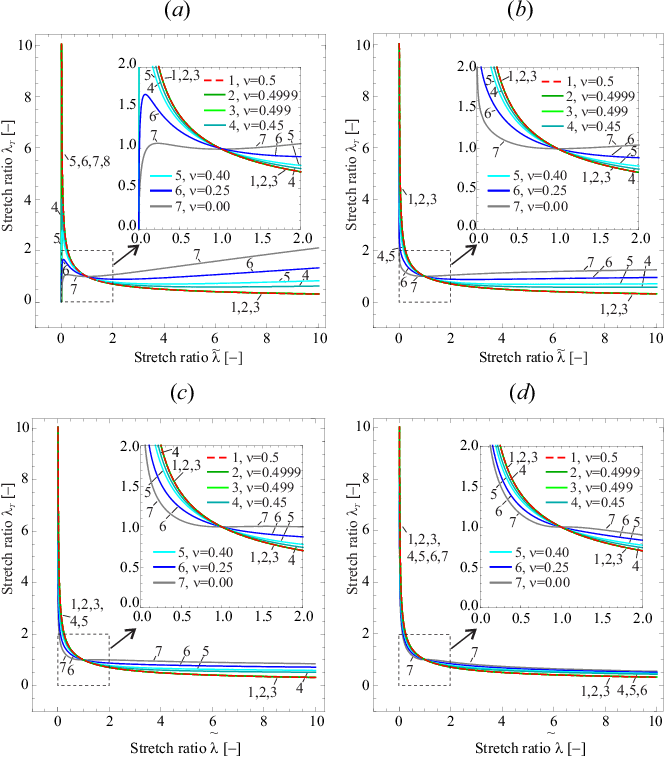}
\end{center}
\caption{Plots of $\lambda_T$ versus $\tilde{\lambda}$ in the UL problem for vol-iso models \#1 (\emph{a}), \#2 (\emph{b}), \#3 (\emph{c}), and \#4 (\emph{d}).}
\label{f10}
\end{figure}
\begin{figure}
\begin{center}
\includegraphics{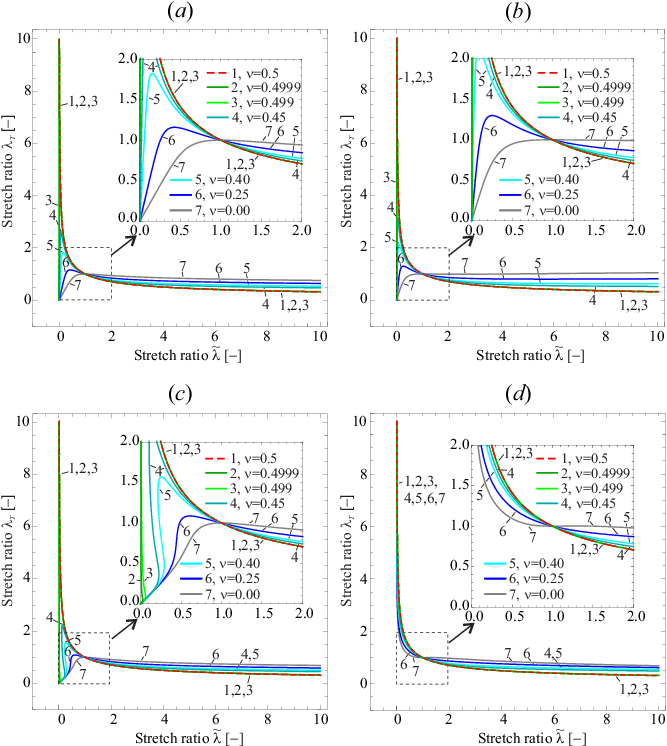}
\end{center}
\caption{Plots of $\lambda_T$ versus $\tilde{\lambda}$ in the UL problem for vol-iso models \#5 (\emph{a}), \#6 (\emph{b}), \#7 (\emph{c}), and \#8 (\emph{d}).}
\label{f11}
\end{figure}
Analysis of the solutions using the Wolfram Mathematica software shows that the limiting values of $\lambda_T$ in extreme states for vol-iso material models with $0\leq\nu< 0.5$ agree with the physically reasonable limiting values in Table \ref{t2} only for the models \#3,4,8 (see Table~\ref{t3}). For the remaining models (\#1,2,5,6,7), data on the limiting physically unreasonable values of $\lambda_T$ are presented in Table \ref{t3}. Note that the plots of $\lambda_T(\tilde{\lambda})$ in Fig.~\ref{f10},\emph{a} for vol-iso model \#1 agree with the plot (a) in Fig. 2 in \cite{EhlersAM1998} for the same material model. The plots for vol-iso model \#3 (see Fig. \ref{f10},\emph{c}) agree with the plots for the same material model in Fig. 9 in \cite{PenceMMS2015}. In addition, the plots of $\lambda_T(\tilde{\lambda})$ in Fig. \ref{f11},\emph{c} for vol-iso material model \#7 agree with the plots in Fig. 2 in \cite{KossaMeccanica2023} for the same material model.

Plots of the Cauchy stress $\sigma_{11}$ versus the stretch $\tilde{\lambda}$ for vol-iso material models \#1-4 and \#5-8 are shown in Figs.~\ref{f12} and~\ref{f13}, respectively.
\begin{figure}
\begin{center}
\includegraphics{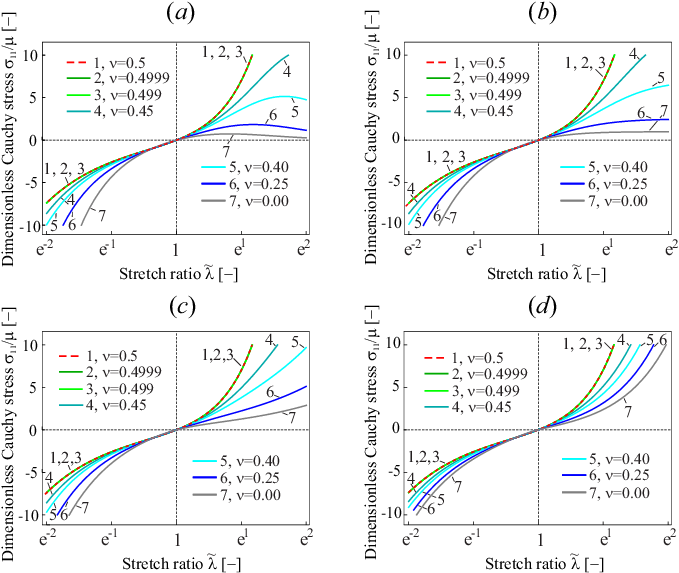}
\end{center}
\caption{Plots of $\sigma_{11}$ versus $\tilde{\lambda}$ in the UL problem for vol-iso models \#1 (\emph{a}), \#2 (\emph{b}), \#3 (\emph{c}), and \#4 (\emph{d}).}
\label{f12}
\end{figure}
\begin{figure}
\begin{center}
\includegraphics{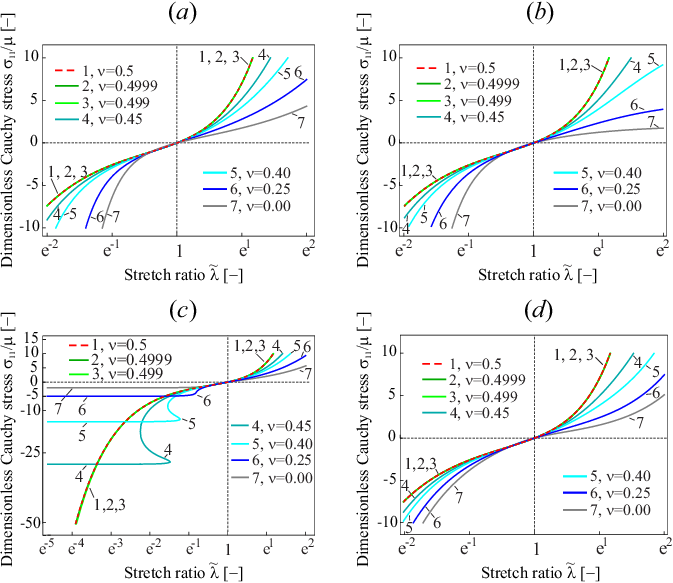}
\end{center}
\caption{Plots of $\sigma_{11}$ versus $\tilde{\lambda}$ in the UL problem for vol-iso models \#5 (\emph{a}), \#6 (\emph{b}), \#7 (\emph{c}), and \#8 (\emph{d}).}
\label{f13}
\end{figure}
Similar plots of the engineering (1st P-K, nominal) stress $P_{11}$ versus the stretch $\tilde{\lambda}$ for vol-iso material models \#1--4 and \#5--8 are shown in Figs.~\ref{f14} and~\ref{f15}, respectively.\footnote{To overcome the difficulties in deriving the dependencies $\sigma_{11}(\tilde{\lambda})$ and $P_{11}(\tilde{\lambda})$ for vol-iso material model \#1 using the Wolfram Mathematica software, we approximated the function $\ln J$ by the function $6(J^{1/12}-J^{-1/12})$ (i.e., we used the  value $q=1/12$ in \eqref{87}).}
\begin{figure}
\begin{center}
\includegraphics{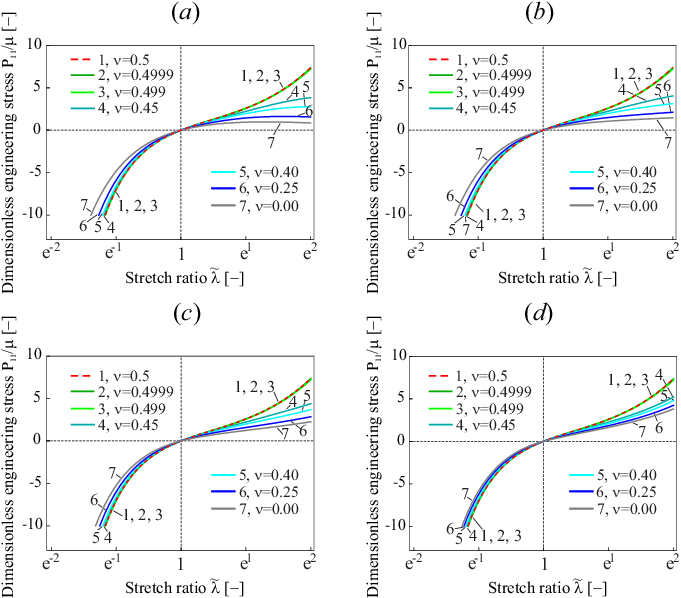}
\end{center}
\caption{Plots of $P_{11}$ versus $\tilde{\lambda}$ in the UL problem for vol-iso models \#1 (\emph{a}), \#2 (\emph{b}), \#3 (\emph{c}), and \#4 (\emph{d}).}
\label{f14}
\end{figure}
\begin{figure}
\begin{center}
\includegraphics{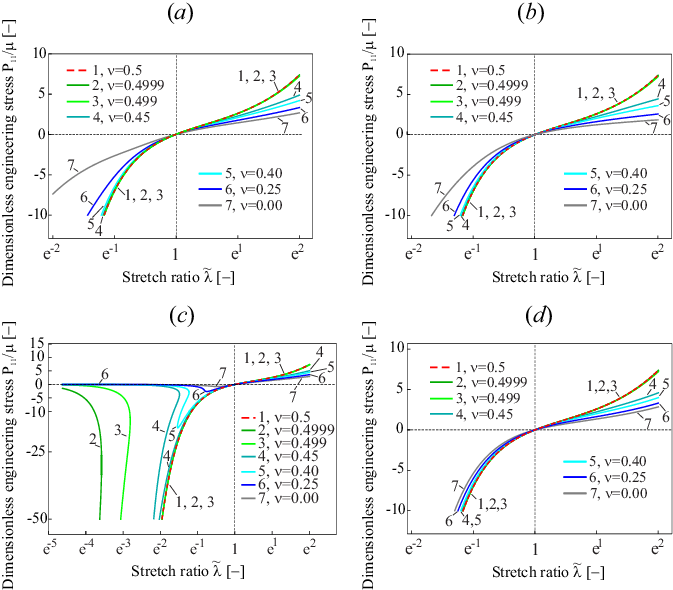}
\end{center}
\caption{Plots of $P_{11}$ versus $\tilde{\lambda}$ in the UL problem for vol-iso models \#5 (\emph{a}), \#6 (\emph{b}), \#7 (\emph{c}), and \#8 (\emph{d}).}
\label{f15}
\end{figure}
Note that the plots of the stress $\sigma_{11}$ versus the stretch $\tilde{\lambda}$ for vol-iso model \#3 in Fig. \ref{f12},\emph{c} agree with the plots in Fig. 10 in \cite{PenceMMS2015} for the same material model (the model of $W_b$ in terms of \cite{PenceMMS2015}). In addition, the plots $\sigma_{11}(\tilde{\lambda})$ and $P_{11}(\tilde{\lambda})$ in Figs. \ref{f13},\emph{c} and \ref{f15},\emph{c} for vol-iso material model \#7 agree with the plots in Figs. 14,\emph{a} and 5 in \cite{KossaMeccanica2023} for the same material model. We observe the non-monotonic dependencies of Cauchy stresses on stretches in Figs. \ref{f13},\emph{c} and \ref{f15},\emph{c} for this material model. This is clearly physically inadmissible as long as the elastic material is not damaged.

The limiting values of $\sigma_{11}$ and $P_{11}$ in extreme states obtained for vol-iso material models with $0\leq\nu< 0.5$ using the Wolfram Mathematica software are given in Table~\ref{t3}. It can be seen that physically reasonable limiting values exist only for vol-iso models \#3--6,8 (see Table~\ref{t2}).

The results of the solution of the UL problem for vol-iso models lead to the conclusion that physically reasonable solutions can be obtained only using volumetric functions \#3,4,8.

\subsection{Equibiaxial loading in plane stress}
\label{sec:6-3}

The dependencies of stresses and the unknown out-of-plane principal stretch versus prescribed in-plane principal stretches obtained by solving the problem of equibiaxial loading in plane stress for the incompressible isotropic neo-Hookean material model and compressible mixed and vol-iso neo-Hookean models are presented in Sections \ref{sec:6-3-1}, \ref{sec:6-3-2}, and \ref{sec:6-3-3}, respectively.

\subsubsection{Incompressible isotropic neo-Hookean material}
\label{sec:6-3-1}

Setting $J=1$, from \eqref{158} we obtain
\begin{equation}\label{185}
  \lambda_T= \tilde{\lambda}^{-2}.
\end{equation}
In view of \eqref{55} and \eqref{159}, the components of the Cauchy stress tensor can be written as
\begin{equation}\label{186}
  \sigma_{11} = \sigma_{22} = \mu\,(\tilde{\lambda}^2-1)-p, \quad\quad \sigma_{33} = \mu\,(\lambda_T^2-1)-p.
\end{equation}
Determining the Lagrange multiplier $p$ from $\eqref{156}_2$ and $\eqref{186}_2$ and using \eqref{185}, from $\eqref{186}_1$ we obtain
\begin{equation}\label{187}
  \sigma_{11} = \mu\,(\tilde{\lambda}^2-\tilde{\lambda}^{-4}).
\end{equation}
Using \eqref{185} and \eqref{187}, from \eqref{161} we get
\begin{equation}\label{188}
   P_{11}= \tilde{\lambda}^{-1} \sigma_{11} = \mu\,(\tilde{\lambda}-\tilde{\lambda}^{-5}).
\end{equation}

Using expressions \eqref{185}, \eqref{187}, and \eqref{188}, we obtain the limiting values of $\lambda_T$, $\sigma_{11}$, and $P_{11}$ in extreme states, which coincide with the corresponding values for the UL problem (see Table~\ref{t2}).

\subsubsection{Compressible isotropic mixed neo-Hookean material models}
\label{sec:6-3-2}

In view of \eqref{65} and \eqref{159}, the components of the Cauchy stress tensor can be written as
\begin{equation}\label{189}
  \sigma_{11} = \sigma_{22} =  \lambda\, h^{\prime}(J) + \frac{\mu}{J} (\tilde{\lambda}^2-1), \quad\quad \sigma_{33} = \lambda\, h^{\prime}(J) + \frac{\mu}{J} (\lambda_T^2-1).
\end{equation}
Using $\eqref{156}_2$ and $\eqref{189}_2$, we obtain the implicit nonlinear dependence of $\lambda_T$ on $\tilde{\lambda}$ in the general case:
\begin{equation}\label{190}
  \lambda\, h^{\prime}(\tilde{\lambda}^2\lambda_T) + \frac{\mu}{\tilde{\lambda}^2} (\lambda_T-\lambda_T^{-1})=0.
\end{equation}

As in Section \ref{sec:6-2-2}, we first consider the value $\nu=0$ for Poisson's ratio. Since $\lambda=0$ for this value of $\nu$, from \eqref{190} we obtain equality \eqref{175}, which does not depend on the choice of the volumetric function. Using \eqref{158}, \eqref{161}, and $\eqref{189}_1$, we get
\begin{equation*}
   \sigma_{11}= \sigma_{22}=\mu\,(1-\tilde{\lambda}^{-2}),\quad\quad P_{11}= P_{22}= \mu\,(\tilde{\lambda}-\tilde{\lambda}^{-1}).
\end{equation*}
For the remaining values of $\nu$ from the interval $0<\nu<0.5$, the value of $\lambda_T$ should be determined from the nonlinear equation \eqref{190}. In the particular case of mixed model \#7, using the function $h^{\prime}(J)$ of the form $\eqref{93}_2$ in \eqref{190}, we obtain the solution of Eq. \eqref{190} in closed form:
\begin{equation*}
  \lambda_T=\left[\mu/(\lambda\tilde{\lambda}^4+\mu)\right]^{1/2}.
\end{equation*}
For the volumetric functions $h^{\prime}(J)$ considered in this book, the dependence $\lambda_T(\tilde{\lambda})$ is derived from \eqref{190} using the Wolfram Mathematica software. Substitution of the obtained dependence into $\eqref{189}_1$ taking into account expression \eqref{158} yields the dependence $\sigma_{11}(\tilde{\lambda})$. The obtained dependencies $\lambda_T(\tilde{\lambda})$ and $\sigma_{11}(\tilde{\lambda})$ are used to derive expressions of $P_{11}(\tilde{\lambda})=P_{22}(\tilde{\lambda})$ from \eqref{161}.

Since the dependencies of out-of-plane principal stretch and stresses on prescribed in-plane principal stretches obtained in the ELP and ULP problems are qualitatively similar to the corresponding dependencies obtained in the UL problem and presented in Section \ref{sec:6-2}, in this and subsequent sections, we restrict ourselves to testing models \#1,4,7 for the most widely used volumetric functions and \#8 for the new one. Plots of $\lambda_T$ versus $\tilde{\lambda}$ are given in Fig.~\ref{f16}, and plots $\sigma_{11}$ and $P_{11}$ versus $\tilde{\lambda}$ are given in Figs.~\ref{f17} and \ref{f18}, respectively.
\begin{figure}
\begin{center}
\includegraphics{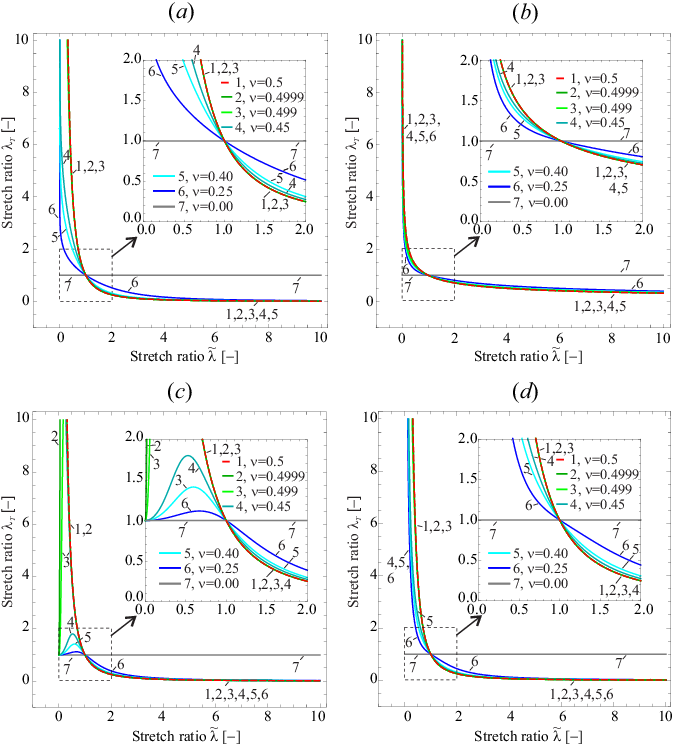}
\end{center}
\caption{Plots of $\lambda_T$ versus $\tilde{\lambda}$ in the ELP problem for mixed models \#1 (\emph{a}), \#4 (\emph{b}), \#7 (\emph{c}), and \#8 (\emph{d}).}
\label{f16}
\end{figure}
\begin{figure}
\begin{center}
\includegraphics{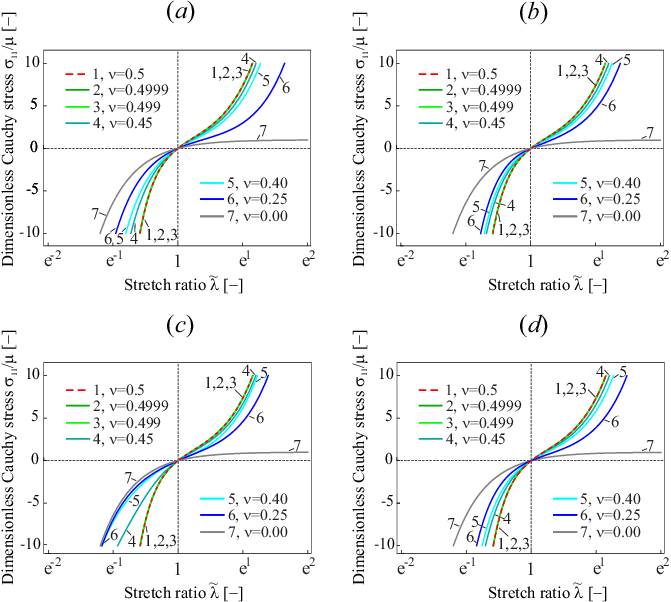}
\end{center}
\caption{Plots of $\sigma_{11}$ versus $\tilde{\lambda}$ in the ELP problem for mixed models \#1 (\emph{a}), \#4 (\emph{b}), \#7 (\emph{c}), and \#8 (\emph{d}).}
\label{f17}
\end{figure}
\begin{figure}
\begin{center}
\includegraphics{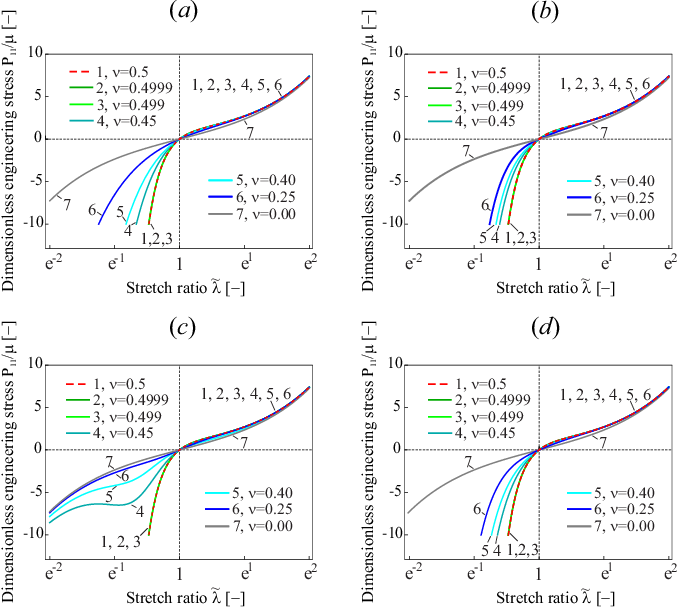}
\end{center}
\caption{Plots of $P_{11}$ versus $\tilde{\lambda}$ in the ELP problem for mixed models \#1 (\emph{a}), \#4 (\emph{b}), \#7 (\emph{c}), and \#8 (\emph{d}).}
\label{f18}
\end{figure}
We observe the non-monotonic physically inadmissible dependencies of Cauchy stresses on stretches in Fig. \ref{f18},\emph{c} for material model \#7.

The limiting values of $\lambda_T$, $\sigma_{11}$, and $P_{11}$ in extreme states are presented in Table~\ref{t4}.
\begin{table}
\caption{Limiting values of $\lambda_T$, $\sigma_{11}$, and $P_{11}$ in extreme states where $\tilde{\lambda}\rightarrow 0$ and $\tilde{\lambda}\rightarrow \infty$ in the solution of the ELP problem for compressible isotropic material models with $0\leq\nu< 0.5$}
\label{t4}
\begin{tabular}{llllll}
\hline\noalign{\smallskip}
Model    &                                & \multicolumn{2}{c}{Mixed models}                                    & \multicolumn{2}{c}{Vol-iso models} \\
ID       & Quantity${}^{\flat}$           & $\tilde{\lambda}\rightarrow 0$ & $\tilde{\lambda}\rightarrow \infty$ & $\tilde{\lambda}\rightarrow 0$ & $\tilde{\lambda}\rightarrow \infty$ \\
\noalign{\smallskip}\hline\noalign{\smallskip}
         & $\lambda_T(\tilde{\lambda})$   & $\textcolor{green}{+\infty}$   &  $\textcolor{green}{0}$             &
                                           \textcolor{red}{0}              &  $\textcolor{red}{+\infty}$         \\
   1     & $\sigma_{11}(\tilde{\lambda})$ & $\textcolor{green}{-\infty}$   &  $\textcolor{green}{+\infty}$       &
                                            $\textcolor{green}{-\infty}$   &  \textcolor{red}{0}                 \\
         & $P_{11}(\tilde{\lambda})$      & $\textcolor{green}{-\infty}$   &  $\textcolor{green}{+\infty}$       &
                                            $\textcolor{green}{-\infty}$   &  \textcolor{red}{0}                 \\
\noalign{\smallskip}\hline\noalign{\smallskip}
         & $\lambda_T(\tilde{\lambda})$   & $\textcolor{green}{+\infty}$   &  $\textcolor{green}{0}$             &
                                            $\textcolor{green}{+\infty}$   &  $\textcolor{green}{0}$            \\
   4     & $\sigma_{11}(\tilde{\lambda})$ & $\textcolor{green}{-\infty}$   &  $\textcolor{green}{+\infty}$       &
                                            $\textcolor{green}{-\infty}$   &  $\textcolor{green}{+\infty}$       \\
         & $P_{11}(\tilde{\lambda})$      & $\textcolor{green}{-\infty}$   &  $\textcolor{green}{+\infty}$       &
                                            $\textcolor{green}{-\infty}$   &  $\textcolor{green}{+\infty}$       \\
\noalign{\smallskip}\hline\noalign{\smallskip}
         & $\lambda_T(\tilde{\lambda})$   & \textcolor{red}{1}             &  $\textcolor{green}{0}$             &
                                            \textcolor{red}{0}             &  $\textcolor{green}{0}$             \\
   7     & $\sigma_{11}(\tilde{\lambda})$ & $\textcolor{green}{-\infty}$   &  $\textcolor{green}{+\infty}$       &
                                            $\textcolor{red}{-3K/2}$       &  $\textcolor{green}{+\infty}$       \\
         & $P_{11}(\tilde{\lambda})$      & $\textcolor{green}{-\infty}$   &  $\textcolor{green}{+\infty}$       &
                                            \textcolor{red}{0}             &  $\textcolor{green}{+\infty}$       \\
\noalign{\smallskip}\hline\noalign{\smallskip}
         & $\lambda_T(\tilde{\lambda})$   & $\textcolor{green}{+\infty}$   &  $\textcolor{green}{0}$             &
                                            $\textcolor{green}{+\infty}$   &  $\textcolor{green}{0}$            \\
   8     & $\sigma_{11}(\tilde{\lambda})$ & $\textcolor{green}{-\infty}$   &  $\textcolor{green}{+\infty}$       &
                                            $\textcolor{green}{-\infty}$   &  $\textcolor{green}{+\infty}$       \\
         & $P_{11}(\tilde{\lambda})$      & $\textcolor{green}{-\infty}$   &  $\textcolor{green}{+\infty}$       &
                                            $\textcolor{green}{-\infty}$   &  $\textcolor{green}{+\infty}$       \\
\noalign{\smallskip}\hline
\end{tabular}
\footnotesize
\begin{itemize}
\item[${}^{\flat}$]Green and red colors indicate physically reasonable and unreasonable values of a quantity, respectively.
\end{itemize}
\normalsize
\end{table}
These limiting values coincide with the corresponding limiting values for the same material models in Table~\ref{t3}. Note that the limiting values of these quantities for model \#7 coincide with the corresponding values for this material model in Table 2 in \cite{PenceMMS2015}.

The conclusion following from the solutions of the ELP problem using mixed models is similar to the conclusion drawn from the analysis of solutions of the UL problem at the end of Section \ref{sec:6-2-2}.

\subsubsection{Compressible isotropic vol-iso neo-Hookean material models}
\label{sec:6-3-3}

In view of \eqref{68} and \eqref{160}, the components of the Cauchy stress tensor can be written as
\begin{equation}\label{193}
  \sigma_{11} = \sigma_{22} = K\, h^{\prime}(J) + \frac{1}{3}\mu J^{-5/3} (\tilde{\lambda}^2-\lambda_T^2), \quad\quad
  \sigma_{33} = K\, h^{\prime}(J) + \frac{2}{3} \mu J^{-5/3} (\lambda_T^2-\tilde{\lambda}^2).
\end{equation}
Regardless of the choice of Poisson's ratio $\nu \in [0,0.5)$, Eqs. $\eqref{156}_2$, \eqref{158}, and $\eqref{193}_2$ lead to the following nonlinear equation for the dependence $\lambda_T(\tilde{\lambda})$:
\begin{equation*}
  K\, h^{\prime}(\tilde{\lambda}^2\lambda_T) + \frac{2}{3} \mu J^{-5/3} (\lambda_T^2-\tilde{\lambda}^2)=0.
\end{equation*}
Summing the left and right sides of the equalities in \eqref{193} and taking into account  equality $\sigma_{33}=0$, we get
\begin{equation}\label{195}
  \sigma_{11} = \sigma_{22} = \frac{3}{2} K\, h^{\prime}(\tilde{\lambda}^2\lambda_T).
\end{equation}
Substitution of $\lambda_T(\tilde{\lambda})$ into the right-hand side of \eqref{195} leads to the dependencies $\sigma_{11}(\tilde{\lambda})=\sigma_{22}(\tilde{\lambda})$. The dependencies $P_{11}(\tilde{\lambda})=P_{22}(\tilde{\lambda})$ are obtained from \eqref{161} using the dependencies $\lambda_T(\tilde{\lambda})$ and $\sigma_{11}(\tilde{\lambda})=\sigma_{22}(\tilde{\lambda})$.

Plots of $\lambda_T$ versus $\tilde{\lambda}$ are shown in Fig.~\ref{f19}, and plots of $\sigma_{11}$ and $P_{11}$ versus $\tilde{\lambda}$ are given in Figs.~\ref{f20} and \ref{f21}, respectively.
\begin{figure}
\begin{center}
\includegraphics{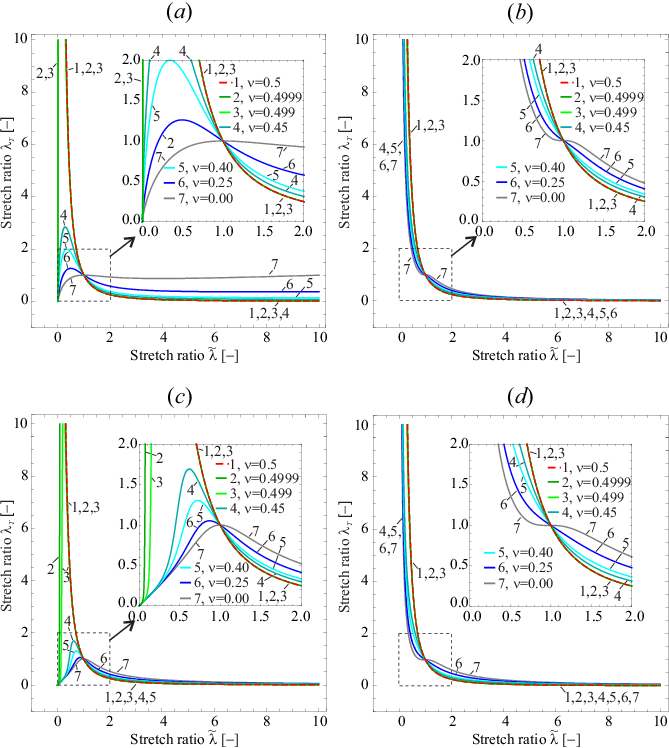}
\end{center}
\caption{Plots of $\lambda_T$ versus $\tilde{\lambda}$ in the ELP problem for vol-iso models \#1 (\emph{a}), \#4 (\emph{b}), \#7 (\emph{c}), and \#8 (\emph{d}).}
\label{f19}
\end{figure}
\begin{figure}
\begin{center}
\includegraphics{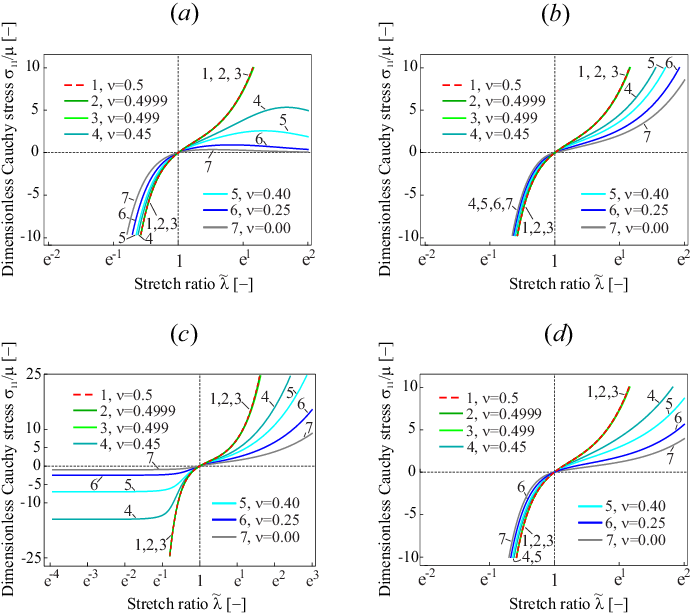}
\end{center}
\caption{Plots of $\sigma_{11}$ versus $\tilde{\lambda}$ in the ELP problem for vol-iso models \#1 (\emph{a}), \#4 (\emph{b}), \#7 (\emph{c}), and \#8 (\emph{d}).}
\label{f20}
\end{figure}
\begin{figure}
\begin{center}
\includegraphics{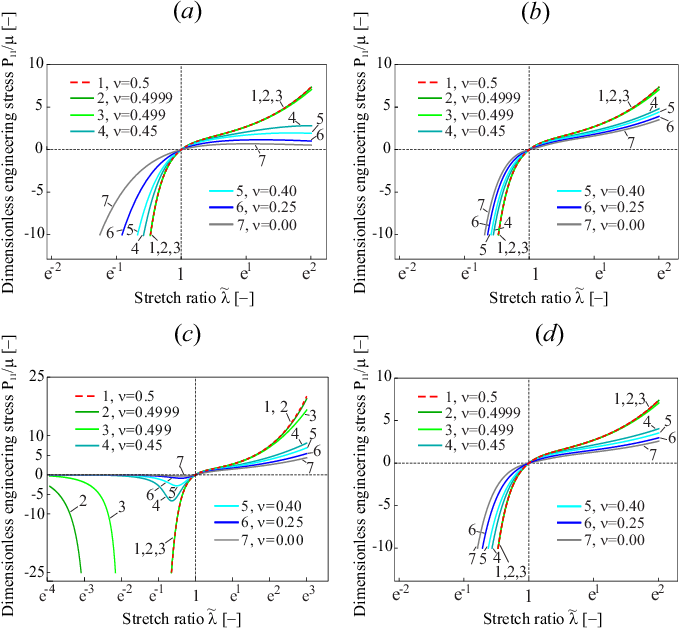}
\end{center}
\caption{Plots of $P_{11}$ versus $\tilde{\lambda}$ in the ELP problem for vol-iso models \#1 (\emph{a}), \#4 (\emph{b}), \#7 (\emph{c}), and \#8 (\emph{d}).}
\label{f21}
\end{figure}
Note that the plots of $\lambda_T(\tilde{\lambda})$ in Fig. \ref{f19},\emph{c}, the plots of $\sigma_{11}(\tilde{\lambda})$ in Fig. \ref{f20},\emph{c}, and the plots of $P_{11}(\tilde{\lambda})$ in Fig. \ref{f21},\emph{c} for vol-iso material model \#7 agree with the plots in Figs. 6, 14,\emph{b}, and 7,\emph{a} in \cite{KossaMeccanica2023} for the same material model. We observe the non-monotonic physically inadmissible dependencies of Cauchy stresses on stretches in Fig. \ref{f21},\emph{c} for material model \#7.

The limiting values of $\lambda_T$, $\sigma_{11}$, and $P_{11}$ in extreme states are presented in Table~\ref{t4}. These limiting values are in qualitative agreement with the limiting values for the same material models presented in Table~\ref{t3}. Note that the limiting values of these quantities for model \#3 are presented in Table 2 in \cite{PenceMMS2015}. The limiting values of these quantities for this model, though not presented in Table~\ref{t4}, agree with the limiting values in Table 2 in \cite{PenceMMS2015}. Note also that the limiting values of these quantities for model \#7 coincide with those for the same model presented in Fig.~11 in \cite{KossaMeccanica2023}.

The results of the solution of the ELP problem for vol-iso models lead to the conclusion that physically reasonable solutions can be obtained only using volumetric function \#4 from the Hartmann--Neff family and new one \#8.

\subsection{Uniaxial loading in plane strain}
\label{sec:6-4}

The dependencies of stresses and the unknown out-of-plane principal stretch on the prescribed longitudinal in-plane principal stretch obtained by solving the problem of uniaxial loading in plane strain for the incompressible isotropic neo-Hookean material model and compressible mixed and vol-iso isotropic neo-Hookean models are presented in Sections \ref{sec:6-4-1}, \ref{sec:6-4-2}, and \ref{sec:6-4-3}, respectively. \index{Homogeneous deformation!uniaxial loading in plane strain (ULP)}

\subsubsection{Incompressible isotropic neo-Hookean materials}
\label{sec:6-4-1}

Setting $J=1$, from \eqref{164} we obtain
\begin{equation}\label{196}
  \lambda_T= \tilde{\lambda}^{-1}.
\end{equation}
In view of \eqref{55} and \eqref{165}, the components of the Cauchy stress tensor can be written as
\begin{equation}\label{197}
  \sigma_{11} =  \mu\,(\tilde{\lambda}^2-1)-p,\quad \sigma_{22} = -p, \quad\quad
  \sigma_{33} =  \mu\,(\lambda_T^2-1)-p.
\end{equation}
Determining the Lagrange multiplier $p$ from $\eqref{162}_3$ and $\eqref{197}_3$ and  using \eqref{196}, from $\eqref{197}_{1,2}$ we obtain
\begin{equation}\label{198}
  \sigma_{11} = \mu\,(\tilde{\lambda}^2-\tilde{\lambda}^{-2}),\quad\quad \sigma_{22} = -\mu\,(\tilde{\lambda}^{-2}-1).
\end{equation}
Using \eqref{196} and \eqref{198}, from \eqref{167} we get
\begin{equation}\label{199}
   P_{11}=  \mu\,(\tilde{\lambda}-\tilde{\lambda}^{-3}),\quad\quad P_{22}= -\mu\,(\tilde{\lambda}^{-2}-1).
\end{equation}

The limiting values of $\lambda_T$, $\sigma_{11}$, $\sigma_{22}$, $P_{11}$, and $P_{22}$ in extreme states are obtained from expressions \eqref{196}, \eqref{198}, and \eqref{199} and are presented in Table \ref{t5}.
\begin{table}
\caption{Limiting values of $\lambda_T$, $\sigma_{11}$, $\sigma_{22}$, $P_{11}$, and $P_{22}$ in extreme states where $\tilde{\lambda}\rightarrow 0$ and $\tilde{\lambda}\rightarrow \infty$ in the ULP problem for the incompressible isotropic neo-Hookean material model}
\label{t5}
\begin{tabular}{lll}
\hline\noalign{\smallskip}
 Quantity                       & $\tilde{\lambda}\rightarrow 0$ & $\tilde{\lambda}\rightarrow \infty$ \\
\noalign{\smallskip}\hline\noalign{\smallskip}
 $\lambda_T(\tilde{\lambda})$   & $+\infty$                      &  0                                  \\
 $\sigma_{11}(\tilde{\lambda})$ & $-\infty$                      &  $+\infty$                          \\
 $\sigma_{22}(\tilde{\lambda})$ & $-\infty$                      &  $+\mu$                          \\
 $P_{11}(\tilde{\lambda})$      & $-\infty$                      &  $+\infty$                          \\
 $P_{22}(\tilde{\lambda})$      & $-\infty$                      &  $+\mu$                          \\
\noalign{\smallskip}\hline
\end{tabular}
\end{table}
We assume that these limiting values correspond to the physically reasonable responses for idealized hyperelastic materials.

\subsubsection{Compressible isotropic mixed neo-Hookean material models}
\label{sec:6-4-2}

In view of \eqref{65} and \eqref{165}, the components of the Cauchy stress tensor can be written as
\begin{equation}\label{200}
  \sigma_{11} = \lambda\, h^{\prime}(J) + \frac{\mu}{J} (\tilde{\lambda}^2-1),\quad\quad  \sigma_{22} = \lambda\, h^{\prime}(J), \quad\quad
  \sigma_{33} = \lambda\, h^{\prime}(J) + \frac{\mu}{J} (\lambda_T^2-1).
\end{equation}
Using $\eqref{162}_3$ and $\eqref{200}_3$, we obtain the nonlinear implicit dependence of $\lambda_T$ on $\tilde{\lambda}$ in the general case:
\begin{equation}\label{201}
  \lambda\, h^{\prime}(\tilde{\lambda}\lambda_T) + \frac{\mu}{\tilde{\lambda}} (\lambda_T-\lambda_T^{-1})=0.
\end{equation}

As in Section \ref{sec:6-2-2}, we first consider the value $\nu=0$ for Poisson's ratio. Since $\lambda=0$ for this value of $\nu$, from \eqref{201} we obtain equality \eqref{175}, which does not depend on the choice of the volumetric function. Using \eqref{164}, \eqref{167}, and $\eqref{200}_{1,2}$, we get
\begin{equation*}
   \sigma_{11}= \mu\,(\tilde{\lambda}-\tilde{\lambda}^{-1}),\quad\quad \sigma_{22}=0,\quad\quad P_{11}=\sigma_{11}= \mu\,(\tilde{\lambda}-\tilde{\lambda}^{-1}),\quad\quad  P_{22}= 0.
\end{equation*}
For the remaining values of $\nu$ from the interval $0<\nu<0.5$, the value of $\lambda_T$ should be determined from the nonlinear equation \eqref{201}. In the particular case of mixed model \#7, using the function $h^{\prime}(J)$ of the form $\eqref{93}_2$ in \eqref{201}, we obtain the solution of Eq. \eqref{201} in closed form:
\begin{equation*}
  \lambda_T=\frac{1}{2(\lambda \tilde{\lambda}^2 + \mu)}(\lambda \tilde{\lambda} + \sqrt{\lambda^2 \tilde{\lambda}^2 + 4\mu(\lambda \tilde{\lambda}^2 + \mu)}).
\end{equation*}
For all remaining volumetric functions $h^{\prime}(J)$ considered in this book, the dependence $\lambda_T(\tilde{\lambda})$ is derived from \eqref{201} using the Wolfram Mathematica software. Substitution of the obtained dependence into $\eqref{200}_{1,2}$ taking into account  expression \eqref{164} yields the dependencies $\sigma_{11}(\tilde{\lambda})$ and $\sigma_{22}(\tilde{\lambda})$. The obtained dependencies $\lambda_T(\tilde{\lambda})$, $\sigma_{11}(\tilde{\lambda})$, and $\sigma_{22}(\tilde{\lambda})$ are used to derive expressions for $P_{11}(\tilde{\lambda})$ and $P_{22}(\tilde{\lambda})$ from \eqref{167}.

Plots of $\lambda_T$ versus $\tilde{\lambda}$ are given in Fig.~\ref{f22} and plots of $\sigma_{11}$, $\sigma_{22}$, $P_{11}$, and $P_{22}$ versus $\tilde{\lambda}$ are shown in Figs.~\ref{f23}--\ref{f26}, respectively.
\begin{figure}
\begin{center}
\includegraphics{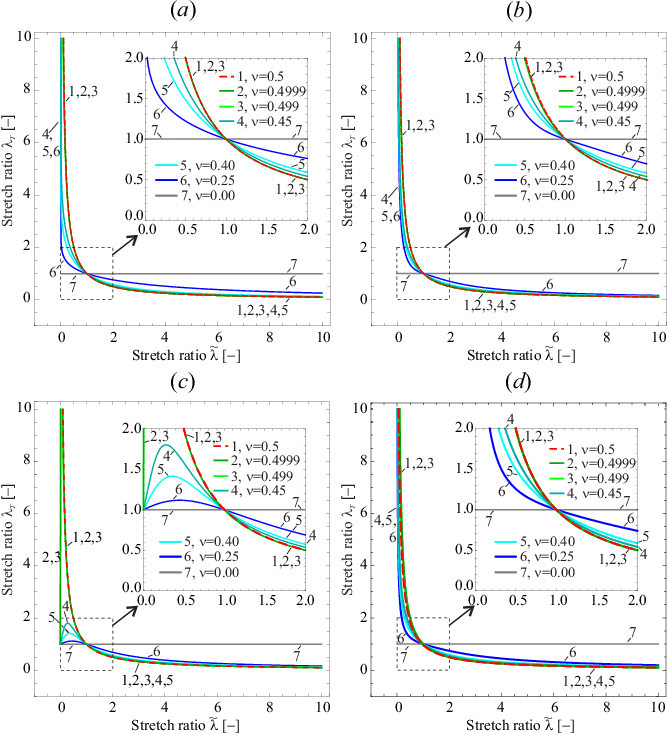}
\end{center}
\caption{Plots of $\lambda_T$ versus $\tilde{\lambda}$ in the ULP problem for mixed models \#1 (\emph{a}), \#4 (\emph{b}), \#7 (\emph{c}), and \#8 (\emph{d}).}
\label{f22}
\end{figure}
\begin{figure}
\begin{center}
\includegraphics{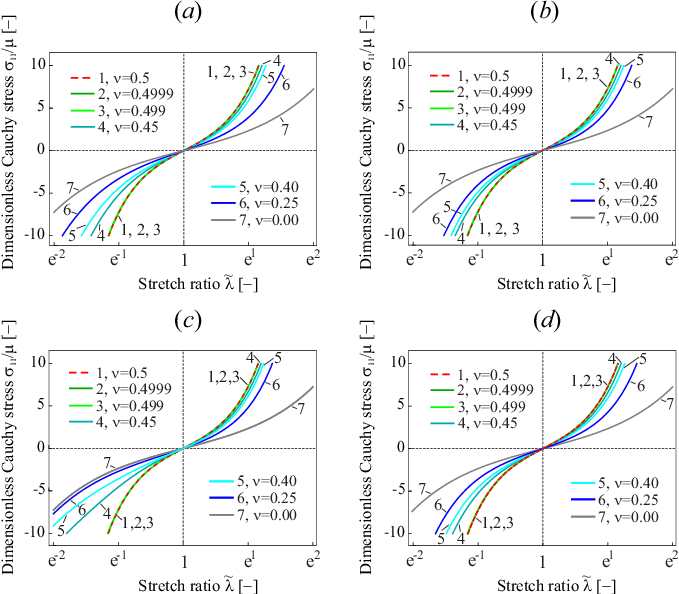}
\end{center}
\caption{Plots of $\sigma_{11}$ versus $\tilde{\lambda}$ in the ULP problem for mixed models \#1 (\emph{a}), \#4 (\emph{b}), \#7 (\emph{c}), and \#8 (\emph{d}).}
\label{f23}
\end{figure}
\begin{figure}
\begin{center}
\includegraphics{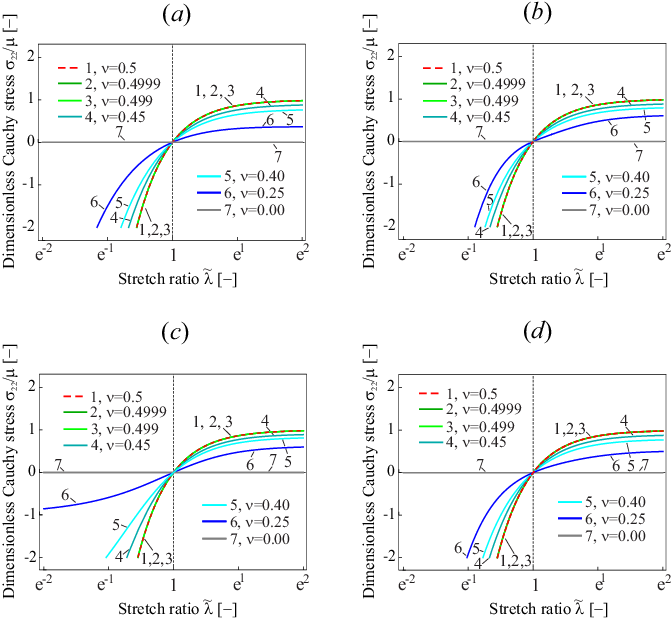}
\end{center}
\caption{Plots of $\sigma_{22}$ versus $\tilde{\lambda}$ in the ULP problem for mixed models \#1 (\emph{a}), \#4 (\emph{b}), \#7 (\emph{c}), and \#8 (\emph{d}).}
\label{f24}
\end{figure}
\begin{figure}
\begin{center}
\includegraphics{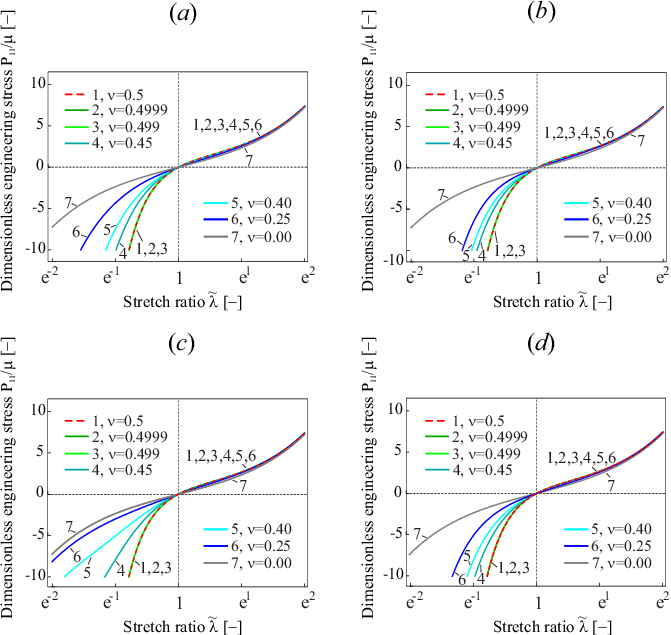}
\end{center}
\caption{Plots of $P_{11}$ versus $\tilde{\lambda}$ in the ULP problem for mixed models \#1 (\emph{a}), \#4 (\emph{b}), \#7 (\emph{c}), and \#8 (\emph{d}).}
\label{f25}
\end{figure}
\begin{figure}
\begin{center}
\includegraphics{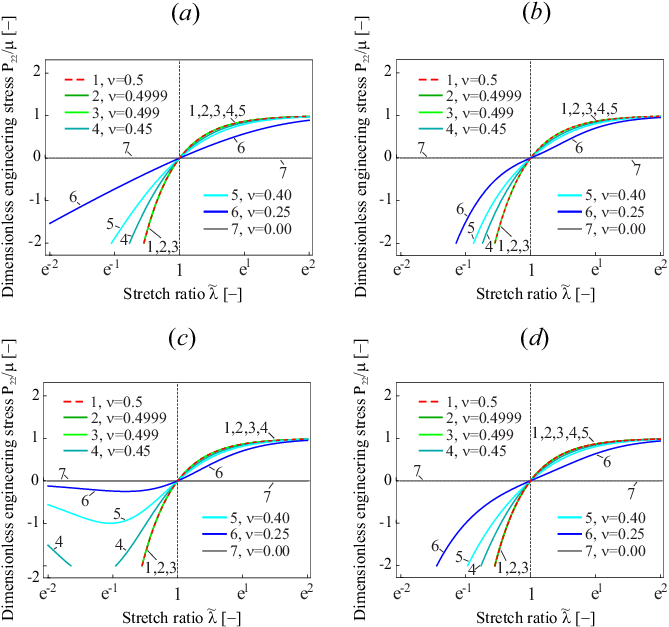}
\end{center}
\caption{Plots of $P_{22}$ versus $\tilde{\lambda}$ in the ULP problem for mixed models \#1 (\emph{a}), \#4 (\emph{b}), \#7 (\emph{c}), and \#8 (\emph{d}).}
\label{f26}
\end{figure}
We observe the non-monotonic physically inadmissible dependencies of Cauchy stresses on stretches in Fig. \ref{f26},\emph{c} for material model \#7.

The limiting values of $\lambda_T$, $\sigma_{11}$, $\sigma_{22}$, $P_{11}$, and $P_{22}$ in extreme states are presented in Table~\ref{t6}.
\begin{table}
\caption{Limiting values of $\lambda_T$, $\sigma_{11}$, $\sigma_{22}$, $P_{11}$, and $P_{22}$ in extreme states where $\tilde{\lambda}\rightarrow 0$ and $\tilde{\lambda}\rightarrow \infty$ in the solution of the ULP problem for the compressible isotropic material models with $0\leq\nu< 0.5$}
\label{t6}
\begin{tabular}{llllll}
\hline\noalign{\smallskip}
Model    &                                & \multicolumn{2}{c}{Mixed models}                                      & \multicolumn{2}{c}{Vol-iso models} \\
ID       & Quantity${}^{\flat}$           & $\tilde{\lambda}\rightarrow 0$ & $\tilde{\lambda}\rightarrow \infty$ & $\tilde{\lambda}\rightarrow 0$ & $\tilde{\lambda}\rightarrow \infty$ \\
\noalign{\smallskip}\hline\noalign{\smallskip}
         & $\lambda_T(\tilde{\lambda})$   & $\textcolor{green}{+\infty}$   &  $\textcolor{green}{0}$             &
                                            $\textcolor{red}{1/\sqrt{2}}$  &  $\textcolor{red}{+\infty}$         \\
         & $\sigma_{11}(\tilde{\lambda})$ & $\textcolor{green}{-\infty}$   &  $\textcolor{green}{+\infty}$       &
                                            $\textcolor{green}{-\infty}$   &  \textcolor{red}{0}                 \\
   1     & $\sigma_{22}(\tilde{\lambda})$ & $\textcolor{green}{-\infty}$   &  $\textcolor{green}{\ast}$          &
                                            $\textcolor{green}{-\infty}$   &  \textcolor{red}{0}                 \\
         & $P_{11}(\tilde{\lambda})$      & $\textcolor{green}{-\infty}$   &  $\textcolor{green}{+\infty}$       &
                                            $\textcolor{green}{-\infty}$   &  \textcolor{red}{0}                 \\
         & $P_{22}(\tilde{\lambda})$      & $\textcolor{green}{-\infty}$   &  $\textcolor{green}{+\mu}$          &
                                            $\textcolor{red}{\pm\infty}$   &  $\textcolor{red}{-\infty}$         \\
\noalign{\smallskip}\hline\noalign{\smallskip}
         & $\lambda_T(\tilde{\lambda})$   & $\textcolor{green}{+\infty}$   &  $\textcolor{green}{0}$             &
                                            $\textcolor{green}{+\infty}$   &  $\textcolor{green}{0}$             \\
         & $\sigma_{11}(\tilde{\lambda})$ & $\textcolor{green}{-\infty}$   &  $\textcolor{green}{+\infty}$       &
                                            $\textcolor{green}{-\infty}$   &  $\textcolor{green}{+\infty}$       \\
   4     & $\sigma_{22}(\tilde{\lambda})$ & $\textcolor{green}{-\infty}$   &  $\textcolor{green}{\ast}$          &
                                            $\textcolor{green}{-\infty}$   &  $\textcolor{green}{+\infty}$       \\
         & $P_{11}(\tilde{\lambda})$      & $\textcolor{green}{-\infty}$   &  $\textcolor{green}{+\infty}$       &
                                            $\textcolor{green}{-\infty}$   &  $\textcolor{green}{+\infty}$       \\
         & $P_{22}(\tilde{\lambda})$      & $\textcolor{green}{-\infty}$   &  $\textcolor{green}{+\mu}$          &
                                            $\textcolor{green}{-\infty}$   &  $\textcolor{red}{+\infty}$         \\
\noalign{\smallskip}\hline\noalign{\smallskip}
         & $\lambda_T(\tilde{\lambda})$   & \textcolor{red}{1}             &  $\textcolor{green}{0}$             &
                                            $\textcolor{red}{1/\sqrt{2}}$  &  $\textcolor{green}{0}$             \\
         & $\sigma_{11}(\tilde{\lambda})$ & $\textcolor{green}{-\infty}$   &  $\textcolor{green}{+\infty}$       &
                                            $\textcolor{green}{-\infty}$   &  $\textcolor{green}{+\infty}$       \\
   7     & $\sigma_{22}(\tilde{\lambda})$ & $\textcolor{red}{-\lambda}$    &  $\textcolor{green}{\ast}$          &
                                            $\textcolor{red}{+\infty}$     &  $\textcolor{red}{+\infty}$         \\
         & $P_{11}(\tilde{\lambda})$      & $\textcolor{green}{-\infty}$   &  $\textcolor{green}{+\infty}$       &
                                            $\textcolor{green}{-\infty}$   &  $\textcolor{green}{+\infty}$       \\
         & $P_{22}(\tilde{\lambda})$      & $\textcolor{red}{0}$           &  $\textcolor{green}{+\mu}$          &
                                            $\textcolor{red}{\infty}$      &  $\textcolor{red}{+\infty}$         \\
\noalign{\smallskip}\hline\noalign{\smallskip}
         & $\lambda_T(\tilde{\lambda})$   & $\textcolor{green}{+\infty}$   &  $\textcolor{green}{0}$             &
                                            $\textcolor{green}{+\infty}$   &  $\textcolor{green}{0}$             \\
         & $\sigma_{11}(\tilde{\lambda})$ & $\textcolor{green}{-\infty}$   &  $\textcolor{green}{+\infty}$       &
                                            $\textcolor{green}{-\infty}$   &  $\textcolor{green}{+\infty}$       \\
   8     & $\sigma_{22}(\tilde{\lambda})$ & $\textcolor{green}{-\infty}$   &  $\textcolor{green}{\ast}$          &
                                            $\textcolor{green}{-\infty}$   &  $\textcolor{green}{+\infty}$       \\
         & $P_{11}(\tilde{\lambda})$      & $\textcolor{green}{-\infty}$   &  $\textcolor{green}{+\infty}$       &
                                            $\textcolor{green}{-\infty}$   &  $\textcolor{green}{+\infty}$       \\
         & $P_{22}(\tilde{\lambda})$      & $\textcolor{green}{-\infty}$   &  $\textcolor{green}{+\mu}$          &
                                            $\textcolor{green}{-\infty}$   &  $\textcolor{red}{+\infty}$         \\
\noalign{\smallskip}\hline
\end{tabular}
\footnotesize
\begin{itemize}
\item[${}^{\flat}$]Green and red colors indicate physically reasonable and unreasonable values of a quantity, a symbol   $\pm\infty$ denotes limiting values $+\infty$ or $-\infty$ for different values of Poisson's ratio.
\end{itemize}
\normalsize
\end{table}
These limiting values mostly coincide with the limiting values for the same material models presented in Table~\ref{t5}.

The conclusion following from the solutions of the ULP problem for mixed models is similar to the conclusion drawn from the analysis of solutions of the UL problem at the end of Section \ref{sec:6-2-2}.

\subsubsection{Compressible isotropic vol-iso neo-Hookean material models}
\label{sec:6-4-3}

In view of \eqref{68} and \eqref{166}, the components of the Cauchy stress tensor can be written as
\begin{align}
  \sigma_{11} &= K\, h^{\prime}(J) + \frac{1}{3}\mu\, J^{-5/3} (2\tilde{\lambda}^2 -1 -\lambda_T^2), \label{204} \\
  \sigma_{22} &= K\, h^{\prime}(J) + \frac{1}{3}\mu\, J^{-5/3} (-\tilde{\lambda}^2 + 2 -\lambda_T^2), \notag \\
  \sigma_{33} &= K\, h^{\prime}(J) + \frac{1}{3}\mu\, J^{-5/3} (2\lambda_T^2 -1 -\tilde{\lambda}^2). \notag
\end{align}
Regardless of the choice of Poisson's ratio $\nu \in [0,0.5)$, Eqs. $\eqref{162}_3$, \eqref{164}, and $\eqref{204}_3$ lead to the following nonlinear equation for the dependence   $\lambda_T(\tilde{\lambda})$:
\begin{equation*}
  K\, h^{\prime}(J) + \frac{1}{3}\mu\, J^{-5/3} (2\lambda_T^2-\tilde{\lambda}^2)=0.
\end{equation*}
Substitution of the dependencies $\lambda_T(\tilde{\lambda})$ into the right-hand side of $\eqref{204}_{1,2}$ yields the dependencies $\sigma_{11}(\tilde{\lambda})$ and $\sigma_{22}(\tilde{\lambda})$. The dependencies $P_{11}(\tilde{\lambda})$ and $P_{22}(\tilde{\lambda})$ are obtained from \eqref{167} using the dependencies $\lambda_T(\tilde{\lambda})$, $\sigma_{11}(\tilde{\lambda})$, and $\sigma_{22}(\tilde{\lambda})$.

Plots of $\lambda_T$ versus $\tilde{\lambda}$ are shown in Fig.~\ref{f27}, and plots of $\sigma_{11}$, $\sigma_{22}$, $P_{11}$, and $P_{22}$ versus $\tilde{\lambda}$ in Figs.~\ref{f28}-\ref{f31}, respectively.
\begin{figure}
\begin{center}
\includegraphics{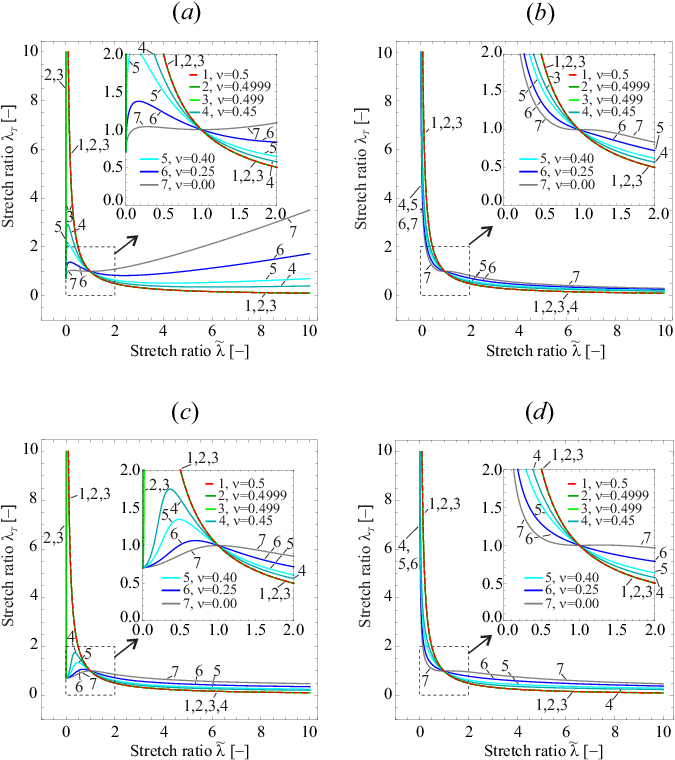}
\end{center}
\caption{Plots of $\lambda_T$ versus $\tilde{\lambda}$ in the ULP problem for vol-iso models \#1 (\emph{a}), \#4 (\emph{b}), \#7 (\emph{c}), and \#8 (\emph{d}).}
\label{f27}
\end{figure}
\begin{figure}
\begin{center}
\includegraphics{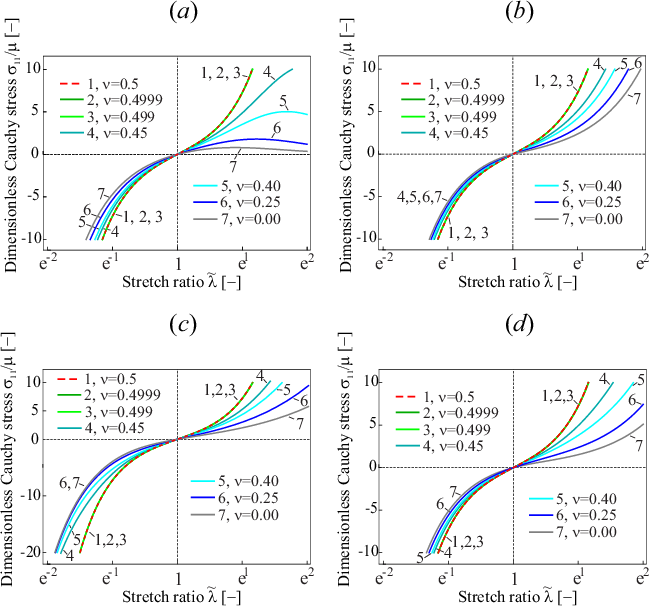}
\end{center}
\caption{Plots of $\sigma_{11}$ versus $\tilde{\lambda}$ in the ULP problem for vol-iso models \#1 (\emph{a}), \#4 (\emph{b}), \#7 (\emph{c}), and \#8 (\emph{d}).}
\label{f28}
\end{figure}
\begin{figure}
\begin{center}
\includegraphics{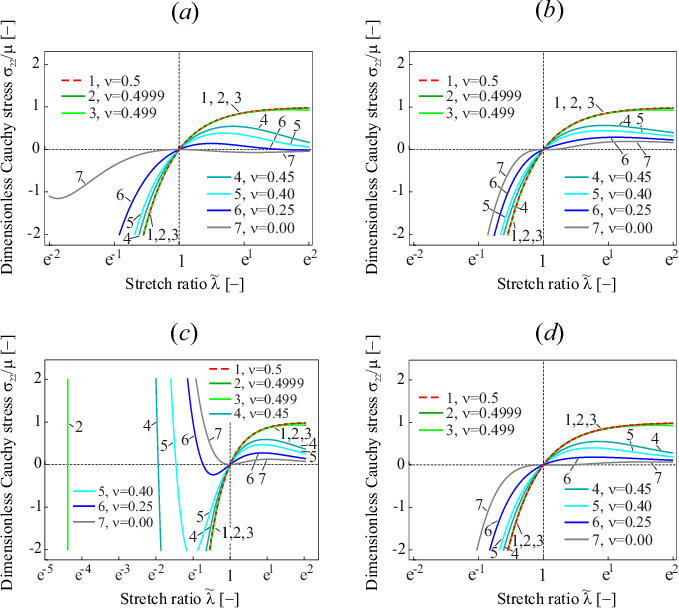}
\end{center}
\caption{Plots of $\sigma_{22}$ versus $\tilde{\lambda}$ in the ULP problem for vol-iso models \#1 (\emph{a}), \#4 (\emph{b}), \#7 (\emph{c}), and \#8 (\emph{d}).}
\label{f29}
\end{figure}
\begin{figure}
\begin{center}
\includegraphics{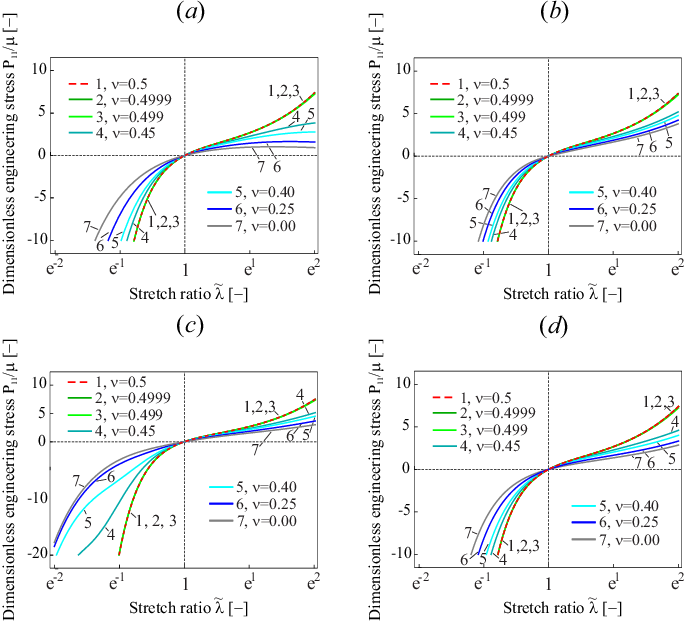}
\end{center}
\caption{Plots of $P_{11}$ versus $\tilde{\lambda}$ in the ULP problem for vol-iso models \#1 (\emph{a}), \#4 (\emph{b}), \#7 (\emph{c}), and \#8 (\emph{d}).}
\label{f30}
\end{figure}
\begin{figure}
\begin{center}
\includegraphics{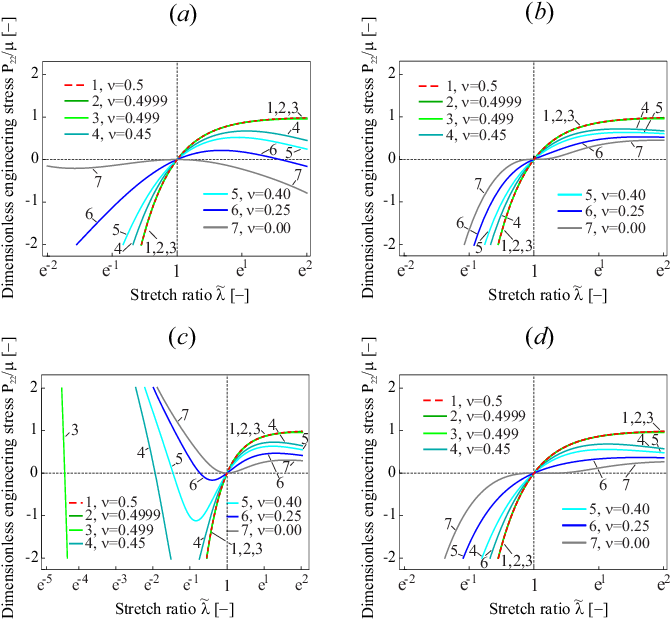}
\end{center}
\caption{Plots of $P_{22}$ versus $\tilde{\lambda}$ in the ULP problem for vol-iso models \#1 (\emph{a}), \#4 (\emph{b}), \#7 (\emph{c}), and \#8 (\emph{d}).}
\label{f31}
\end{figure}
Note that the plots of $\lambda_T(\tilde{\lambda})$ in Fig. \ref{f27},\emph{c}, the plots of $\sigma_{11}(\tilde{\lambda})$ in Fig. \ref{f28},\emph{c}, and the plots of $P_{11}(\tilde{\lambda})$ in Fig. \ref{f30},\emph{c} for vol-iso material model \#7 agree with the plots in Figs. 8, 14,\emph{c}, and 9,\emph{a} in \cite{KossaMeccanica2023} for the same material model. We observe the non-monotonic physically inadmissible dependencies of Cauchy stresses on stretches in Fig. \ref{f29},\emph{c} for material model \#7.

The limiting values of $\lambda_T$, $\sigma_{11}$, $\sigma_{22}$, $P_{11}$, and $P_{22}$ in extreme states are presented in Table~\ref{t6}. These limiting values are in qualitative agreement with the limiting values for the same material models presented in Table~\ref{t4}. Note that the limiting values of these quantities for model \#7 coincide with those for the same model presented in Fig.~11 in \cite{KossaMeccanica2023}.

The conclusion following from the solutions of the ULP problem using vol-iso models is similar to the conclusion drawn from the analysis of solutions of the ELP problem at the end of Section \ref{sec:6-3-3}.

\section{Concluding remarks}
\label{sec:7}

In this section we summarize and comment on all the research presented in this book. In particular, we do a comparative analysis of the performance of mixed and vol-iso material models in Section \ref{sec:7-1} and general conclusions in Section \ref{sec:7-2}.

\subsection{Comparative analysis of the performance of mixed and vol-iso material models}
\label{sec:7-1}

The analysis performed in this book leads to the following conclusions.

\begin{enumerate}
  \item From the application point of view, both mixed and vol-iso models predict close values of stress tensors and principal stretches for slightly compressible materials. However, for the models to be able to predict the dilatation $J-1$, they should be somewhat modified (see, e.g., \cite{FongTSR1975,OgdenJMPS1976}), which is beyond the scope of the present study.
  \item From the point of view of the consistency of physical intuition, in limiting states (i.e., for extreme values of stretches, very large or very small), models \#3,4,8 of both types (mixed and vol-iso) have physically reasonable responses for kinematic and static quantities. For models \#3,4 we use volumetric functions from the Hartmann--Neff family with parameter $q=2,\,5$ and for model \#8 we use new volumetric function. However, for model \#1 with the volumetric function with $q=0$ (i.e., with the function $h^{(0)}=(\ln J)^2/2$), only the mixed model has physically reasonable responses.
  \item Both mixed and vol-iso models \#1--6,8 satisfy the Hill inequality, but model \#7 do not satisfy this inequality for some values of the volume ratio $J$. In addition, all mixed and vol-iso models do not satisfy the CSP.
  \item The constitutive relations for mixed models are more simpler than the constitutive relations for vol-iso models. This is especially true for rate formulations for these models (cf., expression \eqref{115} for mixed models and expression \eqref{122} for vol-iso ones). In this book, we used rate formulations of constitutive relations to answer the question of whether the material models considered here satisfy the Hill and corotational stability postulates. However, rate formulations of constitutive relations are also required to implement any material models in FE systems. Our next goal is to obtain explicit expressions for tangent stiffness tensors for the material models under consideration.
\end{enumerate}

The rate formulations of constitutive relations \eqref{115} for mixed models can be rewritten as
\begin{equation}\label{7-1}
  \frac{\text{D}^{\text{ZJ}}}{\text{D}t}[\boldsymbol{\tau}]= \lambda\, \chi(J)J\,\text{tr}\,\mathbf{d}\, \mathbf{I} + 2(\mu -\lambda\, \ln J)\mathbf{d} + \mathbf{d}\cdot \boldsymbol{\tau} + \boldsymbol{\tau}\cdot \mathbf{d}.
\end{equation}
Similarly, the rate formulations of constitutive relations for vol-iso models \eqref{122} can be rewritten as \index{Rate constitutive relations!for vol-iso models}
\begin{align}\label{7-2}
  \frac{\text{D}^{\text{ZJ}}}{\text{D}t}[\boldsymbol{\tau}]= [K\, \chi(J)J &+ \frac{2}{9}\mu\, \text{tr}\,\mathbf{c}\,J^{-2/3}]\,\text{tr}\,\mathbf{d}\, \mathbf{I} +
[\frac{2}{3}\mu\, \text{tr}\,\mathbf{c}\,J^{-2/3} - 2 K\, J h^{\prime}(J)]\mathbf{d} \\
 & - \frac{2}{3}\mu\, J^{-2/3}[\text{tr}\,\mathbf{d}\, \mathbf{c} + (\mathbf{c}:\mathbf{d})\mathbf{I}] + \mathbf{d}\cdot \boldsymbol{\tau} + \boldsymbol{\tau}\cdot \mathbf{d}. \notag
\end{align}

Alternative forms of the rate constitutive relations \eqref{7-1} and \eqref{7-2} can be obtained using the upper Oldroyd stress rate $\frac{\text{D}^{\overline{\text{Old}}}}{\text{D}t}[\boldsymbol{\tau}]$ instead of the Zaremba--Jaumann stress rate $\frac{\text{D}^{\text{ZJ}}}{\text{D}t}[\boldsymbol{\tau}]$. Based on the relationship between objective tensor rates (see, e.g., Eq. $(25)_2$ in \cite{KorobeynikovAAM2020} and Eq. (47) in \cite{FedericoMMS2025}))
\begin{equation}\label{7-3}
  \frac{\text{D}^{\overline{\text{Old}}}}{\text{D}t}[\boldsymbol{\tau}]=\frac{\text{D}^{\text{ZJ}}}{\text{D}t}[\boldsymbol{\tau}] -  \mathbf{d}\cdot \boldsymbol{\tau} - \boldsymbol{\tau}\cdot \mathbf{d},
\end{equation}
an alternative representation of the rate constitutive relations for mixed models can be obtained from \eqref{7-1}:
\begin{equation}\label{7-4}
 \frac{\text{D}^{\overline{\text{Old}}}}{\text{D}t}[\boldsymbol{\tau}]= \lambda\, \chi(J)J\,\text{tr}\,\mathbf{d}\, \mathbf{I} + 2(\mu -\lambda \ln J)\mathbf{d},
\end{equation}
and that for vol-iso models can be obtained from \eqref{7-2}:
\begin{align}\label{7-5}
  \frac{\text{D}^{\overline{\text{Old}}}}{\text{D}t}[\boldsymbol{\tau}]= [K\, \chi(J)J + \frac{2}{9}\mu\, \text{tr}\,\mathbf{c}\, J^{-2/3}]\,\text{tr}\,\mathbf{d}\, \mathbf{I} & + [\frac{2}{3}\mu\, \text{tr}\,\mathbf{c}\,J^{-2/3} - 2 K\, J h^{\prime}(J)]\,\mathbf{d}\\
& - \frac{2}{3}\mu\, J^{-2/3}[\text{tr}\,\mathbf{d}\, \mathbf{c} + (\mathbf{c}:\mathbf{d})\mathbf{I}]. \notag
\end{align}
In particular, for mixed model \#1 (the \emph{Simo--Pister} \cite{SimoCMAME1984} \emph{hyperelastic model}), the equality $\chi(J)J=1$ holds and expression \eqref{7-4} reduces to the expression
\begin{equation*}
  \frac{\text{D}^{\overline{\text{Old}}}}{\text{D}t}[\boldsymbol{\tau}]= 2(\mu -\lambda \ln J)\mathbf{d} + \lambda \text{tr}\,\mathbf{d}\, \mathbf{I},
\end{equation*}
which represents the rate constitutive relations for the one-parameter (with parameter $n=2$) family of Hooke-like isotropic hyper-/hypo-elastic material models (see Eq. (60) in  \cite{KorobeynikovAAM2023} with the identification $\boldsymbol{\tau}^{\nabla\,(2)}=\frac{\text{D}^{\overline{\text{Old}}}}{\text{D}t}[\boldsymbol{\tau}]$).

Typically, Eulerian formulations of constitutive relations are implemented in FE systems by employing the updated Lagrangian approach (see, e.g., \cite{Bathe1996}) and using the fourth-order tangent stiffness tensors in two alternative forms of constitutive relations (see, e.g., \cite{Korobeynikov2000,KorobeynikovAAM2020,KorobeynikovAAM2023})\footnote{In particular, the commercial Abaqus system uses rate constitutive relations of the form $\eqref{7-7}_1$ \cite{JiJAM2013,NguyenZAMP2016}, and the commercial MSC.Marc system uses rate constitutive relations of the form $\eqref{7-7}_2$ (cf., \cite{MarcA2015}). Due to \eqref{36}, the positive definiteness of $\mathbb{C}^{\text{BH}}$ is equivalent to the Hill stability condition (see Chapter \ref{sec:5}), so Abaqus uses the positive definiteness of $\mathbb{C}^{\text{BH}}$ as material stability condition.}
\begin{equation}\label{7-7}
    \frac{\text{D}^{\text{BH}}}{\text{D}t}[\boldsymbol{\sigma}] = \mathbb{C}^{\text{BH}}:\mathbf{d},\quad\quad \frac{\text{D}^{\text{Tr}}}{\text{D}t}[\boldsymbol{\sigma}] = \mathbb{C}^{\text{Tr}}:\mathbf{d}.
\end{equation}
Based on relations \eqref{36} between the stress rates, the tangent stiffness tensor for mixed models can be obtained from \eqref{7-4}:
\begin{equation}\label{7-8}
 \mathbb{C}^{\text{Tr}}_{\text{mixed}}= \frac{2}{J}(\mu-\lambda\ln\,J) \mathbf{I}\!\overset{\text{sym}}{\otimes}\!\mathbf{I} + \lambda\, \chi(J)\mathbf{I}\otimes\mathbf{I},
\end{equation}
and that for vol-iso models can be obtained from \eqref{7-5}:
\begin{align}\label{7-9}
  \mathbb{C}^{\text{Tr}}_{\text{vol-iso}}= [K\, \chi(J)J + \frac{2}{9}\mu\, \text{tr}\,\mathbf{c}J^{-5/3}]\mathbf{I}\otimes\mathbf{I} & +
[\frac{2}{3}\mu\, \text{tr}\,\mathbf{c}\,J^{-5/3} - 2 K\, h^{\prime}(J)]\, \mathbf{I}\!\overset{\text{sym}}{\otimes}\!\mathbf{I}\\
 & - \frac{2}{3}\mu\, J^{-5/3}(\mathbf{c}\otimes\mathbf{I} + \mathbf{I}\otimes\mathbf{c}). \notag
\end{align}
Similar expressions for the tensors $ \mathbb{C}^{\text{BH}}$ can be derived from expressions \eqref{7-8} and \eqref{7-9} using expressions \eqref{36} and \eqref{7-3}
\begin{equation*}
   \mathbb{C}^{\text{BH}}=\mathbb{C}^{\text{Tr}} + \mathbf{I}\!\overset{\text{sym}}{\otimes}\!\boldsymbol{\sigma}+\boldsymbol{\sigma}\!\overset{\text{sym}}{\otimes}\!\mathbf{I}.
\end{equation*}
Note that all the fourth-order tensors considered here have full symmetry.

Note the simplicity of expression \eqref{7-8} compared to expression \eqref{7-9}. Note also that for mixed model \#1,
\begin{equation*}
  \mathbb{C}^{\text{Tr}}_{\text{mixed}}=\frac{1}{J}\,\widehat{\mathbb{C}}_{\sharp}^{(2)},
\end{equation*}
where the tangent stiffness tensor $\widehat{\mathbb{C}}_{\sharp}^{(2)}$ is defined in Eq. $(68)_2$ in \cite{KorobeynikovAAM2023}.

\subsection{General conclusions}
\label{sec:7-2}

The main purpose of this study was to answer the question: are there compelling reasons to use the more complex neo-Hookean vol-iso model of compressible isotropic hyperelastic material rather than the simpler mixed model of the same material when simulating deformations of both rubber-like (slightly compressible) and foam-like (highly compressible) materials?

To answer this question, we performed a systematic study of the performance of both compressible neo-Hookean models, mixed and vol-iso, using seven well-known volumetric functions and a new one. The results of the study lead to the following conclusions.

First, in applications for simulating deformations of slightly compressible materials, both kinematic and static quantities obtained using the two types of models are close to each other as well as to the same quantities obtained using the incompressible neo-Hookean model.

Second, both types of models satisfy Hill's postulate in the same range of the volume ratio, and both types of models do not satisfy the corotational stability postulate (CSP) for all volumetric functions used in this study. Further, both model variants satisfy the polyconvexity requirement, provided that the volumetric term $h(J)$ is convex in $J$.

Third, compared to vol-iso models, mixed models have physically reasonable responses in extreme states for a wider set of volumetric functions. In particular, the popular volumetric function of the form $(\ln J)^2/2$ (not convex in J!) leads to physically reasonable responses for kinematic and static variables in extreme states when using mixed models, but does not lead to the same responses when using vol-iso models.

However, it should be noted that vol-iso models can be used in the range of Poisson's ratio $-1<\nu <0.5$, whereas mixed models can be used only in the range $0\leq\nu< 0.5$.

To recapitulate, both the present work and previous study \cite{EhlersAM1998,KossaMeccanica2023,PenceMMS2015} have shown that when using volumetric functions from the Hartmann--Neff family with parameter $q\geq 2$ (the preferred value is $q=5$ \cite{HartmannIJSS2003}), mixed and vol-iso models show similar performance in applications and have physically reasonable responses in extreme states, which is convenient  for theoretical studies. However, mixed models allow the use of a wider set of volumetric functions with physically reasonable responses in extreme states, compared to vol-iso models. A second important advantage of mixed models over vol-iso models are more simple expressions for stresses and tangent stiffness tensors.

Note that the neo-Hookean material model has a narrow range application for simulating deformations of elastomers. First, its application is limited to simulating engineering strains of only about 10\%, and, second, this material model does not take into account second-order effects, in particular the Pointing effect in simple shear or torsion of circular cross-section rods (see, e.g., \cite{KorobeynikovMTDM2024}). The purpose of this study was to assess the feasibility of implementing different approaches to the simulation of deformations of compressible and slightly compressible elastomers in commercial FE systems. In particular, in MSC.Marc FE simulations of deformations using the generalized Mooney--Rivlin or Ogden models, compressible vol-iso material models are used for slightly compressible (rubber-like) elastomers and compressible mixed material models for compressible (foam-like) materials. Since the neo-Hooken model is a special case of both the Mooney--Rivlin and Ogden models, the present study shows that it is inappropriate to use models with different types of accounting for compressibility in FE systems. In fact, the simpler (in the mathematical sense) mixed Mooney--Rivlin or Ogden material models can equally successfully simulate deformations of both sufficiently compressible and slightly compressible materials.

\newpage
\footnotesize
\bibliographystyle{plain}
\bibliography{Korobeynikov_Springer_2025}

\end{document}